\documentclass[draft]{amsart}

\usepackage[Symbol]{upgreek}

\usepackage{amssymb}
\usepackage{epic}
\usepackage{mathrsfs}
\usepackage{accents}
\usepackage{stmaryrd}
\usepackage{enumitem}

\usepackage[titletoc]{appendix}
\usepackage{titlecaps}
\usepackage{pgfplots}
\pgfplotsset{width=6.75 cm,compat=1.13}
\usepgfplotslibrary{fillbetween}
\usetikzlibrary{patterns, decorations.pathreplacing}

\input{xy}
\xyoption{matrix}
\xyoption{arrow}

\numberwithin{equation}{section}

\newcommand{\bbold}{\mathbb}
\newcommand{\cal}{\mathcal}

\def \tr{\operatorname{tr}}
\def\H {\operatorname{H}}

\def\R { {\bbold R} }
\def\Q { {\bbold Q} }
\def \S{\mathcal{S}}
\def\Z { {\bbold Z} }
\def\C { {\bbold C} }
\def\N { {\bbold N} }

\def \dim{\operatorname{dim}}

\def \Hy{{\mathcal H}}
\def \alg{\operatorname{alg}}
\def \cl{\operatorname{cl}}
\def \oN{\operatorname{N}}
\def \t{\operatorname{t}}

\def \ex{\operatorname{e}}

\renewcommand\epsilon{\varepsilon}

\def \<{\langle}
\def \>{\rangle}

\def \tilde {\widetilde}

\def \hat {\widehat}

\def \((  {(\!(}
\def \)) {)\!)}

\def \k {{{\boldsymbol{k}}}}

\DeclareMathSymbol{\precequ}{\mathrel}{symbols}{"16}
\DeclareMathSymbol{\succequ}{\mathrel}{symbols}{"17}

\newtheorem{theorem}{Theorem}[section]
\newtheorem{lemma}[theorem]{Lemma}
\newtheorem{prop}[theorem]{Proposition}
\newtheorem{cor}[theorem]{Corollary}

\theoremstyle{definition}

\theoremstyle{remark}
\newtheorem*{example}{Example}

\let\oldi\i
\let\oldj\j
\renewcommand\i{\relax\ifmmode{\boldsymbol{i}}\else\oldi\fi}
\renewcommand\j{\relax\ifmmode{\boldsymbol{j}}\else\oldj\fi}

\renewcommand\leq{\leqslant}
\renewcommand\geq{\geqslant}
\renewcommand\preceq{\preccurlyeq}

\renewcommand\le{\leq}
\renewcommand\ge{\geq}

\DeclareMathAlphabet{\mathbf}{OML}{cmm}{b}{it}

\DeclareFontFamily{U}{fsy}{}
\DeclareFontShape{U}{fsy}{m}{n}{<->s*[.9]psyr}{}
\DeclareSymbolFont{der@m}{U}{fsy}{m}{n}
\DeclareMathSymbol{\der}{\mathord}{der@m}{182}

\DeclareSymbolFont{der@m}{U}{fsy}{m}{n}
\DeclareMathSymbol{\derdelta}{\mathord}{der@m}{100}

\DeclareSymbolFont{imag@m}{OT1}{cmr}{m}{ui}
\DeclareMathSymbol{\imag}{\mathord}{imag@m}{105}

\DeclareFontFamily{OMS}{smallo}{}
\DeclareFontShape{OMS}{smallo}{m}{n}{<->s*[.65]cmsy10}{}
\DeclareSymbolFont{smallo@m}{OMS}{smallo}{m}{n}
\DeclareMathSymbol{\smallo}{\mathord}{smallo@m}{79}

\DeclareFontFamily{OMS}{largerdot}{}
\DeclareFontShape{OMS}{largerdot}{m}{n}{<->s*[.8]cmsy10}{}
\DeclareSymbolFont{largerdot@m}{OMS}{largerdot}{m}{n}
\DeclareMathSymbol{\largerdot}{\mathord}{largerdot@m}{15}

\DeclareMathSymbol{\llambda}{\mathord}{der@m}{108}
\DeclareMathSymbol{\rrho}{\mathord}{der@m}{114}

\newcommand{\equationqed}[1]{\[\pushQED{\qed}#1 \qedhere\popQED\]\let\qed\relax}
\newcommand{\alignqed}[1]{\begin{align*}\pushQED{\qed} #1 \qedhere\popQED\end{align*}\let\qed\relax}

\makeatletter
\newcommand{\dminus}{\mathbin{\text{\@dminus}}}

\newcommand{\@dminus}{%
  \ooalign{\hidewidth\raise1ex\hbox{\bf.}\hidewidth\cr$\m@th-$\cr}%
}
\makeatother

\def \Lr{L^{\operatorname{r}}}
\def \Lf{L^{\operatorname{f}}}
\def \cM{\mathcal{M}}
\def \cN{\mathcal{N}}
\def \Var{\operatorname{Var}}

\title{On The Pila-Wilkie Theorem}

\author[Bhardwaj]{Neer Bhardwaj}
\address{Department of Mathematics\\
University of Illinois at Urbana-Cham\-paign\\
Urbana, IL 61801\\
U.S.A.}
\email{nbhard4@illinois.edu}

\author[van den Dries]{Lou van den Dries\\
\\
{\em T\lowercase{o the memory of }B\lowercase{as }E\lowercase{dixhoven} (1962-2022)}
}
\address{Department of Mathematics\\
University of Illinois at Urbana-Cham\-paign\\
Urbana, IL 61801\\
U.S.A.}
\email{vddries@illinois.edu}

\begin{document}

\begin{abstract} This expository paper gives an account of the Pila-Wilkie counting theorem and some of its extensions and generalizations. We use semialgebraic cell
decomposition to simplify part of the original proof. We also include complete treatments of a result due to Pila and Bombieri and of the o-minimal 
Yomdin-Gromov theorem that are used in this proof. For the latter we follow Binyamini and Novikov.  
\end{abstract}

\date{February 2022}

\maketitle

\section{Introduction and some notations}

\noindent
In these notes we prove the Pila-Wilkie theorem following the original paper~\cite{PW}, but exploiting cell decomposition more thoroughly to simplify the deduction from its main ingredients. By including proofs of these ingredients and adding an Appendix on o-minimality we make this account 
self-contained, within reason.  The technically most demanding ingredient is an o-minimal Yomdin-Gromov theorem, but thanks to
Binyamini and Novikov~\cite{BN} this can now be done more directly than in~\cite{PW}.  

We also obtain two generalizations due to Pila~\cite{P2}, one where instead of rational points we count points with coordinates in a $\Q$-linear subspace of $\R$ with a finite bound on its dimension, and one where instead we count points with coordinates that are algebraic of at most a given degree over $\Q$. The general approach is as in \cite{P2}, but the technical details seem to us a bit simpler.  

We thank Chieu-Minh Tran and James Freitag for helpful input. 

\bigskip\noindent
Throughout, $d,e, k,l, m, n\in \mathbb{N}=\{0,1,2,\dots\}$, and $\epsilon, c, K\in \mathbb{R}^{>}:=\{t\in \R:\ t>0\}$. For $\alpha = (\alpha_{1},\ \ldots,\ \alpha_{m})\in \mathbb{N}^{m}$ we set $|\alpha|:=\alpha_{1}+\cdots+\alpha_{m}$, and 
given a field $\k$ (often $\k=\R$) we set
$x^\alpha:=x_1^{\alpha_1}\cdots x_m^{\alpha_m}$ for 
the usual coordinate functions $x_1,\dots, x_m$ on $\k^m$, and likewise
 $a^{\alpha} :=a_{1}^{\alpha_{1}}\cdots a_{m}^{\alpha_{m}}$ for any point
$a=(a_1,\dots, a_m)\in \k^m$. Let $U\subseteq \R^m$ be open.
For a function $f: U\to \R$ 
of class $C^k$ and $\alpha\in \N^m$,  $|\alpha|\le k$, 
$$f^{(\alpha)}\ :=\ \frac{\partial^{|\alpha|}}{\partial x^\alpha}f $$ denotes the corresponding partial derivative of order $\alpha$.  We extend the above to $C^k$-maps $f=(f_1,\dots, f_n): U \to \R^n$ by
$$f^{(\alpha)}\ :=\ (f_1^{(\alpha)},\dots,f_n^{(\alpha)}): U \to \R^n$$ 
for $\alpha$ as before. This includes the case $m=0$, where $\R^0$ has just one point and any map $f: U \to \R^n$ is of class $C^k$ for all $k$, with
$f^{(\alpha)}=f$ for the unique $\alpha\in \N^0$. For $a_1,\dots, a_n\in \R^{\ge}:=\{t\in \R:\ t\ge 0\}$ the number $\max\{a_1,\dots, a_n\}\in \R^{\ge}$
equals $0$ by convention if $n=0$. 
For $a=(a_1,\dots, a_n)\in \R^n$ we set $|a|:= \max\{|a_1|,\dots, |a_n|\}$
in $\R^{\ge}$; this conflicts with our notation $|\alpha|$ for
$\alpha\in \N^n$, but in practice no confusion will arise. 
We also use these notational conventions when instead of $\R$ we have any o-minimal field with $U$ and $f$ definable in it. We include an appendix on the basic facts concerning o-minimality and definability for those not familiar with these topics. Here we just mention that an o-minimal field is by 
definition an ordered field equipped with an o-minimal structure on it; this ordered field is then necessarily real closed: adjoining $\imag$ with
$\imag^2=-1$ makes it algebraically closed. An {\em o-minimal expansion of $\R$\/} is an o-minimal field whose underlying ordered field is $\R$. 

\subsection*{The Pila-Wilkie theorem and two ingredients of the proof}

First some notation needed to state the theorem. 
We define the {\em multiplicative height function\/} $\H:\mathbb{Q}\rightarrow \mathbb{R}$ by $\H(\frac{a}{b})\ :=\  \max(|a|,\ |b|)\in \N^{\ge 1}$ for coprime $a, b\in \Z$, $b\ne 0$. Thus $\H(0)=\H(1)=\H(-1)=1$, and for $q\in \Q$ we have $\H(q)\ge 2$ if $q\notin \{0, 1, -1\}$, $\H(q)=\H(-q)$, and $\H(q^{-1})=\H(q)$ for $q\ne 0$. 
For  $a= (a_1, . . . , a_{n}) \in \mathbb{Q}^{n}$, 
 $$\H(a)\ :=\ \max\{\H(a_1),\dots, \H(a_n)\}\in \N.$$ Let
 $X \subseteq \mathbb{R}^{n}$. We set $X(\Q)=X\cap \Q^n$.  Throughout $T$ ranges over
 real numbers $\ge 1$, and we set 
 $X(\mathbb{Q}, T):=\{a\in X(\Q):\ \H(a)\le T\}$ be the (finite)
 set of rational points of $X$ of height $\leq T$, and set $\oN(X,T):=\#X(\mathbb{Q}, T)\in \N$.
 The {\em algebraic part\/} of $X$, denoted by $X^{\alg}$, is the union of the connected infinite semialgebraic subsets of $X$. So for $n\ge 1$, the interior of $X$ (in $\R^n$) is part of $X^{\alg}$.  

\medskip\noindent
{\bf Example}. The set $X:=\{(a,b,c)\in \R^3:\ 1<a,b<2,\ c=a^b\}$ is definable in the o-minimal field $\R_{\exp}$. (See the subsection ``O-Minimal Structures'' in part A of the Appendix for $\R_{\exp}$.) For rational $q\in (1,2)$,
we have a semialgebraic curve $$X_q\ :=\ \{(a,q,c):\ c=a^q\}\ 
\subseteq\ X,$$
and $X^{\alg}$ is the union of those $X_q$ (proved at the end of part A of the appendix).  

\medskip\noindent
We also set
$$X^{\tr}\ :=\ X\setminus X^{\alg}\qquad(\text{\em the transcendental part of  $X$}).$$ We can now state the Pila-Wilkie theorem, also called the {\em Counting Theorem}:

\begin{theorem}\label{PW} Let $X\subseteq \R^n$ with $n\ge 1$ be definable in some o-minimal expansion of $\R$. Then for all $\epsilon$ there is a $c$ such that for all $T$,
$$\oN(X^{\tr},T)\ \leq\ cT^{\epsilon}.$$
\end{theorem} 

\noindent
Roughly speaking, it says there are few rational points on the transcendental part of a set definable in an o-minimal expansion of $\R$: the number of such points grows slower than any power $T^\epsilon$ with $T$ bounding their height. To apply the counting theorem one needs to describe $X^{\alg}$ in some useful way. This typically involves {\em Ax-Schanuel} type transcendence results. 

Note that $X^{\tr}(\Q)=\emptyset$ in the example above, so the theorem is trivial for this $X$. We shall include a refinement, Theorem~\ref{PWr}, which is nontrivial for this $X$. 

 The proof of Theorem~\ref{PW} depends on two intermediate results. 
The first of these has nothing to do with o-minimality. To state it we
define for $k,n \ge 1$ and $X\subseteq \mathbb{R}^{n}$ a {\em strong
$k$-parametrization of $X$\/} to be a $C^k$-map $f$ : $(0,1)^{m} \rightarrow \mathbb{R}^{n}$, $m<n$, with image $X$, such that $|f^{(\alpha)}(a)|\le 1$ for all $\alpha \in \mathbb{N}^{m}$ with $|\alpha|\leq k$ and all $a\in(0,1)^{m}$. We also
define a {\em hypersurface in $\R^n$ of degree $\le e$\/} to be the zeroset in $\R^n$ of a nonzero polynomial in $x=(x_1,\dots, x_n)$ over $\R$ of (total) degree $\le e$. The first of these intermediate results is essentially due to Pila and Bombieri, cf. \cite{BP, P1}. 

\begin{theorem}\label{Pp} Let $n\ge 1$ be given. Then for any $e\ge 1$ there are $k = k(n,\ e)\ge 1$, $\epsilon =\epsilon(n,\ e)$,  and $c=c(n,\ e)$, such that if $X\subseteq \R^n$ has a strong $k$-parametrization, then for all $T$ at most $cT^\epsilon$ many hypersurfaces in $\R^n$ of degree $\le e$ are enough
to cover $X(\mathbb{Q},T)$, with $\epsilon(n,\ e)\rightarrow 0$ as $e\rightarrow\infty$.
\end{theorem}

\medskip\noindent  
We prove this in Section~\ref{P1}. In Section~\ref{pct} we obtain Theorem~\ref{PW} from Theorem~\ref{Pp} by induction on $n$, using a strong parametrization result. Yomdin~\cite{Yo} and Gromov~\cite{Gr} proved such a strong parametrization
uniformly for the members of any semialgebraic family of
subsets of $[-1,1]^n$. We need this for any definable ``o-minimal'' family. To make this precise, let $E\subseteq \R^m$ and
$X\subseteq E\times \R^n$. For $s\in E$,
$$X(s)\ :=\ \{a\in \R^n:\ (s,a)\in X\}  \quad\text{(a {\em section\/} of $X$)}$$
We consider $E,X$ as describing the family
$\big(X(s)\big)_{s\in E}$ of sections $X(s)\subseteq \R^n$; the sets $X(s)$ are the members of the family. If $E$ and $X$ are definable in the o-minimal expansion $\tilde{\R}$ of $\R$, then its members are definable in $\tilde{\R}$.   

\begin{theorem}\label{YG+} Let $\tilde{\R}$ be an o-minimal expansion of $\R$ and $E\subseteq \R^m$ and $X\subseteq E\times \R^n$ with $n\ge 1$
definable in $\tilde{\R}$
such that every section $X(s)$ is a subset of $[-1,1]^n$ with empty interior. Then there is for every $k\ge 1$ an $M\in \N$ such that every section $X(s)$ is the union of at most $M$ subsets, each having a strong $k$-parametrization.  
\end{theorem} 

\noindent
This is enough for use in the next section, but the proof in Sections~4--7 gives more precise results. In the appendix we also expose some elementary model theory used in that proof. In Section~\ref{elab} we treat
extensions of the Counting Theorem from \cite{P2}.  

The $M$ in Theorem~\ref{YG+} is a source of ineffectivity, and in this connection we call attention to~\cite{BNAnn1, BNAnn2} where among other things Binyamini and Novikov establish better (logarithmic) bounds for certain o-minimal expansions of $\R$. 




\section{Proof of the Counting Theorem from the two ingredients}\label{pct}

\medskip\noindent
Throughout this section we assume $n\ge 1$. We begin by stating some elementary facts about $X^{\alg}$ and $X^{\tr}$ for $X\subseteq \R^n$. The first is obvious:

\begin{lemma}\label{ef1} If $X=X_1\cup\cdots \cup X_m$, then $X^{\alg}\supseteq
X_1^{\alg}\cup \cdots \cup X_m^{\alg}$, and thus
$$X^{\tr}\ \subseteq\ X_1^{\tr} \cup\cdots \cup X_m^{\tr}.$$
\end{lemma}

\noindent
Note also that if $X$ is open in $\R^n$, then $X^{\tr}=\emptyset$.

\begin{lemma}\label{ef2} Suppose $S\subseteq \R^n$ is semialgebraic, $f: S\to \R^m$ is semialgebraic and injective,~and $f$ maps the set $X\subseteq S$ homeomorphically onto $Y=f(X)\subseteq \R^m$. Then $f(X^{\alg})=Y^{\alg}$ and thus $f(X^{\tr})=Y^{\tr}$. $($We allow $m=0$ for later inductions.$)$ 
\end{lemma}
\begin{proof} It is clear that $f(X^{\alg})\subseteq Y^{\alg}$. Also, for any connected infinite semialgebraic set $C\subseteq Y$, the set
$f^{-1}(C)\subseteq S$ is semialgebraic (since $C$ and $f$ are), contained in $X$ (since $f$ is injective), hence connected and infinite, and thus
$f^{-1}(C)\subseteq X^{\alg}$. This shows $f^{-1}(Y^{\alg})\subseteq X^{\alg}$, and thus $f(X^{\alg})=Y^{\alg}$.
\end{proof}

\noindent
In order to apply Theorem~\ref{YG+} we need to reduce to the case of subsets of $[-1,1]^n$. This is done as follows. For
$X\subseteq \R^n$ and $I\subseteq \{1,\dots,n\}$, set 
$$X_I\ :=\ \{a\in X:\ |a_i|>1 \text{ for all }i\in I,\ |a_i|\le 1 \text{ for all }i\notin I\}$$
and define the semialgebraic map $f_I: \R_I^n \to \R^n$ by
$f_I(a)=b$ where $b_i:=a_i^{-1}$ for $i\in I$ and $b_i:= a_i$ for $i\notin I$.
Thus $f_I$ maps $\R_I^n$ homeomorphically onto its image, a subset of $[-1,1]^n$. If $I=\emptyset$, then $f_I$ is the inclusion map $\R_I^n=[-1,1]^n \to \R^n$. Note that for $a\in \Q^n$ we have $f_I(a)\in \Q^n$ and
$\H(a)=\H\big(f_I(a)\big)$. Moreover, $X$ is the disjoint union of
the sets $X_I$, and for $Y_I=f_I(X_I)$ we have $Y_I\subseteq [-1,1]^n$, $Y_I^{\tr}=f_I(X_I^{\tr})$ by Lemma~\ref{ef2}, so $\oN(Y_I^{\tr},T)=\oN(X_I^{\tr},T)$ for all $T$.

\bigskip\noindent
The sketch below actually proves the Counting Theorem, modulo a uniformity assumption that
arises at the end of the sketch. This motivates a stronger ``definable family'' version of the theorem, which we then prove as in the sketch. 
 
In the rest of this section we fix an o-minimal expansion $\tilde{\R}$ of $\R$, and {\em definable} is with respect to $\tilde{\R}$. We exploit facts about semialgebraic cells $C\subseteq \R^n$ and the corresponding homeomorphisms $p_C: C \to p(C)$; see the subsections ``Cells'' and ``Cell Decomposition'' in part A of the Appendix.

\subsection*{Sketch of the proof of Theorem~\ref{PW} from Theorems~\ref{Pp} and ~\ref{YG+}} Let $X\subseteq \mathbb{R}^{n}$ be definable. We need to show that there are `few' rational points on $X$ outside $X^{\alg}$. We proceed by induction on $n$. By Lemma~\ref{ef1} and the remark following it we can remove the interior of $X$ in $\R^n$ from $X$ and arrange that $X$ has empty interior. 
As indicated just before this sketch we also arrange 
$X\subseteq [-1,1]^n$.  

Let $\epsilon$ be given, and take
$e\ge 1$ so large that $\epsilon(n,e)\le \epsilon/2$ in Theorem~\ref{Pp}, and take $k= k(n,e)$. Theorem~\ref{YG+} gives $M\in \N$ such that $X$ is a union of at most $M$ subsets, each admitting a strong $k$-parametrization. 
Then Theorem~\ref{Pp} gives $X(\Q,T)\subseteq \bigcup_{i=1}^M\bigcup_{j=1}^J H_{ij}$, where the $H_{ij}$ are hypersurfaces in $\R^n$ of degree $\le e$, and $J\in \N$, $J\le cT^{\epsilon/2}$, $c=c(n,e)$ as in that theorem. If $a\in X^{\tr}(\Q,T)$
and $a\in H_{ij}$, then clearly $a\in (X\cap H_{ij})^{\tr}$. Thus 
$$X^{\tr}(\Q,T)\ \subseteq\ \bigcup_{i=1}^M\bigcup_{j=1}^J (X\cap H_{ij})^{\tr}(\Q,T).$$
Let $H$ be any hypersurface in $\R^n$ of degree $\le e$.
We aim for an upper bound on 
$\oN\big((X\cap H)^{\tr},T\big)$ of the form $c_1T^{\epsilon/2}$ with $c_1\in \R^{>}$ independent of $H$ and $T$. (If we achieve this, then applying this to the hypersurfaces $H_{ij}$
we obtain $$\oN( X^{\tr},T)\  \le\ MJ c_1T^{\epsilon/2}\ \le\ M\cdot cT^{\epsilon/2}\cdot c_1T^{\epsilon/2}\  =\  Mcc_1\cdot T^{\epsilon},$$
and we are done.) Take semialgebraic cells $C_1,\dots, C_L$ in $\R^n$, $L\in \N$, such that $$H\ =\ C_1\cup \cdots \cup C_L.$$ Suppose $C=C_l$ is one of those cells. Then we have a semialgebraic homeomorphism $p=p_C: C \to p(C)=p(C_l)$
onto an open cell $p(C_l)$ in $\R^{n_l}$ with $n_l < n$, and so $p$ maps
$X\cap C_l$ homeomorphically onto its image $Y_l\subseteq p(C_l)\subseteq \R^{n_l}$. Now  $p$ is given by omitting $n-n_{l}$ of the  coordinates, so for $a\in C_{l}(\Q)$ we have $p(a)\in \Q^{n_l}$ and $\H\big(p(a)\big)\le\H(a)$. The 
hypersurfaces of degree $\le e$ in $\R^n$ belong to a single semialgebraic
family, so by Proposition~\ref{fdc} we can (and do) take here $L\le L(e,n)$, with $L(e,n)\in \N^{\ge 1}$ depending only on $e,n$.   
By Lemma~\ref{ef1}, $$(X\cap H)^{\tr}\ \subseteq\ (X\cap C_1)^{\tr}\cup\cdots \cup(X\cap C_L)^{\tr}.$$ 
Since the $n_l<n$ we can (and do) assume inductively that for all $T$, 
$$\oN(Y_l^{\tr}, T)\ \le\  B_l T^{\epsilon/2}, \quad l=1,\dots, L$$ with $B_l\in \R^{>}$ independent of $T$. Hence for all $T$,
$$\oN\big((X\cap C_l)^{\tr}), T\big)\ \le\ B_l T^{\epsilon/2}, \quad l=1,\dots, L$$
by Lemma~\ref{ef2} applied to the maps $p=p_{C_l}$, and thus $$\oN\big((X\cap H)^{\tr}, T\big)\ \le\  (B_1+\dots + B_L)T^{\epsilon/2}.$$ 
{\em Assume we can take $B_1,\dots, B_L\le B$ with $B\in \R^{>}$ depending only on $X,\epsilon$, not on $H, Y_1,\dots, Y_L$}. Then $c_1\ :=L(e,n)B$ is a positive real number as aimed for. 

\bigskip\noindent
The above sketch is a proof, modulo the assumption at the end. 
The hypersurfaces $H$ in the sketch belong fortunately to a single semialgebraic family, a fact we already used, and so the sets $Y_l$ as $H$ varies can be taken to belong to a single definable family.  
The inductive hypothesis should accordingly include this uniformity, and
so the full proof should be carried out not just for one set $X$, but uniformly for all sets from a definable family, with constants depending only on the family. 
This is why we need Theorem~\ref{YG+} not just for a single definable $X\subseteq [-1,1]^n$ but for all members of a definable family of such sets.  (As to the $M$ introduced at the beginning of the sketch, Theorem~\ref{YG+} also
provides an $M$ that works for all members of the family.) Below we carry out the details.


\medskip\noindent
The next lemma is a routine consequence of Theorem~\ref{celldec} and Proposition~\ref{fdc}. The $\i$-cells in this lemma and the projection maps $p_{\i}: \R^n\to \R^d$ in the proof of Theorem~\ref{piwi} are defined in the subsection ``Cells'' from part A of the Appendix. 

\begin{lemma}\label{ef3} Let $e\ge 1$ and set $D:=\binom{e+n}{n}$, the dimension of the $\R$-linear space of polynomials over $\R$ in $n$ variables and of degree $\le e$.  Then there are $L\in \N^{\ge 1}$ and semialgebraic sets $\Hy, \mathcal{C}_1,\dots, \mathcal{C}_L\subseteq F\times \R^n,\ F:=\R^D\setminus \{0\}$, such that
$$\{\Hy(t):\ t\in F\}=\ \text{set of hypersurfaces in $\R^n$ of degree $\le e$},$$
$\Hy(t)=\mathcal{C}_1(t)\cup\cdots \cup \mathcal{C}_L(t)$ for all $t\in F$, and for each $l\in \{1,\dots,L\}$ there is an $\i=(i_1,\dots, i_n)$ in $\{0,1\}^n$, $\i\ne (1,\dots,1)$, with the property that every $\mathcal{C}_{l}(t)$ with $t\in F$ is a semialgebraic $\i$-cell in $\R^n$ or empty. 
\end{lemma} 

\subsection*{Two family versions of the counting theorem} In this subsection we assume that $E\subseteq \R^m$ and $X\subseteq E\times \R^n$ are definable.

\begin{theorem}\label{piwi} Let any $\epsilon$ be given. Then there is a constant $c=c(X,\epsilon)$ such that for all $s\in E$ and all $T$ we have $\oN(X(s)^{\tr},T)\leq cT^{\epsilon}$.
\end{theorem}
\begin{proof} We proceed by induction on $n$. As in the sketch we reduce to the case where $X(s)$ is for every $s\in E$ a subset of $[-1,1]^n$ with empty interior. Take $e\ge 1$ so large that $\epsilon(n,e)\leq \epsilon/2$ in Theorem~\ref{Pp}, and set $k=k(n,e)$. So for every $Z\subseteq \R^n$ with a strong $k$-parameterization we can cover $Z(\Q,T)$ with at most $cT^{\epsilon/2}$ hypersurfaces of degree $\leq e$ where  $c=c(n,e)$ is as in Theorem~\ref{Pp}. Theorem~\ref{YG+} gives $M\in \N$ such that each
section $X(s)$ is a union of at most $M$ subsets, each admitting a strong $k$-parametrization. Let $s\in E$, and let $H$ be a hypersurface of degree $\leq e$. As in the sketch we see that by our choice of $k, e$ it is enough to show: $$\oN\big((X(s)\cap H)^{\tr},T\big)\ \leq\ c_1T^{\epsilon/2}, \text{ for all }T,$$ where $c_1\in \R^{>}$  depends only on $X,\epsilon$, not on $s, H, T$. Below we provide such $c_1$. 

With the present values of $e$ and $n$, set $D:=\binom{e+n}{n}$, 
 $F:= \R^D\setminus \{0\}$, and let $\Hy, \mathcal{C}_1,\dots, \mathcal{C}_L\subseteq F\times \R^n$ be as in Lemma \ref{ef3}. 
For $l=1,\dots, L$, take $\i^{l}=(i^l_{1},\dots, i^l_{ n})$ in $\{0,1\}^n$, not equal to $(1,\dots,1)$, such that for all $t\in F$ the subset $\mathcal{C}_l(t)$ of $\R^n$ is a semialgebraic $\i^{l}$-cell or empty, so 
 $$p_{\i^{l}}: \R^n \to \R^{n_{l}},\quad n_{l}:=i^l_{1} + \cdots + i^l_{n} < n,$$ maps $\mathcal{C}_l(t)$ homeomorphically onto its image. Then
 we have for $l=1,\dots, L$ a definable set $Y_{l}\subseteq (E\times F)\times \R^{n_{l}}$ such that for all $(s,t)\in E\times F$,
$$Y_{l}(s,t)\ =\ p_{\i^{l}}\big(X(s)\cap \mathcal{C}_{l}(t)\big).$$
Since all $n_{l} < n$ we can assume inductively that for all $(s,t)\in E\times F$ and all $T$,
$$\oN\big(Y_l(s,t)^{\tr}, T\big)\ \le\  B_l T^{\epsilon/2}, \quad l=1,\dots, L$$ with $B_l=B_l(Y_l,\epsilon)\in \R^{>}$ independent of $s,t,T$.
Since $H=\Hy(t)$ for some $t\in F$,  
$$\oN\big((X(s)\cap H)^{\tr}, T\big)\ \le\  (B_1+\dots + B_L)T^{\epsilon/2},$$
as in the sketch.   
Thus $c_1\ := B_1+\cdots + B_L$ is as promised.
\end{proof}

\noindent
Next a variant of Theorem~\ref{piwi} where we remove from the sets $X(s)$ only a
definable part $V(s)$ of $X(s)^{\alg}$ instead of all of it. The example preceding the statement of Theorem~\ref{PW} shows that this variant is strictly stronger than Theorem~\ref{piwi}. 

\begin{theorem}\label{PWr}
Let any $\epsilon$ be given. Then there is a definable set $V=V(X,\epsilon)\subseteq X$ and a constant $c=c(X,\epsilon)$ such that for all $s\in E$ and all $T$,
$$ V(s)\subseteq X(s)^{\alg} \quad \text{and} \quad \oN\big(X(s)\setminus V(s),T\big)\leq cT^{\epsilon}.$$
\end{theorem}
\begin{proof}
By induction on $n$. We follow closely the proof of Theorem \ref{piwi}. Let $V_0\subseteq X$ be given by $V_0(s)=\text{interior of }X(s)$ in $\R^n$ for $s\in E$.
This definable set $V_0$ will be part of a $V$ as required. 
Replacing $X$ by $X\setminus V_0$ we arrange that $X(s)$ has empty interior for all $s\in E$. We arrange in addition that $X(s)\subseteq [-1,1]^n$ for all $s\in E$. Now take $e$ and  $k=k(n,e)$ as in the proof of Theorem~\ref{piwi}. It will be enough
to find a definable $V\subseteq X$ and a constant $c_1\in \R^{>}$ such that for all $s\in E$, every hypersurface $H$ of degree $\le e$ in $\R^n$, and all $T$ we have
$$V(s)\ \subseteq\ X(s)^{\alg}, \qquad \oN\big((X(s)\cap H)\setminus V(s),T\big)\ \leq\ c_1T^{\epsilon/2}.$$ 
We take the semialgebraic sets $\Hy, \mathcal{C}_1,\dots, \mathcal{C}_L\subseteq F\times \R^n$ and the definable sets $Y_{l}\subseteq E\times F\times \R^{n_{l}}$ for  $l=1,\ldots, L$ as in the proof of Theorem \ref{piwi}. For such $l$ we have $n_{l}<n$, so we can assume inductively that we have a definable set $W_{l} \subseteq Y_{l}$ and a number $B_{l}=B_{l}(Y_{l},\epsilon)\in \R^{>}$ such that for all $s\in E$, $t\in F$, and $T$ we have
$$  W_{l}(s,t)\ \subseteq\ Y_{l}(s,t)^{\alg} \quad \text{and} \quad N\big(Y_{l}(s,t)\setminus W_{l}(s,t),T\big)\ \leq\ B_{l}T^{\epsilon/2}.$$
It is now easy to check that the definable set $V\subseteq X$ such that for all $s\in E$, $$V(s)\ =\ \bigcup_{l=1}^L\bigcup_{t\in F}\mathcal{C}_{l}(t)\cap p^{-1}_{\i^{l}}\big(W_{l}(s,t)\big) $$ has the desired property.  
\end{proof}
 
\noindent
In the next sections we establish the results used in the proofs above, namely Theorems~\ref{Pp} and~\ref{YG+}. In Section~\ref{elab} we strengthen and extend Theorem~\ref{PWr} in several ways without changing the basic inductive set-up of its proof. 

\section {Proof of Theorem~\ref{Pp}}\label{P1}

\noindent
We begin with introducing a key determinant. 
Let $\k$ be a field and set
$$   D(n,\ e)\ :=\ \binom{e+n}{n}\ =\ \#\{\alpha\in \N^n:\ |\alpha|\ \leq\ e\}\in \N^{\ge 1},$$
the dimension of the $\k$-linear space of $n$-variable polynomials over
$\k$ of (total) degree at most $e$. Thus $D(n,\ 0)=1$, $D(n,e)=\frac{e^n}{n!}\big(1+o(1)\big)$
as $e\to \infty$, and if $n\ge 1$, then $D(n,e)$ is strictly increasing as a function of $e$. 

For now we fix $n$ and $e$, set $D:= D(n,e)$ and let $\alpha$ range over $\N^n$. By a {\em hypersurface in $\k^n$ of degree $\le e$} we mean the set of zeros in $\k^n$ of a nonzero $n$-variable polynomial of degree $\leq e$ with coefficients in $\k$.  

\begin{lemma}\label{hs} A set $S\subseteq \k^n$ is contained in some hypersurface in $\k^n$ of degree at most $e$ if and only if $\det(a_{i}^{\alpha})_{|\alpha|\leq e, i=1,\dots,D}=0$ for all $a_1,\dots, a_D\in S$. 
\end{lemma}
\begin{proof} Let $f = \displaystyle \sum_{|\alpha|\leq e}c_{\alpha}x^{\alpha}$ be a nonzero polynomial in $x=(x_1,\dots, x_n)$ of degree at most $e$ with coefficients $c_{\alpha}\in \k$ such that $f=0$ on $S$. Then for any points $a_1,\dots, a_D\in S$ we have  
$f(a_1)=\cdots= f(a_D)=0$, that is,
$$\sum_{|\alpha|\leq e}c_{\alpha}\big(a_1^{\alpha}, \dots, a_D^\alpha\big)\ =\ 0\ \text{ in }\ \k^D,$$ so the $D$ vectors 
$(a_1^\alpha,\dots, a_D^\alpha)$ ($|\alpha|\le e$) in the $\k$-linear space $\k^D$ are linearly dependent, which gives the desired conclusion about the determinant. 

Conversely, suppose $\det(a_{i}^{\alpha})_{|\alpha|\leq e, i=1,\dots,D}=0$
for all $a_1,\dots, a_D\in S$. Then for $A:=\{\alpha: |\alpha|\le e\}$, the $\k$-linear subspace of $\k^A$ spanned by the vectors $(a^\alpha)_{|\alpha|\le e}$
with $a\in S$ has dimension $< D$. Take $a_1,\dots, a_M\in S$ such that
$$(a_1^\alpha)_{|\alpha|\le e},\ \dots,\ (a_M^\alpha)_{|\alpha|\le e}$$ is a basis of this subspace. Then $M<D$, so we have $c_{\alpha}\in \k$ for $|\alpha|\le e$, with $c_{\alpha}\ne 0$ for some $\alpha$ and
$\sum_{|\alpha|\le e} c_{\alpha}a^\alpha=0$ for $a=a_1,\dots, a_M$, and thus for
all $a\in S$.
\end{proof}

\noindent
Next we introduce some numbers related to $D=D(n,e)$: 
$$E(n,\ e)\ :=\ \binom{e+n-1}{n-1}\ =\ \#\{\alpha:\ |\alpha|=e\},$$ the dimension of the $\k$-linear space of homogeneous $n$-variable polynomials of degree $e$ over $\k$. (Here $\binom{-1}{-1}:=1$ and $\binom{k}{-1}:=0$.) So
$D(n,e)=\sum_{i=0}^e E(n,i)$. Next, we set
$V(n,e) := \sum_{i=0}^e iE(n,i)$. Now for $i\ge 1$,
\begin{align*} iE(n,i)\ &=\ i\binom{i+n-1}{n-1}\ =\ n\binom{i+n-1}{n}\ =\ nE(n+1,i-1),\ \text{ so}\\ 
V(n,e)\ &=\ n\sum_{i=1}^{e}E(n+1,i-1)\ =\ nD(n+1,e-1) \text{ for }e\ge 1, \quad V(n,0)=0,
\end{align*}
and thus for fixed $n$ we have $V(n,e)=\frac{ne^{n+1}}{(n+1)!}\big(1+o(1)\big)$ as $e\to \infty$.

\medskip\noindent 
Let $e, m, n \ge 1$ below and define $b=b(m,\ n,\ e)\in \N$ by requiring
$$
D(m,\ b)\ \leq\  D(n,\ e)\ <\ D(m,\ b+1)\ .
$$
Next, we set for $b=b(m,n,e)$:
\begin{align*} 
B(m,\ n,\ e)\ :&=\ \sum_{i=0}^{b}iE(m,i)+(b+1)\cdot\big(D(n,\ e)-\sum_{i=0}^{b}E(m,i)\big)\\
&=\ V(m,b) + (b+1)\cdot\big(D(n,e)-D(m,b)\big)\in \N^{\ge 1},\\ 
\epsilon(m,\ n,\ e)\ :&=\ \frac{mneD(n,e)}{B(m,n,e)}.
\end{align*}

\begin{lemma} With fixed $m,n\ge 1$ and $e\to \infty$, we have: 
\begin{enumerate}
\item $b(m,\ n,\ e)\ =\ \left(\frac{m!e^n}{n!}\right)^{1/m}\big(1+o(1)\big)$;
\item $B(m,\ n,\ e)\ =\  \frac{m}{(m+1)!}\big(\frac{m!}{n!})^{(m+1)/m} e^{n(m+1)/m}\big(1+o(1))$;
\item if $m<n$, then $\epsilon(m,\ n,\ e)\to 0$.
\end{enumerate}
\end{lemma}
\begin{proof} As to (1), for $e\to \infty$ we have $b=b(m,n,e)\to \infty$,  
so 
$$D(m,b)\ =\ \frac{b^m}{m!}\big(1+o(1)\big)\ \le\ \frac{e^n}{n!}\big(1+o(1)\big)\ \le\ \frac{(b+1)^m}{m!}\big(1+o(1)\big),$$
but the last term here is also $\frac{b^m}{m!}\big(1+o(1)\big)$, like the first term, and this easily yields the asymptotics claimed for $b$. For (2), substituting the result of (1) in the asymptotics for $D(m,b)$ as $b\to \infty$
leads to $(b+1)\cdot\big(D(n,e)-D(m,b)\big)=o\big(e^{n(m+1)/m}\big)$, and
then in the asymptotics for $V(m,b)$ yields the asymptotics claimed for $B(m,n,e)$. Now (3) is an easy consequence of (2). 
\end{proof} 

\noindent
In the proof of Proposition~\ref{bound} below we need a reasonable bound
on the absolute value of the determinant of a certain $(D\times D)$-matrix 
of the form $\big(a_{i}^{\alpha}\big)_{|\alpha|\leq e, i=1,\dots,D}$. We achieve this by expressing the matrix as a sum of simpler matrices. In this connection we need a useful expression for the determinant of a sum of matrices. 

Turning to this, let $N\in \N$ and consider an $(N\times N)$-matrix $a=(a_{\mu\nu})_{1\le \mu,\nu\le N}$ over a field $\k$. The determinant of an $(N\times N)$-matrix over $\k$ is an alternating multilinear function of its columns. The columns of $a$ are $a_1,\dots, a_N\in \k^N$ where $a_\nu=(a_{1\nu},\dots, a_{N\nu})^{\t}\in \k^N$ is the $\nu$th column of $a$. Thus $$a\ =\ (a_1,\dots, a_N)\in \k^N\times \cdots \times \k^N\ \text{(with $N$ factors $\k^N$)}.$$ Next, let 
$a=a^1+ \cdots + a^r$ with $r\in \N$ and $a^1,\dots, a^r$ also $(N\times N)$-matrices over $\k$, with $a^j$ having $\nu^{\text{th}}$ column $a^j_{\nu}$. Then
\begin{align*} \det a\ &=\ \det\big(a_1,\dots, a_N\big)\ =\ \det\big(\sum_{j=1}^r a^j_1,\dots, \sum_{j=1}^r a^j_N\big)\\
&=\ \sum_{\j} \det\big(a^{j_1}_1,\dots, a^{j_N}_N\big) 
\end{align*}  
where $\j=(j_1,\dots, j_N)$ ranges here and below over elements of $\{1,\dots, r\}^N$. Let $\j$ be given. 
If for some $j$ in $\{1,\dots, r\}$ the number of $\nu\in \{1,\dots,N\}$ with
$j_\nu=j$ is more than $\operatorname{rank} a^j$, then the column vectors $a_1^{j_1},\dots, a_N^{j_N}$ are $\k$-linearly dependent, so
$\det\big(a^{j_1}_1,\dots, a^{j_N}_N\big)=0$. Thus if $J\subseteq \{1,\dots, r\}^N$ contains all $\j$ such that 
$$\#\{\nu\in \{1,\dots,N\}:\ j_{\nu}=j\}\ \le\ \operatorname{rank} a^j, \text{ for }j=1,\dots,r,$$
then 
$$(*)\qquad \qquad\qquad \det a\ =\ \sum_{\j\in J} \det\big(a^{j_1}_1,\dots, a^{j_N}_N\big)\ =\ \sum_{\j\in J} \det\big(a^{j_{\nu}}_{\mu\nu}\big)_{1\le\mu,\nu\le N} \qquad\qquad.$$
 
\noindent
We shall also use the following observation:

\begin{lemma}\label{obs} Let $A$ be a set and $V$ a finite-dimensional subspace of the $\k$-linear
space $\k^A$. Then for any $N\in \N$, functions $f_1,\dots, f_N\in V$, and points $a_1,\dots, a_N$ in~$A$,  the rank of the $(N\times N)$-matrix $\big(f_\mu(a_\nu)\big)_{1\le \mu,\nu\le N}$ over $\k$  is $\le \dim V$.
\end{lemma}
\begin{proof} The map $f\mapsto \big(f(a_1),\dots, f(a_N)\big): V \to \k^N$ is $\k$-linear, so the image of this map is a subspace of
the $\k$-linear space $\k^N$ of dimension $\le \dim V$.
\end{proof} 

\noindent
Recall our norm $|(t_1,\dots, t_m)|:=\max\{|t_1|,\dots, |t_m|\}$ on $\R^m$ for $m\ge 1$.

\begin{prop}\label{bound} Let $e,m,n\ge 1$, $m < n$, and $k:= b(m,n,e)+1$. Then there is a constant $K=K(m,n,e)$ with the following property: if $f: (0,1)^m \to \R^n$ is a strong $k$-parametrization, $0 < r \le 1$, and
$a_0,\dots, a_D\in (0,1)^m$ with $D=D(n,e)$ are such that $|a_i-a_0|\le r$ for $i=1,\dots, D$, then 
$$|\det\big(f(a_i)^\alpha\big)_{|\alpha|\le e,\ i=1,\dots,D}|\ <\  Kr^{B(m,n,e)}.$$
\end{prop}
\begin{proof} Let $f=(f_1,\dots, f_n)$ with $f_j: (0,1)^m\to \R$.   Taylor expansion of $f_j$ of order $b:=b(m,n,e)$ around $a_0$ with explicit remainder gives for $a\in (0,1)^m$ (and $\alpha,\beta$ ranging over $\N^m$, $k=b+1$):
\begin{align*}f_j(a)\ &=\  \sum_{|\alpha|\le b} \frac{f_j^{(\alpha)}}{\alpha!}(a-a_0)^\alpha  +\ \sum_{|\beta|=k}R_{\beta,j}(a)(a-a_0)^\beta,\\
\text{where }\ R_{\beta,j}\ &=\ \frac{|\beta|}{\beta!}\int_0^1(1-t)^b f_j^{(\beta)}\big(a_0+t(a-a_0)\big)dt\  \text{ for }\ |\beta|=k.
\end{align*}
Thus for $i=1,\dots, D$ and $j=1,\dots,n$:
$$
f_{j}(a_{i})\ =\  P_j(a_{i}-a_0)+ R_{ij}(a_i-a_0)
$$
where $P_j\in \R[x_1,\dots,x_m]$ has degree $\le b$, the remainder is given by a homogeneous polynomial
 $R_{ij}\in \R[x_1,\dots, x_m]$ of degree $k=b+1$,  and all coefficients
of $P_j$ and $R_{ij}$ are bounded in absolute value by $1$. Let $|\alpha|\le e$. Then for $i=1,\dots, D$ we have
$$\prod_{j=1}^n(P_j+R_{ij})^{\alpha_j}\ =\ P_{\alpha} + R_{i\alpha}$$
with $P_{\alpha}\in \R[x_1,\dots, x_m]$ of degree $\le b$, the remainder 
$R_{i\alpha}\in \R[x_1,\dots, x_m]$ has only monomials of degree $>b$, and
every coefficient of $P_{\alpha}$ and $R_{i\alpha}$ is bounded in absolute value by $D(m,k)^{|\alpha|}$, the latter because $\prod_{j=1}^n (P_j+ R_{ij})^{\alpha_j}$ is a product of $|\alpha|$ factors of the form $\sum c_{\beta}x^\beta$, with the summation over the $\beta\in \N^m$ with $|\beta|\le k$, and real coefficients $c_{\beta}$ with $|c_{\beta}|\le 1$. Hence for $i=1,\dots,D$,
$$ f(a_i)^\alpha\ =\ \prod_{j=1}^n f_j(a_i)^{\alpha_j}\ =\ P_{\alpha}(a_i-a_0) + R_{i\alpha}(a_i-a_0).$$ 
We have $D(m,k)^{|\alpha|}\le D(m,k)^e\le c$ for a positive constant $c=c(m,n,e)$ depending only on $m,n,e$. Next,  
$P_{\alpha} = \sum_{j=0}^b P_{\alpha}^j$ where $P_{\alpha}^j\in \R[x_1,\dots, x_m]$ is homogeneous of degree $j$.
In the matrix algebra $\R^{D\times D}$ this yields the sum decomposition
\begin{align*} \big(f(a_i)^\alpha\big)_{\alpha,i}\ &=\ \sum_{j=0}^b \big(P_{\alpha}^j(a_i-a_0)\big)_{\alpha,i} + \big(R_{i\alpha}(a_i-a_0)\big)_{\alpha,i}\\
&=\ \sum_{j=0}^k \big(P_{i\alpha}^j(a_i-a_0)\big)_{\alpha,i}
\end{align*}
where $P_{i\alpha}^j:= P_{\alpha}^j$ for $j=0,\dots,b$ and $P_{i\alpha}^k:=R_{i\alpha}$.
For $j=0,\dots, b$ the rank of the matrix 
$\big(P_{i\alpha}^j(a_i-a_0)\big)_{\alpha,i}=\big(P_{\alpha}^j(a_i-a_0)\big)_{\alpha,i}$ is at most $E(m,j)$ by Lemma~\ref{obs}, so 
expression $(*)$ for the determinant of such a sum gives
$$\det\big(f(a_i)^\alpha\big)_{\alpha, i}\ =\ \sum_{\j\in J} \det\big(P^{j_i}_{i\alpha }(a_i-a_0)\big)_{\alpha, i}  $$
where $J$ is the set
of all $\j=(j_1,\dots, j_D)\in \{0,\dots,b+1\}^D$ such that 
$$\#\{\nu\in \{1,\dots,D\}:\ j_{\nu}=j\}\ \le\ E(m,j),\quad \text{ for } j=0,\dots,b .$$
Then for $\j \in J$ we have  $|\det\big(P^{j_i}_{i\alpha }(a_i-a_0)\big)_{\alpha, i}|\le D!c^D r^{|\j|}$. It remains to show that for $\j\in J$ we have $|\j| \geq B(m,n,e)$, because then $$|\det\big(f(a_i)^\alpha\big)_{|\alpha|\le e,\ i=1,\dots,D}|\ \le\ \#J\cdot D!c^D r^{B(m,n,e)},$$ 
which gives a constant $K=K(m,n,e)$ as claimed.

Fix any $\j\in J$, and let $d_j\in \N$ for
$j=0,\dots,b$ be such that $$\#\{\nu\in \{1,\dots,D\}:\ j_{\nu}=j\}\ =\  E(m,j)-d_j,$$ and set $N:=\#\{\nu\in \{1,\dots,D\}:\ j_{\nu}=b+1\}$.
Then 
$$D\ =\ D(n,e)\ =\ \sum_{j=0}^b( E(m,j)-d_j) + N\ =\ D(m,b)-\sum_{j=0}^b d_j + N,$$
so $N\ =\  D(n,e)-D(m,b)+d$ with $d:=\sum_{j=0}^b d_j$.
Hence 
\begin{align*} |\j|\ &=\ \sum_{\nu=1}^{D}j_{\nu}\ =\ 
\sum_{j=0}^b j\big(E(m,j)-d_j\big) + (b+1)N\\
&=\ V(m,b) -\sum_{j=0}^b jd_j + (b+1)\big(D(n,e)-D(m,b) +d\big)\\
&=\ V(m,b)+ (b+1)\big(D(n,e)-D(m,b)\big) + \sum_{j=0}^b (b+1-j)d_j\\ 
&=\ B(m,n,e)+\sum_{j=0}^b (b+1-j)d_j\ \ge\ B(m,n,e). \qedhere
\end{align*} 
\end{proof}

\noindent
Next an observation that allows us to exploit (as Liouville did) the simple fact that if $r\in \Z$ and $|r|<1$, then $r=0$. 

\begin{lemma}\label{denominator} Let points $b_1,\dots, b_D\in \Q^n$ with $D=D(n,e)$ be given  such that $\H(b_1),\dots, \H(b_D)\le t$, where $t\ge 1$. Then
$$\det\big(b_i^{\alpha}\big)_{|\alpha|\le e,i}\ \in\ \frac{\Z}{s}\quad \text{ with }s\in \N^{\ge 1},\ s\le t^{neD}.$$ 
\end{lemma}
\begin{proof} For $i=1,\dots, D$ we have $b_i=(b_{i1},\dots, b_{in})$ with
$b_{ij}= c_{ij}/s_{ij}$, $c_{ij}, s_{ij}\in \Z$, $1\le s_{ij}\le t$, so 
$$b_i^\alpha=\prod_{j=1}^n c_{ij}^{\alpha_j}/\prod_{j=1}^n s_{ij}^{\alpha_j}\ \in\ \frac{\Z}{s_{i\alpha}}, \quad s_{i\alpha}\ :=\ \prod_{j=1}^n s_{ij}^{\alpha_j}.$$
Let $\{\alpha:\ |\alpha|\le e\}=\{\alpha_1,\dots, \alpha_D\}$. Then $\det\big(b_i^{\alpha}\big)_{|\alpha|\le e,i}$ is a sum of terms of the form $\pm \prod_{i=1}^D b_i^{\alpha_{\sigma(i)}}$ where 
$\sigma$ is a permutation of $\{1,\dots, D\}$. Now the term $\pm \prod_{i=1}^D b_i^{\alpha_{\sigma(i)}}$ corresponding to
$\sigma$ lies in $\frac{\Z}{s_{\sigma}}$ with 
$$s_{\sigma}\ :=\ \prod_{i=1}^D s_{i\alpha_{\sigma(i)}}\ =\ \prod_{i=1}^D \prod_{j=1}^n s_{ij}^{\alpha_{\sigma(i)j}}$$
and clearly $s:= \prod_{i=1}^D\prod_{j=1}^ns_{ij}^e$ is a common integer multiple of the integers $s_{\sigma}$ with $1\le s\le t^{neD}$, so
$s$ has the desired property.
\end{proof}   

\noindent
The following is Theorem~\ref{Pp} with more explicit values of $k$ and $\epsilon$. 

\begin{theorem}\label{covthm} Let $e,m,n\ge 1$, $m<n$; set $k:=b(m,n,e)+1$,
$\epsilon:= \epsilon(m,n,e)$. Let $X\subseteq \R^n$ have a strong $k$-parametrization $f: (0,1)^m\to \R^n$. Then for all $T$ at most $cT^\epsilon$ hypersurfaces in $\R^n$ of degree $\le e$
are enough to cover $X(\Q,T)$, where $c=c(m,n,e)$ depends only on $m,n,e$.
\end{theorem}
\begin{proof} Let $K=K(m,n,e)$ be as in Proposition~\ref{bound}, and let
$T$ be given. With $D=D(n,e)$, let $a_1,\dots, a_D\in (0,1)^m$ be such that
$f(a_1),\dots, f(a_D)\in X(\Q,T)$. Then Lemma~\ref{denominator} gives $s\in \N^{\ge 1}$ with $s\le T^{neD}$ (so $T^{-neD}\le 1/s$) such that $$\det\big(f(a_i)^\alpha\big)_{|\alpha|\le e,i=1,\dots D}\ \in\ \frac{\Z}{s}.$$ Assume also that $0<r\le 1$ and $a_0\in (0,1)^m$ are such that
$|a_i-a_0|\le r$ for $i=1,\dots,D$. Can we guarantee that 
$\det\big(f(a_i)^\alpha\big)_{|\alpha|\le e,i=1,\dots D}=0$
if $r$ is small enough? 
Proposition~\ref{bound} gives $$|\det\big(f(a_i)^\alpha\big)_{|\alpha|\le e,\ i=1,\dots,D}|\ <\  Kr^B,\quad B\ =\ B(m,n,e). $$
So the answer to the question is {\em yes}: it is enough that $Kr^B\le T^{-neD}$, that is, $r\le \big(K^{-1}T^{-neD}\big)^{1/B}$.
Next, considering closed balls of radius $r$ with respect to the norm $|\cdot|$, centered at a point in $(0,1)^m$, how many are enough to cover
$(0,1)^m$? For $m=1$, the interval $(0,1)$ is covered by
$e$ segments $[a-r, a+r]$ with $0<a<1$, for any natural number $e$ with $2re\ge 1$, and there is clearly such an $e$ with $e\le r^{-1}$. Hence at most $r^{-m}$ closed balls of radius $r$
centered at points in $(0,1)^m$ are enough to cover
$(0,1)^m$. Taking $r=\big(K^{-1}T^{-neD}\big)^{1/B}$ it follows from Lemma~\ref{hs} that 
at most $r^{-m}=K^{m/B} T^{mneD/B}=K^{m/B}T^\epsilon$ hypersurfaces in $\R^n$ of degree $\le e$ are enough to cover the set $X(\Q,T)$. So the theorem holds with $c=K^{m/B}$. 
\end{proof}

\section{Parametrization}\label{par1}

\noindent
Throughout $R$ is an o-minimal field. We refer to part A of the Appendix for the basic facts about o-minimal fields. We rely in particular on the later subsections in this part A concerning differentiation and smoothness. 
As usual we identify $\Q$ with the prime subfield of $R$. We drop the subscript $R$ in
expressions like $(0,1)_R$ ($=\{t\in R: 0 < t <1\}$) and $[a,b]_R$ ($=\{t\in R:\ a\le t \le b\}$) for $a < b$ in $R$.

\medskip\noindent
Let $X\subseteq R^m$ be definable. Call $X$ {\em strongly bounded\/} if $X\subseteq [-N,N]^m$ for some $N$ in $\N$. Call a definable map $f: X \to R^n$ strongly bounded if its graph $\Gamma(f)\subseteq R^{m+n}$ is strongly bounded; equivalently, $X\subseteq R^m$ and $f(X)\subseteq R^n$ are strongly bounded.

A {\em partial $k$-parametrization of $X$\/} is a
definable $C^k$-map $f:(0,1)^l\to R^m$ such that $l=\dim X$, the image of $f$ is contained in $X$, and
$f^{(\beta)}$ is strongly bounded for all $\beta\in \N^l$ with $|\beta|\le k$.  A {\em $k$-parametrization of $X$\/} is a finite set of partial $k$-parametrizations of $X$ whose images cover $X$; note that then $X$ is strongly bounded. As a trivial example, if $X$ is finite and strongly bounded, then $X$ has the $k$-parametrization $\{\phi_a:\ a\in X\}$, where
$\phi_a: (0,1)^0\to R^m$ takes the value $a$. 

The basic ideas for the proofs of the next two parametrization theorems stem from Yomdin~\cite{Yo} and Gromov~\cite{Gr} who
considered the semialgebraic case. For our purpose we need to work in an arbitrary o-minimal field.     

\begin{theorem}\label{pars} Any strongly bounded definable set $X\subseteq R^m$ has for every $k\ge 1$ a $k$-parametrization.
\end{theorem} 

\noindent
The inductive proof of this theorem also requires a version for definable maps. A {\em $k$-reparametrization\/} of a definable map 
$f: X \to R^n$ is a $k$-parametrization $\Phi$ of its domain $X$ such that for every
$\phi:(0,1)^l\to R^m$ in $\Phi$, $f\circ \phi$ is of class $C^k$ and
$(f\circ \phi)^{(\beta)}$ is strongly bounded for all $\beta\in \N^l$
with $|\beta|\le k$; note that then $\{f\circ \phi:\ \phi\in \Phi\}$ is a
$k$-parametrization of $f(X)$, provided
 $\dim X= \dim f(X)$.

\begin{theorem}\label{parm} Any strongly bounded definable map
$f: X\to R^n$, $X\subseteq R^m$ has for every $k\ge 1$ a $k$-reparametrization.
\end{theorem}

\noindent
Sections 5,6,7 contain the proof of Theorems~\ref{pars} and ~\ref{parm}. In Section 7 we assume $R$ is $\aleph_0$-saturated, and thus non-archimedean. This can always be arranged by passing to a suitable elementary extension of $R$ and noting that the statements of \ref{pars} and ~\ref{parm} pull back to the original $R$. (See part B of the Appendix for ``$\aleph_0$-saturated'' and ``elementary extension'' and the relevant facts about these notions. See in particular the last two subsections of that part B for a more detailed explanation of how these facts apply to proving Theorems~\ref{pars} and ~\ref{parm}.)

We often use the following, obtained by repeated use of the Chain Rule:

\begin{lemma}\label{compfact} Let $f: U \to R,\ g: V\to R$ be definable of class $C^k$, $k\ge 1$, with $U,V$ $($definable$)$ open subsets of $R$. Then 
$f\circ g: V\cap g^{-1}(U)\to R$ is of class $C^k$ with
$$(f\circ g)^{(k)}\ =\ \sum_{i=1}^k (f^{(i)}\circ g)\cdot p_{ik}\big(g^{(1)},\dots, g^{(k-i+1)}\big)$$
where the $p_{ik}\in \Z[x_1,\dots, x_{k-i+1}]$ have constant term $0$ and
$p_{kk}=x_1^k$. 
\end{lemma}  

\begin{lemma}\label{sbfact} With $U\subseteq R^l$, $V\subseteq R^m$, let $f: U \to R^m$,  $g: V \to R^n$ be definable of class $C^k$ such that $f(U)\subseteq V$ and $f^{(\alpha)}$ and $g^{(\beta)}$ are strongly bounded
for all $\alpha\in \N^l$ and $\beta\in \N^m$ with $|\alpha|\le k$ and
$|\beta|\le k$. Then the definable map $g\circ f: U \to R^n$ is of class $C^k$ with strongly bounded $(g\circ f)^{(\alpha)}$ for all $\alpha\in \N^l$ with $|\alpha|\le k$.
\end{lemma}

\section{Reparametrizing unary functions}\label{par2}

\noindent 
Much in this section is bookkeeping, but we begin with a key analytic fact:

\begin{lemma}\label{karep} Let $f : (0, 1)\to R$ be a definable $C^k$-function, $k\ge 2$,  with strongly bounded $f^{(j)}$ for $0 \le j \le k - 1$ and decreasing 
$|f^{(k)}|$. Define $g : (0, 1)\to R$ by
$g(t) = f(t^2)$. 
Then $g^{(j)}$ is strongly bounded for $0 \le j \le k$.
\end{lemma}  
\begin{proof} Let $t$ range over $(0,1)$.  Lemma~\ref{compfact} gives
$$g^{(j)}(t)\ =\ \sum_{i=0}^j \rho_{ij}(t).f^{(i)}(t^2),\qquad j=0,\dots,k$$ 
where each function $\rho_{ij}$ is given by a $1$-variable polynomial with integer coefficients, of degree $\le i$, and with $\rho_{jj}(t)=2^jt^j$.
All summands here are strongly bounded except possibly the one with $i = j = k$, which is $2^kt^k f^{(k)}(t^2)$. So it suffices that
$t^kf^{(k)}(t^2)$ is strongly bounded.  Let $c\in \Q^{>0}$ be a strong bound for $f^{(k-1)}$. We claim that then
$|f^{(k)}(t)|\le 4c/t$ for all $t$. Suppose towards a contradiction that
$t_0\in (0,1)$ is a counterexample, that is, 
$|f^{(k)}(t_0)| > 4c/t_0$. Then the Mean Value Theorem (Lemma~\ref{mvt1}) provides
a $\xi\in [t_0/2, t_0]$ such that 
$$f^{(k-1)}(t_0) - f^{(k-1)}(t_0/2)\ =\ f^{(k)}(\xi).(t_0 - t_0/2)\ =\ f^{(k)}(\xi)\cdot t_0/2.$$ Since $|f^{(k)}|$ is decreasing by assumption, $|f^{(k)}(\xi)|\ \ge\  |f^{(k)}(t_0)| > 4c/t_0$. Hence
$$2c\ \ge\ |f^{(k-1)}(t_0)-f^{(k-1)}(t_0/2)|\ >\ (4c/t_0)\cdot (t_0/2)\ =\ 2c. $$
This contradiction proves our claim.
Then for all $t$,
$$|t^kf^{(k)}(t^2)|\ \le\  t^k\cdot( 4c/t^2)\ =\ 4ct^{k-2}\ \le\ 4c$$
using $k\ge 2$ for the last inequality.  
\end{proof} 

\noindent
The lemma fails for $k=1$, with $t\mapsto t^{1/3}$ as a counterexample.

\begin{lemma}\label{onerep} Let $f : (0, 1) \to R$ be definable and strongly bounded. Then $f$ has a $1$-reparametrization $\Phi$ such that for every $\phi\in \Phi$, $\phi$ or $f\circ \phi$ is given by a $1$-variable polynomial with strongly bounded coefficients in $R$.
\end{lemma}
\begin{proof}  Take elements $a_0 = 0  < a_1 < \dots < a_n < a_{n+1} = 1$ in $R$ such that, for $i = 0,1,\dots,n$, $f$ is of class $C^1$ on $(a_i, a_{i+1})$, and either $|f'| \le  1$ on $(a_i,a_{i+1})$, or $|f'| > 1$ on $(a_i, a_{i+1})$. Let $i\in \{0,\dots,n\}$. 
If $|f'| \le  1$ on $(a_i,a_{i+1})$, define 
$$\phi_i\ :\ (0,1)\to R, \qquad \phi_i(t)\ :=\ a_i+ (a_{i+1}- a_i)t.$$
If $|f'|>1$ on $(a_i, a_{i+1})$, set 
$$b_i\ :=\ \lim_{t\downarrow a_i} f(t), \qquad b_{i+1}\ :=\ \lim_{t\uparrow a_{i+1}} f(t)$$ and as in this case $f$ is continuous and strictly monotone
on $(a_i, a_{i+1})$ we can define $\phi_i : (0,1) \to R$ by $\phi_i(t)= f^{-1}\big(b_i+(b_{i+1}-b_i)t \big)$, where $f^{-1}$ denotes the compositional inverse of the restriction of $f$ to $(a_i, a_{i+1})$; this compositional inverse
has domain $(b_i, b_{i+1})$ if $b_i< b_{i+1}$, and domain $(b_{i+1},b_i)$
if $b_i>b_{i+1}$. 
 
In either case, $\phi_i$ maps $(0,1)$ onto $(a_i , a_{i+1})$ and both $\phi_i$ and $f\circ \phi_i$ are of class $C^1$ with strongly bounded derivative. Moreover, $\phi_i$ or $f\circ \phi_i$ is given by a univariate polynomial of degree $1$
with strongly bounded coefficients in $R$. Thus $$\Phi\ : =\ \{\phi_0,\dots,\phi_n,\hat{a}_1,\dots, \hat{a}_n\}$$ is a $1$-reparametrization of $f$ as required, where $\hat{a}_i$ denotes the constant function on $(0, 1)$ with value $a_i$.
\end{proof}

\begin{lemma}\label{krep} Let $k \ge 1$ and suppose $f : (0, 1) \to R$ is definable and strongly bounded. Then $f$ has a $k$-reparametrization $\Phi$ such that for all $\phi\in \Phi$,  $\phi$ or $f\circ \phi$ is given by a $1$-variable polynomial with strongly bounded coefficients in $R$.
\end{lemma}
\begin{proof} By induction on $k$. The case $k=1$ is Lemma~\ref{onerep}.
Suppose $k\ge 2$ and $\Phi$ is a $(k-1)$-reparametrization of $f$ with the additional property. Let $\phi\in \Phi$. Then 
$\{\phi, f \circ \phi\} = \{g, h\}$ where $g$ is given by a univariate polynomial with strongly bounded coefficients in $R$. Thus 
$g$ is of class $C^{\infty}$, and $g^{(i)}$ is strongly bounded for all $i\in \N$, and $h$ is of class $C^{k-1}$ with strongly bounded $h^{(j)}$ for $j = 0,\dots,k-1$. In order to apply Lemma~\ref{karep} we use o-minimality: take elements $$a_0\ =\ 0\  <\ a_1\ <\ \dots\ <\ a_{n_\phi}\ <\ a_{n_{\phi}+1}\ =\ 1$$ in $R$ such that for $i = 0,\dots,n_{\phi}$, the function $h$ is of class $C^k$ on $(a_i,a_{i+1})$ and $|h^{(k)}|$ is monotone on $(a_i, a_{i+1})$.
Define $\theta_{\phi,i} : (0, 1) \to R$ as
$t\mapsto a_i+(a_{i+1} - a_i)t$, if $|h^{(k)}|$ is decreasing, and as
$t\mapsto a_{i+1}+(a_i - a_{i+1})t$, otherwise; so $\theta_{\phi,i}$ has
image $(a_i, a_{i+1})$.  
Then 
$h\circ \theta_{\phi,i}  : (0, 1) \to R$ is of class $C^k$, 
$(h\circ \theta_{\phi,i})^{(j)}$ is strongly bounded for $j = 0,\dots,k-1$,
and $|h\circ \theta_{\phi,i}^{(k)}|$ is decreasing. Let $\rho : (0, 1) \to (0, 1)$ be the $C^{\infty}$-bijection sending $t$ to $t^2$. By Lemma~\ref{karep}, the definable $C^k$-function $h\circ \theta_{\phi,i}\circ \rho: (0,1) \to R$ has strongly bounded $j$th derivative for $j = 0,\dots,k$. The function $g\circ \theta_{\phi,i}\circ\rho$  is still given by
a $1$-variable polynomial with strongly bounded coefficients in $R$,  and $\{g\circ\theta_{\phi,i}\circ \rho, h\circ \theta_{\phi,i}\circ \rho\}=\{\phi\circ \theta_{\phi,i}\circ \rho, f\circ(\phi\circ \theta_{\phi,i}\circ\rho)\}$.  The images of the functions $\phi\circ \theta_{\phi,i}\circ \rho$ with $i\in \{0,\dots, n_{\phi}\}$ cover the image of $\phi$ apart from finitely many points. So  adding finitely many constant functions with domain $(0,1)$ and values in $(0, 1)$ to the set $\{\phi\circ \theta_{\phi,i}\circ\rho:\ \phi\in \Phi,\ i=0,\dots,n_{\phi}\}$ we obtain a $k$-reparametrization of $f$ as claimed in the statement of the lemma.
\end{proof}        

\begin{cor}\label{corkrep} Let $f : X \to R$ be definable and strongly bounded with $X\subseteq R$. Then $f$ has a $k$-reparametrization, for every $k \ge 1$.
\end{cor}
\begin{proof} The case that $X$ is finite is obvious. Suppose $X$ is infinite,  $k\ge 1$. Since $X$ is a finite union of strongly bounded intervals and points, it has a $k$-parametrization $\Phi$ by univariate polynomial functions of degree $\le 1$. Now Lemma~\ref{krep} provides for every $\phi: (0,1)\to R$ in $\Phi$ a $k$-reparametrization $\Psi_{\phi}$ of $f\circ \phi: (0,1)\to R$; then $\{\phi\circ \psi:\ \phi\in \Phi,\ \psi\in \Psi_{\phi}\}$ is a $k$-reparametrization of $f$. 
\end{proof}

\noindent
Next one might reparametrize ``curves'' $(0,1)\to R^n$ with $n\ge 2$, but
there is nothing special about the univariate case here, so we do the general case:

\begin{lemma}\label{knrep} Let $k,m\ge 1$, and suppose that every strongly bounded definable function $X\to R$ with $X\subseteq R^l$, $l \le m$, has a $k$-reparametrization. Then every strongly bounded definable map $X \to R^n$ with $X\subseteq R^l$, $l\le m$ and $n\ge 1$ has a $k$-reparametrization.
\end{lemma}
\begin{proof}  By induction on $n\ge 1$. Suppose $F: X\to R^n$ and $f: X \to R$ with 
$X\subseteq R^m$ are definable, strongly bounded, and $F$ has a $k$-reparametrization. It is enough to show that then the
strongly bounded definable map $(F,f): X \to R^{n+1}$ has a $k$-reparametrization. The case of finite $X$ being trivial, assume $X$ is infinite.
Let $\Phi$ be a $k$-reparametrization of $F$ and let $\phi\in \Phi$,
$\phi : (0,1)^l \to R^m$, $l= \dim X \le m$. Applying the hypothesis of the lemma to the map $f\circ \phi:(0,1)^l \to R$ we obtain a $k$-reparametrization $\Psi_{\phi}$ of it. Then using Lemma~\ref{sbfact},
$\{\phi\circ \psi:\ \phi\in \Phi,\ \psi\in \Psi_{\phi}\}$ is a $k$-reparametrization of $(F, f)$.
\end{proof}

\noindent
{\bf Remark.} At one point we need a slight variant of this lemma, with the same proof: {\em Let $k,m\ge 1$, and suppose that every strongly bounded definable function $(0,1)^l\to R$ with $l \le m$ has a $k$-reparametrization. Then every strongly bounded definable map $(0,1)^l \to R^n$ with $l\le m$ and $n\ge 1$ has a $k$-reparametrization}.

\begin{cor}\label{unirepfinal} Let $n\ge 1$ and suppose $f : X \to R^n$ is definable and strongly bounded, with $X\subseteq R$. Then $f$ has a $k$-reparametrization, for every $k\ge 1$.
\end{cor}
\begin{proof} Immediate from Corollary~\ref{corkrep} and the case $m = 1$ of Lemma~\ref{knrep}.
\end{proof}

\section{Convergence}\label{par3} 

\noindent
In this section we continue to work with our o-minimal field $R$. A set $X\subseteq R$ is said to be {\em bounded} if
$X\subseteq [-r,r]$ for some $r\in R^{>}$. Since each definable subset of $R$ is a finite disjoint union of intervals and singletons, we can assign to each bounded definable set $X\subseteq R$ its {\em  length\/} $\ell(X)\in R$ so that $\ell(X)=b-a$ if $X$ is an interval $(a,b)$, $a < b$ in $R$, $\ell(X)=0$ for $X=\{a\}$, $a\in R$, and
$\ell(X)=\ell(X_1)+\ell(X_2)$ if $X$ is the disjoint union of definable subsets $X_1, X_2$.



Let $a,b\in R,\ a<b$, let $f: (a,b)\to R$ be definable and bounded (the latter  meaning that $\text{image}(f)$ is bounded), and let $L\in R^{>}$. For  $s\in (a,b)$ we declare ``$|f'(s)|>L$'' to mean:
$f$ is differentiable at $s$, and $|f'(s)|>L$. Suppose that $|f'(s)|>L$ for all $s\in (a,b)$. Then by the Mean Value Theorem (Lemma A.11),
$$\ell\big( \text{image}( f)\big) > L\cdot (b-a).$$  Let $E\subseteq R^m$ be definable and $(f_s)_{s\in E}$ a {\em definable family\/} of functions $f_s: (a,b)\to R$, meaning that the function $(s,t)\mapsto f_s(t): E\times (a,b)\to R$ is definable. We now use the 
observation above to obtain the following: 

\begin{lemma} \label{bn1} Suppose $N\in \N$ is such that $|f_s(t)|\le N$ for all $s\in E$ and $t\in (a,b)$. Then there is
$M\in \N$ such that for all $L\in R^{>}$ and $s\in E$,
$$\ell\big(\{t\in (a,b):\ |f_s'(t)|>L\}\big)\ \le\  M/L.$$
\end{lemma} 
\begin{proof} Let $L\in R^{>}$, $s\in E$, and set $X_{L,s}:=\{t\in (a,b):\ |f_s'(t)|>L\}$. O-minimality (Theorem A.2, Proposition A.4, Theorem A.14) gives a finite bound $m$ independent of $L,s$, and
disjoint intervals $(a_i, b_i)\subseteq X_{L,s}$, $i=1,\dots,m_{L,s}\le m$, such that $X_{L,s}\setminus\bigcup_{i=1}^{m_{L,s}}(a_i, b_i)$ is finite.
By the observation above, $2N>L (b_i-a_i)$ for $i=1,\dots, m_{L,s}$, so $2m_{L,s}N>L\ell\big(X_{L,s}\big)$. Thus $\ell\big(X_{L,s}\big)<  2mN/L$. 
\end{proof} 

\medskip\noindent 
{\bf Some notation and terminology}.  For definable open $U\subseteq R^{m+1}$, $V\Subset U$ means that 
$V$ is a definable open subset of $R^{m+1}$ with $V \subseteq U$ and $\dim(U \setminus V ) \le  m$.

A {\em cofinite subset of a set $X$\/} is a set $X_0\subseteq X$ such that $X\setminus X_0$ is finite.  Let $\pi_{m+1}: R^{m+1}\to R$ be  given by $\pi(t_1,\dots, t_m, t_{m+1})=t_{m+1}$.  Note that if $X\subseteq R^{m+1}$ is definable, then so is
$\pi_{m+1}(X)\subseteq R$. 

Recall: $|y|=\max\{|y_1|,\dots, |y_n|\}$ for $y$ in $R^n$.  For a definable map $f: X \to R^n$ with (necessarily definable)
$X\subseteq R^m$, we set $$\|f\|\ :=\ \sup_{a\in X} |f(a)|\in [0,+\infty].$$ 
Note that if $X$ is nonempty and closed and bounded in $R^m$ and $f$ is continuous, then this supremum is a maximum, by Corollary~\ref{dc2}.

\begin{lemma}\label{bn2}
Let $f : U \to R$, $U \Subset (0,1)^{m+1}$,  be a strongly bounded definable $C^1$-function. Suppose $ \partial f/\partial x_i$  is strongly bounded for $i=1,\dots,m$. Then  $$\{t\in \pi_{m+1}(U):\  \partial f/\partial x_{m+1}(-,t) \text{ is bounded}\}$$ is a cofinite subset of $\pi_{m+1}(U)$, and thus of $(0,1)$.
\end{lemma}

\begin{proof}
Suppose not.  Then the set 
$\{t\in \pi_{m+1}(U): \partial f/\partial x_{m+1}(-,t) \text{ is unbounded}\}$
contains an interval $(a,b)\subseteq (0,1)$.  Definable Selection (Proposition \ref{defsel}) then gives a definable family $(\gamma_L)_{L\in R^{>}}$ of maps  $$\gamma_L\ =\ (\gamma_{L,1},\dots,\gamma_{L,m})\ :\  (a,b)\to (0,1)^m$$ such that for all 
$L\in R^{>}$ and $t\in (a,b)$ we have $\big(\gamma_L(t),t\big)\in U$ and
$$\big|\frac{\partial f}{\partial x_{m+1}}\big(\gamma_L(t),t\big)\big|\  >\  L.$$
Take $N\in \N$  such that  $\|f\|\leq N$ and $\|\frac{\partial f}{\partial x_{i}}\|\leq N$ for $i=1,\ldots, m$. 
Let $f_L: (a,b) \to R$ be given by $f_L(t):=f\big(\gamma_L(t),t\big)$. Applying Lemma~\ref{bn1} to the definable families 
 $\big(3mN\gamma_{L,1}\big)_{L\in R^{>}},\dots, \big(3mN\gamma_{L,m}\big)_{L\in R^{>}}$, $(3f_L)_{L\in R^{>}}$ gives $M\in \N$ with the property that for all $L\in R^>$ there is a definable closed subset $X_L$ of $(a,b)$ with $\ell(X_L)\le M/L$ such that for all $t\in (a,b)\setminus X_L$ the map $\gamma_L$ is differentiable at $t$ and
 $$3mN|\gamma_{L,i}'(t)|\ \leq\  L, \quad  \quad i=1,\ldots,m,\quad \quad \text{and}$$
$$
 |f_L^{\prime}(t)|\ =\ \big|\sum_{i=1}^{m}\frac{\partial f}{\partial x_{i}}(\gamma_L(t),t)\cdot\gamma_{L,i}'(t) + \frac{\partial f}{\partial x_{m+1}}(\gamma_L(t),t)\big|\  \leq\ \frac{L}{3}.$$
Now take $L$ with $M/L < b-a$ and  $t\in (a,b)$ satisfying the $m+2$ displayed inequalities. The case $m=0$ gives an immediate contradiction, and for $m\ge 1$,
$$ \big|\sum_{i=1}^{m}\frac{\partial f}{\partial x_{i}}(\gamma_L(t),t)\cdot\gamma_{L,i}'(t)\big|\ \le\ mN\cdot \frac{L}{3mN}\ =\ 
\frac{L}{3},$$ contradicting the conjunction of the first and last inequality. 
\end{proof}

\medskip\noindent
The {\em normalization\/} of a function $\psi: I \to R$ on an interval $I=(a,b)\subseteq (0,1)$ is the function
$t\mapsto \psi\big((b-a)t+a\big): (0,1) \to R$; its image is $\psi(I)$. 

\medskip\noindent
Notation about ``changing the last variable'':  
For  $\phi : (0,1) \to R$ we set $$I_{\phi}\ :\ (0,1)^{m+1}\to R^{m+1}, \qquad (t_1,\dots,t_m,t_{m+1})\mapsto  \big(t_1,\dots,t_m,\phi(t_{m+1})\big),$$
and for $f : X \to  R^n$, $X\subseteq R^{m+1}$ we set
 $$f_{\phi}\ :=\ f\circ I_{\phi}\ :\ (I_\phi)^{-1}(X) \to R^n, \quad (t_1,\dots,t_m,t_{m+1})\mapsto  f\big(t_1,\dots,t_m,\phi(t_{m+1})\big).$$ 

\begin{lemma}\label{limit} Let  
$f : U \to R$, $U \Subset (0,1)^{m+1}$, be a strongly bounded definable $C^1$-function  such that $ \partial f/\partial x_i$  is strongly bounded for $i=1,\dots,m$.
Then there is for each $k\ge 1$ a $k$-parametrization $\Phi$ of a cofinite subset of $(0, 1)$ and a set $V\Subset U$ such that for every $\phi\in \Phi$: $I_{\phi}(V) \subseteq U$, $f_{\phi}$ is of class $C^1$ on $V$, and $ \partial f_\phi/\partial x_i$ is strongly bounded on $V$, for $i = 1, \dots, m + 1$.
\end{lemma} 
\begin{proof} 
Lemma \ref{bn2} yields a 
finite $F\subseteq \pi_{m+1}(U)$ such that $\frac{\partial f}{\partial x_{m+1}}(-, t)$ is bounded,  
for every $t \in \pi_{m+1}(U)\setminus F$. Set $V_0:=U\setminus \pi_{m+1}^{-1}(F)$, so 
$V_0\Subset U$ and $\pi_{m+1}(V_0)=\pi_{m+1}(U)\setminus F$. 
For each $t\in \pi_{m+1}(V_0)$ we take a point $a=a(t)\in (0,1)^m$ such that
$$ (a,t)\in V_0,\qquad \    \|\frac{\partial f}{\partial x_{m+1}}(-, t)\|\ \leq\  2|\frac{\partial f}{\partial x_{m+1}}(a, t)|.$$
We arrange by Proposition \ref{defsel} that $t\mapsto a(t): \pi_{m+1}(V_0)\to R^m$ is definable.  Let $\gamma:\pi_{m+1}(V_0)\to V_0$ be defined by $\gamma(t)=\big(a(t),t\big)$. 
Let $k\ge 1$. Corollary~\ref{unirepfinal} gives a $k$-reparametrization $\Phi_0$ of the map 
$$g\ :\ \pi_{m+1}(V_0)\to R^{m+2}, \quad t\mapsto  \big(\gamma(t), f(\gamma(t))\big).$$
We now change $V_0, \Phi_0$ to $V, \Phi$ as follows.  The Monotonicity Theorem~\ref{mono} yields for each $\phi\in \Phi_0$  a finite partition $\cal{P}_\phi$ of its domain
$(0,1)$ into subintervals and singletons such that on each interval
in $\cal{P}_{\phi}$ the function $\phi$ is either constant or strictly monotone. First, replace each  $\phi\in \Phi_0$ by the restrictions of $\phi$ to those intervals in $\cal{P}_\phi$ on which
$\phi$ is strictly monotone. Next, replace each of those restrictions with its normalization. The resulting set $\Phi$ of (strictly monotone) functions is still a $k$-parametrization of a cofinite subset of $\pi_{m+1}(V_0)$.  Now set 
$$V\ :=\  V_0\cap\bigcap_{\phi\in \Phi}I_{\phi}^{-1}(U)\ =\ V_0 \setminus \bigcup_{\phi\in \Phi}  I_{\phi}^{-1} [(0,1)^{m+1} \setminus U].$$
The injectivity and continuity of the $\phi\in \Phi$ gives $V \Subset U$. 
Let $\phi\in \Phi$. Then $I_\phi(V )\subseteq U$,  so $f_{\phi}$ is of class $C^1$ on $V$ in view of $k\geq 1$, and $\partial f_{\phi}/\partial x_i=(\partial f/\partial x_i)\circ I_{\phi}$ is strongly bounded on $V$ for
$i = 1,\dots,m$.  It only remains to show that $\partial f_{\phi}/\partial x_{m+1}$ is
strongly bounded on $V$. For $(t_1,\dots, t_{m+1})\in V$ we have
$$\frac{\partial f_{\phi}}{\partial x_{m+1}}(t_1,\dots, t_m, t_{m+1})\  =\ \phi'(t_{m+1})\cdot \big(\frac{\partial f}{\partial x_{m+1}}\circ I_\phi\big)(t_1,\dots, t_m, t_{m+1}),$$
 and by the properties of the  map $a$ we have for all $(t_1,\dots, t_m, t_{m+1})\in V$,
$$\big|\big(\frac{\partial f}{\partial x_{m+1}}\circ I_\phi\big)(t_1,\dots, t_m, t_{m+1})\big|\ \le\  2\big|\frac{\partial f}{\partial x_{m+1}}\big(\gamma\circ \phi)(t_{m+1})\big|.$$
Combining the last two displays it is enough to strongly bound
$$\phi' \cdot \frac{\partial f}{\partial x_{m+1}}\circ(\gamma\circ\phi)$$
 on $\pi_{m+1}(V)$. Since $\Phi$ is a $k$-reparametrization of 
$g|_{\pi_{m+1}(V)}$, we have:
\begin{enumerate}
\item[(i)]  $(\gamma\circ \phi)'$ is strongly bounded on $\pi_{m+1}(V)$, and
\item[(ii)] $ \big(f\circ \gamma\circ\phi\big)'$ is strongly bounded on $\pi_{m+1}(V)$.
\end{enumerate}
Let $\gamma=(\gamma_1,\dots,\gamma_m, \gamma_{m+1})$, $\gamma_i: \pi_{m+1}(V_0)\to R$.  By the Chain Rule (subsection ``Differentiability'' in part A of the Appendix),
   we have on $\pi_{m+1}(V)$: 
$$\big(f\circ \gamma\circ\phi\big)'\ =\  \sum_{i=1}^m (\gamma_i\circ \phi)'\cdot\frac{\partial f}{\partial x_i}\circ(\gamma\circ\phi)\ +\  \phi' \cdot \frac{\partial f}{\partial x_{m+1}}\circ(\gamma\circ\phi)$$
Now $\partial f /\partial x_i$ for $i = 1, \dots , m$ is 
strongly bounded, so by (i) the above $\sum_{i=1}^m$ is strongly bounded on $\pi_{m+1}(V)$.  Also the left hand side is strongly bounded on $\pi_{m+1}(V)$ by (ii),  hence the remaining term
$\phi' \cdot \frac{\partial f}{\partial x_{m+1}}\circ(\gamma\circ\phi)$ on the right is strongly bounded on $\pi_{m+1}(V)$ as well, which we already know to be enough. 
 \end{proof}

\begin{cor}\label{limitcor}
Let $k,n\ge 1$, $U \Subset (0,1)^{m+1}$ and let 
$f : U \to R^n$ be a strongly bounded definable $C^1$-map. Suppose also that $ \partial f/\partial x_i$  is strongly bounded for $i=1,\dots,m$.
Then there is a $k$-parametrization $\Phi$ of a cofinite subset of $(0, 1)$ and a set $V\Subset U$ such that for every $\phi\in \Phi$: $I_{\phi}(V) \subseteq U$, $f_{\phi}$ is of class $C^1$ on $V$, and $ \partial f_\phi/\partial x_i$ is strongly bounded on $V$ for $i = 1, \dots, m + 1$.
\end{cor}
\begin{proof} For $n=1$ this is Lemma~\ref{limit}. As an inductive assumption, let $f: U\to R^n$ be as in the hypothesis of the corollary
and $\Phi$ and $V$ as in its conclusion. Let $g: U \to R$ be a strongly bounded definable $C^1$-function such that $\partial g/\partial x_i$ is strongly bounded for $i=1,\dots,m$. Then the strongly bounded definable $C^1$-map $(f,g): U \to R^{n+1}$ has strongly bounded partial
$\partial(f,g)/\partial x_i=(\partial f/\partial x_i, \partial g/\partial x_i)$ for $i=1,\dots,m$. It now suffices to show that  there is a
$k$-parametrization $\Theta$ of a cofinite subset of $(0,1)$ and a set
$W\Subset U$ such that for all $\theta\in \Theta$: $I_{\theta}(W) \subseteq U$, $(f,g)_{\theta}$ is of class $C^1$ on $W$, and $ \partial (f,g)_\theta/\partial x_i$ is strongly bounded on $W$ for $i = 1, \dots, m + 1$.
To construct $\Theta$ and $W$, let $\phi\in \Phi$. Then applying Lemma~\ref{limit} to the function $g_{\phi}: V \to R$ gives a
$k$-parametrization $\Psi_{\phi}$ of a cofinite subset of $(0,1)$
and a set $V_{\phi}\Subset V$ such that for all $\psi\in \Psi_{\phi}$: $I_{\psi}(V_{\phi})\subseteq V$  and $(g_{\phi})_{\psi}=g_{\phi\circ \psi}$ is of class $C^1$
on $V_{\phi}$, and $\partial g_{\phi,\psi}/\partial x_i$ is strongly bounded on $V_{\phi}$.  Now we set 
$$\Theta:=\{\phi\circ \psi:\ \phi\in \Phi,\ \psi\in \Psi_{\phi}\}, \quad W:= \bigcap_{\phi\in \Phi} V_{\phi}.$$
It follows easily from Lemma~\ref{sbfact} that $\Theta$ and $W$ have the desired properties.   
\end{proof}

\noindent
To state the next corollary, let $U$ be a definable open subset of $R^{m+1}$.  Recall that for $t\in R$ we have
the definable open subset
$U^t$ of $R^m$ given by
 $$ U^{t}\:=\ \{(t_1,\dots, t_m)\in R^m:\ (t_1,\dots, t_m,t)\in U\}.$$  
We call a definable map $f: U \to R^n$ {\em of class $C^k$ in the first $m$ variables\/} if for every $t\in R$ the (definable) map 
$$f^t: U^t\to R^n, \qquad (t_1,\dots, t_m)\mapsto f(t_1,\dots, t_m,t)$$
is of class $C^k$. In that case $f^{(\alpha)}$ for $\alpha\in \N^m$ with $|\alpha|\le k$ denotes the definable map 
$$(t_1,\dots, t_m,t)\mapsto (f^t)^{(\alpha)}(t_1,\dots, t_m)\ :\ U \to R^n,$$
which for fixed $t$ is continuous as a function of $(t_1,\dots, t_m)$.

\begin{cor}\label{limitcorcor} Let $k,n \ge 1$, $U \Subset (0,1)^{m+1}$ and  let $f : U \to R^n$ be a strongly bounded definable map that is of class $C^k$ in the first $m$ variables, such that
$f^{(\alpha)}$ is strongly bounded for all $\alpha\in \N^{m}$ with $|\alpha| \le k$. 
Then for every $l\le k$ there is a $ V_l \Subset U$ and a $k$-parametrization
$\Phi_l$ of a cofinite
subset of $(0,1)$ such that for all $\phi \in \Phi_l$: $I_\phi(V_l) \subseteq U$, $f_{\phi}$ is of class $C^k$ on $V_l$ and
 $f_{\phi}^{(\alpha)}:= \big(f_{\phi}\big)^{(\alpha)}$ is strongly bounded on $V_l$ for all $\alpha\in \N^{m+1}$ with $|\alpha| \le k$, $\alpha_{m+1} \le l$.
\end{cor} 
\begin{proof} The last sentence in the subsection on $C^k$-maps in part A of the Appendix gives $V_0 \Subset U$ such that $f$ is of class $C^k$ on $V_0$. Then $V_0$ and $\Phi_0 = \{\text{id}|_{(0,1)}\}$
have the desired properties for $l=0$.  Suppose, inductively, that $l< k$ and
$V_l$ and $\Phi_l$ are as stated in the Corollary.
Let
$$\Delta\ :=\ \{\alpha\in \N^{m+1}\ :\ |\alpha|\le k-1,\ \alpha_{m+1}\le l\},$$
set $\tilde{n}:=\#\Delta \cdot \#\Phi_l$, and let $F_1,\dots, F_{\tilde n}: V_l\to R^n$ enumerate the set of $C^1$-maps 
$$ \{f_{\phi}^{(\alpha)}:\ V_l\to R^n:\ \alpha\in \Delta,\ \phi\in \Phi_l\}.$$ 
Then we can apply Corollary~\ref{limitcor} to $F := (F_1,\dots,F_{\tilde n}) : V_l \to R^{\tilde n\cdot n}$ in the role of $f$, and $V_l$,
$\tilde {n}\cdot  n$ instead of $U, n$. This gives a $k$-parametrization
$\Psi$ of a cofinite subset of $(0, 1)$ and a set $V_{l+1}\Subset V_l$ such
that for all $\psi\in \Psi$:  $I_{\psi}(V_{l+1}) \subseteq V_l$,
 $F_{\psi}$ is of class $C^1$ on $V_{l+1}$, and $\partial F_\psi/\partial x_i$ is strongly bounded on $V_{l+1}$ for $i = 1, \dots, m + 1$. Next we set 
 $$\Phi_{l+1}\ :=\ \{\phi\circ \psi:\ \phi\in \Phi_l,\ \psi\in \Psi\}.$$
Then $\Phi_{l+1}$ is a $k$-parametrization of a cofinite subset of $(0,1)$ and $I_\theta(V_{l+1})\subseteq U$, with $f_{\theta}$ of class $C^k$ for all $\theta\in \Phi_{l+1}$. 

 Let $\theta=\phi\circ \psi$ with $\phi\in \Phi_l, \psi\in \Psi$ and
 let $\alpha\in \N^{m+1}$, $|\alpha|\le k$, $\alpha_{m+1}\le l+1$;
it remains to show that then $f_\theta^{(\alpha)}$ is strongly bounded on $V_{l+1}$.
If $\alpha_{m+1} = 0$, then this holds because 
$f_{\theta}^{(\alpha)} = \big(f_{\phi}^{(\alpha)}\big)_{\psi}$ and $f_{\phi}^{(\alpha)}$ is strongly bounded on $V_l$. Suppose that $\alpha_{m+1}>0$. Then $\alpha=\beta+(0,\dots,0,j)$ with $\beta_{m+1}=0$ and $j=\alpha_{m+1}\ge 1$, so for $a=(a_1,\dots,a_m, a_{m+1})\in V_{l+1}$ we have
\begin{align*} f_{\theta}^{(\alpha)}(a)&\ =\ \frac{\partial^j f_{\theta}^{(\beta)}}{\partial x_{m+1}^j}(a)\ =\ \frac{\partial^j \big(f_{\phi}^{(\beta)}\big)_{\psi}}{\partial x_{m+1}^j}(a)\\\ &=\ \sum_{i=1}^j\frac{\partial^i f_{\phi}^{(\beta)}}{\partial x_{m+1}^i}\big(a_1,\dots, a_m, \psi(a_{m+1})\big)\cdot p_{ij}\big(\psi^{(1)}(a_{m+1}),\dots, \psi^{(j-i+1)}(a_{m+1})\big) 
\end{align*}
using Lemma~\ref{compfact} and the polynomials $p_{ij}$ from that lemma for the last equality.  Since we assumed inductively that the
$\frac{\partial^i f_{\phi}^{(\beta)}}{\partial x_{m+1}^i}$ are strongly bounded on $V_l$ and $\psi^{(1)},\dots, \psi^{(k)}$ are strongly bounded
on $(0,1)$, $f_{\theta}^{(\alpha)}$ is strongly bounded on $V_{l+1}$.
\end{proof}

\section{Finishing the proofs of the parametrization theorems}\label{par4}

\noindent
In this section we assume that our ambient o-minimal field $R$ is $\aleph_0$-saturated. 
We consider the following statements depending on $m$: \begin{enumerate}
\item[${\rm(I)}_m$] For all $k, n \ge 1$, every strongly bounded definable map $f : (0, 1)^m \to R^n$ has a $k$-reparametrization.
\item[${\rm(II)}_m$] For all $k \ge 1$, every strongly bounded definable set $X\subseteq R^{m+1}$ has a $k$-parametrization.
\end{enumerate}
It is clear that ${\rm(I)}_0$ and ${\rm(II)}_0$ hold;
${\rm(I)}_1$ holds by Corollary~\ref{unirepfinal}. We proceed by induction to show that ${\rm(I)}_m$ and ${\rm(II)}_m$ hold for all $m$. So let $m\ge 1$ and suppose that ${\rm(I)}_l$ holds for all $l\le m$ and that ${\rm(II)}_l$ holds for all $l < m$. We show that then ${\rm(II)}_m$ holds and next that ${\rm(I)}_{m+1}$ holds. 
For ${\rm(II)}_m$, let $k \ge 1$ and let $X \subseteq R^{m+1}$ be definable and strongly bounded. In order to show that $X$ has a $k$-parametrization we
can reduce to the case that $X$ is a cell in $R^{m+1}$; we do the more difficult of the two cases, namely $X = (f,g)_Y$ where 
$Y$ is a (strongly bounded) cell in $R^m$, and $f, g: Y \to R$ are
strongly bounded continuous definable functions with $f(y) < g(y)$ for all
$y\in Y$; the other case, where $X$ is the graph of such a function $Y\to R$, is left to the reader.

Using ${\rm(II)}_{m-1}$ we have a $k$-parametrization $\Phi$ of $Y$. 
Set $l:=\dim Y$. 
Let $\phi\in \Phi$ be given. Then $\phi: (0,1)^l\to Y$ and ${\rm(I)}_l$ gives a $k$-reparametrization $\Psi_{\phi}$ of the map $(f\circ \phi, g\circ \phi):(0,1)^l\to R^2$. 
For $\psi\in \Psi_{\phi}$ we have $\psi:(0,1)^l \to (0,1)^l$, and we define 
$\theta_{\phi,\psi} : (0, 1)^{l+1} \to X$ by
$$ \theta_{\phi,\psi}(s, t)\ :=\ \big((\phi\circ \psi)(s),\ (1-t)\cdot(f\circ \phi\circ \psi)(s)+ t\cdot(g\circ \phi\circ \psi)(s)\big)$$
where $(s,t)=(s_1,\dots, s_l,t)\in (0,1)^{l+1}$. 
 Then the set $\{\theta_{\phi,\psi}:\ \phi\in \Phi,\ \psi\in \Psi_{\phi}\}$ is readily seen to be a $k$-parametrization of $X$, and we have established ${\rm(II)}_m$.

\medskip\noindent
For ${\rm(I)}_{m+1}$ we need only do the case $n=1$ by the remark
following the proof of Lemma~\ref{knrep}.
So let $k\ge 1$ and let $f : (0, 1)^{m+1}\to R$ be a strongly bounded definable function; our job is to show that $f$ has a $k$-reparametrization.

In the rest of this proof $t$ ranges over the interval $(0,1)$.
By ${\rm(I)}_m$ there is for all $t$ a $k$-reparametrization of the function $f^t : (0, 1)^m \to R$ given by $f^t(s)=f(s,t)$.
Now $R$ is $\aleph_0$-saturated, and together with Definable Selection (see end of Appendix~$B$ for details on this use of model-theoretic compactness) this yields an $N\in \N^{\ge 1}$ and  definable 
families $(\phi^t_{1}),\dots, (\phi^t_{N})$ 
of maps $$\phi^t_{j}\ :\ (0,1)^m\to (0,1)^m\qquad(j=1,\dots,N)$$ such that
$\Phi^t:=\{\phi^t_{1},\dots, \phi^t_{N}\}$ is for every $t$ a $k$-reparametrization of
$f^t$.

\medskip\noindent
Now, for $j = 1,\dots,N$ we define the function $f_j : (0,1)^{m+1}\to R$ by $$f_j(s,t)\ :=\ f\big(\phi_{j}(s,t),t\big),$$
where $\phi_j:(0,1)^{m+1}\to (0,1)^m$ is given by $\phi_{j}(s,t):=\phi^t_{j}(s)$. Consider the map
$$F\ :=\ \big(\phi_1,\dots,\phi_N,f_1,\dots,f_N\big)\ :\ (0,1)^{m+1}\to  R^{Nm+N}.$$
Then the hypotheses of Corollary~\ref{limitcorcor} are satisfied for $F$
and $(0,1)^{m+1}$ in the role of $f$ and $U$, and $Nm + N$ for $n$: this is just restating that $\Phi^t$ is a $k$-reparametrization of $f^t$, uniformly in $t$.
 The conclusion of that corollary for $l=k$ gives a set $V \Subset (0, 1)^{m+1}$ and a $k$-parametrization $\Psi$ of a cofinite subset of $(0,1)$ such that for all $\psi\in \Psi$ the map $F_{\psi}: (0,1)^{m+1}\to R^{Nm+N}$ is of class $C^k$ on $V$ with strongly bounded $F_{\psi}^{(\alpha)}$ on $V$ for all $\alpha\in\N^{m+1}$ with $|\alpha|\le k$. 
 
For $j=1,\dots,N$ and $\psi\in \Psi$, let $\phi_j*\psi: (0,1)^{m+1}\to (0,1)^{m+1}$ be given by 
$$(\phi_j*\psi)(s,t)\ :=\ \big(\phi_j(s,\psi(t)),\psi(t)\big)\ =\ \big(\phi_j^{\psi(t)}(s), \psi(t)\big).$$
The images of the $\psi\in \Psi$ cover a set $(0,1)\setminus \{t_1,\dots, t_d\}$
and for every $t$ the images of $\phi_1^t,\dots, \phi_N^t$ cover
$(0,1)^m$, and thus the images of the above $\phi_j*\psi$ cover
$(0, 1)^{m+1}$ apart from finitely many hyperplanes $x_{m+1} = t_i$.
Setting
$$W\ :=\ \bigcup_{1\le j\le N,\ \psi\in \Psi} (\phi_j*\psi)(V)$$
it follows that the definable set $(0,1)^{m+1}\setminus W$ has dimension $\le m$. 
Using the now established ${\rm(II)}_m$, let $\Theta_1$ be a 
$k$-parametrization of $V$ and $\Theta_2$ a $k$-parametrization of  $(0,1)^{m+1}\setminus W$. 
For $\theta\in \Theta_2$ we have $\theta: (0,1)^l\to (0,1)^{m+1}$ with $l\le m$ and then ${\rm(I)}_l$ yields 
a $k$-reparametrization $\Lambda_{\theta}$ of the function $f\circ \theta:(0,1)^l\to R$. The required $k$-reparametrization of $f$ is now given by
$$ \{(\phi_j*\psi)\circ \chi:\ j=1,\dots,N,\ \psi\in \Psi,\chi\in \Theta_1\}\cup\{\theta\circ \hat{\lambda}:\ \theta\in \Theta_2,\lambda\in \Lambda_{\theta}  \}$$
where $\hat{\lambda}: (0,1)^{m+1}\to (0,1)^l$ (for $l\le m$ as above)
is given by $\hat{\lambda}(t_1,\dots, t_{m+1}):=\lambda(t_1,\dots, t_l)$. 
This finishes the proof of ${\rm(I)}_{m+1}$, and the induction is complete. In particular, Theorem~\ref{pars} is now established.  Theorem~\ref{parm} requires one more easy step and we leave this to the reader.

\begin{cor}\label{corstr} Let $k,n\ge 1$; suppose $X \subseteq [-1, 1]^n$ is definable, $d:=\dim X\ge 0$. Then there exists a finite set $\Phi$ of definable
$C^k$-maps $f: (0, 1)^{d} \to  R^n$ such that \begin{enumerate}
\item[$\rm{(i)}$] $\bigcup_{f\in \Phi} \operatorname{image}(f) = X$; 
\item[$\rm{(ii)}$] $|f^{(\alpha)}(t)| \le 1$ for all $f\in \Phi$ and
$\alpha\in \N^{d}$ with $|\alpha| \le k$ and all 
$t \in (0,1)^{d}$.
\end{enumerate}
\end{cor} 
\begin{proof} Let $\Phi^*$ be a $k$-parametrization of $X$. Then (i) holds for $\Phi^*$ instead of $\Phi$ and (ii) holds for $\Phi^*$ instead of $\Phi$, with a certain $c\in \N^{\ge 1}$ in place of $1$. Cover 
$(0, 1)^d$ with $(c+1)^d$ translates of the `box' $(0,\frac{1}{c})^d$ 
and for each such translate $B$, let $\lambda_B : (0, 1)^d \to B$ be the obvious affine bijection. Then the set of maps $f\circ \lambda_B$ as $f$ varies over $\Phi^*$ and $B$ over the above translates is the required $\Phi$, since $(f\circ \lambda_B)^{(\alpha)}=c^{-|\alpha|}\cdot\big(f^{(\alpha)}\circ \lambda_B\big)$ for such $f$ and $B$
and $\alpha\in \N^{d}$ with $|\alpha| \le k$.
\end{proof} 

\noindent 
Definable Selection and $\aleph_0$-saturation lead to a uniform version, as explained in more detail at the end of part B of the Appendix: 

\begin{cor}\label{unifstr} Let $d,k,m,n$ be given with $k,n\ge 1$ and suppose $E\subseteq R^m$ and
$$Z\ \subseteq\ E\times [-1,1]^n\ \subseteq\ R^{m+n}$$ are definable with 
$\dim Z(s)=d$ for all $s\in E$.  Then there are $N\in \N^{\ge 1}$ and a definable set 
$F \subseteq E\times R^d\times R^{Nn}$ such that for all $s\in E$, $F(s)\subseteq R^d\times R^{Nn}$ is the graph of a $C^k$-map $(f_1,\dots, f_N): (0,1)^d\to (R^n)^N=R^{Nn}$ such that: \begin{enumerate}
\item[$\rm{(i)}$] $\bigcup_{j=1}^N \operatorname{image}(f_j) = Z(s)$; 
\item[$\rm{(ii)}$] $|f_j^{(\alpha)}(t)| \le 1$ for $j=1,\dots,N$,
$\alpha\in \N^{d}$ with $|\alpha| \le k$, and 
$t \in (0,1)^{d}$.
\end{enumerate}
\end{cor}

\noindent
The proof of Corollary~\ref{unifstr} uses that $R$ is $\aleph_0$-saturated, but this corollary goes
through without this assumption: pass to an $\aleph_0$-saturated elementary extension and then go back. Thus it applies to o-minimal expansions of the real field to give Theorem~\ref{YG+}, and we can also combine it with Theorem~\ref{covthm} to give:

\begin{cor}\label{unifcov} Let $n\ge 1$ and let an o-minimal expansion $\tilde{\R}$ of the real field be given. Suppose $E\subseteq \R^m$ and
$Z\subseteq E\times [-1,1]^n\subseteq \R^{m+n}$ are definable.
Then there is for every $\epsilon>0$ an $e = e(\epsilon,n)$ and a $K$ with the
following property: for all $s\in E$ with $\dim Z(s)<n$ and all $T$, at most $KT^\epsilon$ many hypersurfaces in $\R^n$ of degree $\le e$ are enough to cover the set $Z(s)(\Q,T)$. 
\end{cor}

\noindent
The expression ``$e=e(\epsilon,n)$'' means: $e$ can be chosen to depend only on $\epsilon$ and $n$. The proof below uses the numbers
$\epsilon(d,n,e):=\frac{dneD(n,e)}{B(d,n,e)}$ from Section~\ref{P1}.

\begin{proof} Replacing $E$ by finitely many definable subsets over each of which
$\dim Z(s)$ takes a given value, we arrange that for a certain $d<n$ we have $\dim Z(s)=d$ for all $s\in E$.
If $d=0$, then we have $K\in \N^{\ge 1}$ such that
$\# Z(s)\le K$ for all $s\in E$, and so at most $K$ hypersurfaces in $\R^n$
of degree $\le 1$ are enough to cover $Z(s)$. Assume $d\ge 1$. Take $e\ge 1$ such that 
$\epsilon(d, n, e) \leq \epsilon$ and set $k := b(d, n, e)+1$ as in Theorem~\ref{covthm}.  Corollary~\ref{unifstr} gives an $N\in N^{\ge 1}$ and for every $s\in E$ maps $f_1,\dots, f_N: (0,1)^d\to R^n$
of class $C^k$ such that $Z(s)=\bigcup_{j=1}^N \operatorname{image}(f_j)$ and $|f_j^{(\alpha)}(t)|\le 1$ for $j=1,\dots, N$ and all
$\alpha\in \N^d$ with $|\alpha|\le k$ and all $t\in (0,1)^d$. Applying Theorem~\ref{covthm} to each map $f_j$ separately we obtain that for
$K := N \cdot C(d, n, e)$ at most $KT^\epsilon$ many hypersurfaces in $\R^n$ of degree $\le e$ are enough to cover the set $Z(s)(\Q,T)$. 
\end{proof}

\section{Strengthening and Extending the Counting Theorem}\label{elab}

\noindent
In this section we fix an o-minimal expansion $\tilde{\R}$ of the real field, and {\em definable} is with respect to $\tilde{\R}$. Throughout $n\ge 1$ and $E\subseteq \R^m$ and $X\subseteq E\times \R^n$ are definable. 

A closer look at the proof of Theorem~\ref{PWr} gives useful extra information about the definable subsets $V(s)$ of $X(s)^{\alg}$: Theorem~\ref{pwvar}. To express this information efficiently requires the notion of a {\em block family}, which is here simpler than in \cite{P2} and well suited to the inductive set-up of Section~\ref{pct}. See the subsection {\em Dimension\/} in part A of the Appendix for the local dimension $\dim_a$ used in defining blocks.

\subsection*{A block family version of the Counting Theorem} Let $d\le n$. 
A {\em block in $\R^n$ of dimension $d$\/}
is a definable connected open subset of a semialgebraic set $A\subseteq \R^n$ for which
$\dim_a A =d$ for all $a\in A$.
Thus the empty subset of $\R^n$ counts as a block in $\R^n$ of dimension $d$, but if $B$ is a nonempty block in $\R^n$ of dimension $d$, then $\dim B = d$. Also, a nonempty block of dimension $0$ in $\R^n$ consists just of one point. A {\em block family in $\R^n$ of dimension $d$\/} is a definable set $V\subseteq E\times \R^n$ 
all whose sections $V(s)$ are blocks in $\R^n$ of dimension $d$. Here are two easy lemmas:

\begin{lemma}\label{lemba} Suppose $U\subseteq \R^m$ is open and semialgebraic, $m\ge 1$, and $f: U \to \R^n$ is semialgebraic and maps $U$ homeomorphically onto
$f(U)$. Then $f$ maps any block $B\subseteq U$ in $\R^m$ of dimension $d\le m$ 
onto a block $f(B)$ in $\R^n$ of dimension $d$. 
\end{lemma}

\noindent
In the proof of Theorem~\ref{pwvar} we apply Lemma~\ref{lemba} for every $I\subseteq \{1,\dots,n\}$ to the map $a\mapsto b: \{a\in \R^n: a_i\ne 0 \text{ for }i\in I\}\to \R^n$ with $b_i=a_i^{-1}$ for $i\in I$ and $b_i=a_i$
for $i\notin I$; these maps extend the maps $f_I$ from Section~\ref{pct}.

\medskip
\begin{lemma}~\label{lembb} Let $B$ be a block in $\R^n$ of dimension $d\le n$. Then $B$ is a union of connected semialgebraic subsets of dimension $d$.
\end{lemma}
\begin{proof} Take semialgebraic $A\subseteq \R^n$ such that $\dim_a A=d$ for all $a\in A$, and $B$ is an open subset of $A$. For $b\in B$, take a semialgebraic open neighborhood $U$ of $b$ in $A$ such that $U\subseteq B$.
Now use that the connected components of $U$ are open in $A$, by Corollary~\ref{dcoc}, and thus of dimension $d$.  
\end{proof}

\begin{cor}\label{lembl} Let $Y\subseteq \R^n$ and $1\le d\le n$. \begin{enumerate}
\item[$\rm(i)$] if $B\subseteq Y$ and $B$ is a block in $\R^n$ of dimension $d$, then $B\subseteq Y^{\alg}$;
\item[$\rm(ii)$] if $V$ is a block family in $\R^n$ of dimension $d$, then the union of the sections of $V$ that are
contained in $Y$ is contained in $Y^{\alg}$. 
\end{enumerate} 
\end{cor} 

\noindent
For the inductive proof below we also define a {\em block family in $\R^0$ of dimension $0$\/} to be a definable set $V\subseteq E\times \R^0$, with $E\times \R^0$ identified with $E$ in the obvious way. 

\begin{theorem}\label{pwvar} 
Let $\epsilon$ be given. Then there are a natural number $N=N(X,\epsilon)\ge 1$, a block family 
$V_j\subseteq (E\times F_j)\times \R^n$ in $\R^n$ of dimension $d_j\le n$ with definable $F_j\subseteq \R^{m_j}$,  for $j=1,\dots,N$,
and a constant $c=c(X,\epsilon)$, such that: \begin{enumerate} 
\item[$\rm(i)$] $V_j(s,t)\subseteq X(s)$ for $j=1,\dots,N$ and $(s,t)\in E\times F_j$; 
\item[$\rm(ii)$] for all $T$ and all $s\in E$,  $X(s)(\Q,T)$ is covered by at most $cT^\epsilon$ blocks $V_j(s,t)$,  $(1\le j\le N,\ t\in F_j)$.
\end{enumerate}
\end{theorem}

\noindent
This yields an improved Theorem~\ref{PWr} as follows. Let $V_1,\dots, V_N$ and $c$ be as in Theorem~\ref{pwvar}. Then
for all $s\in E$ the definable set $V(s)\subseteq \R^n$ given by 
$$V(s)\ :=\ \bigcup_{d_j\ge 1, t\in F_j} V_j(s,t)$$
is contained in  $X(s)^{\alg}$ and
$\oN(X(s)\setminus V(s), T) \le cT^{\epsilon}$ for all $T$.

\begin{proof} If Theorem~\ref{pwvar} holds for definable sets $X_1,\dots, X_\nu\subseteq E\times \R^n$, $\nu\in \N$, then also for $X=X_1\cup\dots\cup X_\nu$. We shall tacitly use this below.  

We proceed by induction on $n$, and follow the proof of Theorem \ref{PWr} closely. Set $V_0(s):=$ interior of $X(s)$. 
Then Theorem~\ref{celldec} and Proposition~\ref{fdc} give $M\in \N^{\ge 1}$ such that for all $s\in E$, 
$$ \#\{\text{connected components of }V_0(s)\}\ \le\ M.$$
Definable Selection (Proposition~\ref{defsel}) and the lexicographic ordering on $\R^n$ give definable subsets $V_1,\dots,V_M$ of
$E\times \R^n$ such that for all $s\in E$ the sets $V_1(s),\dots, V_M(s)$ are connected (possibly empty),  open in
$V_0(s)$, pairwise disjoint, with $V(s)\ =\ \bigcup_{i=1}^M V_i(s)$.
So $V_1,\dots, V_M$ are block families in $\R^n$ of dimension $n$;
we make them the first $M$ of the $V_1,\dots, V_N$ to be constructed.
Now replacing
$X$ with $X\setminus V_0$ we arrange that $X(s)$ has empty interior for all $s\in E$. Applying Lemma~\ref{lemba} to the natural extensions of the maps $f_I$, $I\subseteq \{1,\dots,n\}$,
 we arrange also that $X(s)\subseteq [-1,1]^n$ for all $s\in E$.

 Next, take $e$ and  $k=k(n,e)$ as in the proof of Theorem~\ref{piwi}. So we have $C=C(X,\epsilon)\in \R^{>}$ such that for any $s\in E$, $X(s)(\Q,T)$ is covered by at most $CT^{\epsilon/2}$ many hypersurfaces in $\R^n$ of degree $\le e$. Therefore it suffices to  find $V_1,\dots, V_N$ and $c$ as in the theorem but with $\rm(ii)$ replaced by
\begin{enumerate}
    \item[$\rm(ii)^{*}$] for all $T$, all $s\in E$, and all hypersurfaces $H$ of degree $\leq e$, $(X(s)\cap H)(\Q,T)$ is covered by at most $\frac{c}{C}T^{\epsilon/2}$ blocks $V_j(s,t)$, $(1\le j\le N,\ t\in F_j)$;
\end{enumerate} 
We use again the semialgebraic sets $\Hy, \mathcal{C}_1,\dots, \mathcal{C}_L \subseteq F\times \R^n$, and the definable sets $Y_{l}\subseteq E\times F\times \R^{n_{l}}$,  $l=1,\ldots, L$, as in the proof of Theorem \ref{piwi}. Since $n_{l}<n$, the induction assumption gives a natural number $N_{l}=N(Y_l,\epsilon)\ge 1$,
a block family $$W_{l,i} \ \subseteq\  \big((E\times F)\times G_{l,i}\big) \times \R^{n_l}$$ in $\R^{n_l}$ of dimension $d_{l,i}\leq n_{l}$ with definable $G_{l,i}\subseteq\R^{m_{l,i}}$, for $i=1,\dots, N_l$, and $B_{l}=B_{l}(Y_{l},\epsilon)\in \R^{>}$, such that
\begin{enumerate}
\item[$\rm(i)'$] $W_{l,i}(s,t,g)\subseteq Y_{l}(s,t)$ for $i=1,\dots, N_l$, $(s,t,g)\in (E\times F)\times G_{l,i}$; 
\item[$\rm(ii)'$] for all $T$ and all $(s,t)\in E\times F$,  $Y_{l}(s,t)(\Q,T)$ is covered by at most $B_{l}T^{\epsilon/2}$ blocks $W_{l,i}(s,t,g)$, $(1\le i\le N_l,\ g\in G_{l,i})$. 
\end{enumerate}
Set $N:=N_1+\cdots+N_{L}$, and for $l=1,\dots, L$, $1\le i\le N_l$ and
 $j=N_1+ \cdots + N_{l-1} + i$, set  $F_j:= F\times G_{l,i}$, 
and let $V_j\subseteq (E\times F_j)\times \R^n$ be the definable set given by 
$$V_j\big(s,(t,g)\big)=\mathcal{C}_{l}(t)\cap p^{-1}_{\i^{l}}\big(W_{l,i}(s,t,g)\big), \qquad(s\in E,\ t\in F,\ g\in G_{l,i}),$$
so $V_j$ is a block family in $\R^n$ of dimension $d_{l,i}<n$, by Lemma~\ref{lemba}.
It is easy to check that $V_1,\dots, V_N$ and $c:= C(B_1+\cdots+B_{L})$ are as desired.
\end{proof}


\subsection*{A generalization.}
In this subsection we fix $d \geq 1$.
Instead of rational points we now allow points with coordinates in a $\Q$-linear subspace of $\R$ of dimension $\le d$. 
Let $\lambda=(\lambda_1,\dots,\lambda_d)\in \R^d$, and set $\Q\lambda:=\Q\lambda_1+ \cdots +\Q\lambda_d\subseteq \R$. 
For $a\in \Q\lambda$ we set 
$$\H_{\lambda}(a)\ :=\ \min \{\H(q):\ q\in \Q^d,\ q\cdot \lambda=a\}\in \N^{\ge 1}.$$
Here $q\cdot \lambda:=q_1\lambda_1+ \cdots + q_d\lambda_d$. We define a height function $\H_{\lambda}$ on $(\Q\lambda)^n\subseteq \R^n$ by
$$ \H_{\lambda}(a)\ =\ \max\{\H_{\lambda}(a_1),\dots,\H_{\lambda}(a_n)\}\ \text{ for }
a=(a_1,\dots,a_n)\in (\Q\lambda)^n.$$
For $Y\subseteq \R^n$ we introduce its finite subsets $Y_{\lambda}(T)$ and their cardinalities:
$$Y_{\lambda}(T)\ :=\ \{a\in Y\cap (\Q\lambda)^n:\ \H_{\lambda}(a)\le T\}, \qquad 
\oN_{\lambda}(Y,T)\ :=\ \# Y_{\lambda}(T).$$

\begin{theorem}\label{pwupcor} Let any definable $Y\subseteq \R^n$ and any $\epsilon$ be given. Then there is a constant $c=c(Y,d,\epsilon)\in \R^{>}$ such that for all $T$ and all $\lambda\in \R^d$,
$$ \oN_{\lambda}(Y^{\tr},T)\ \le\ cT^{\epsilon}.$$ 
\end{theorem}

\subsection*{Proof of Theorem~\ref{pwupcor}} First a useful lemma about blocks:

\begin{lemma}\label{lempath} If $B$ is a block in $\R^n$ $($of some dimension$)$ and $p,q\in B$, then $\gamma(0)=p$ and $\gamma(1)=q$ for some continuous semialgebraic path $\gamma: [0,1]\to B$.
\end{lemma} 
\begin{proof} Even better, let $B$ be a connected open subset of a semialgebraic set $A\subseteq \R^n$, 
and let $p\in B$. We claim: there is for every $q\in B$ a continuous semialgebraic path $\gamma: [0,1]\to \R^n$ with $\gamma(0)=p$, $\gamma(1)=q$, and $\gamma([0,1])\subseteq B$. To see this, let $B(p)$ be the set of all $q\in B$ for which there is such a path. The sets $B(p)$ as $p$ ranges over 
$B$ form a partition of $B$, so it is enough to show that the $B(p)$ are open in $B$, which reduces to showing that $B(p)$ is a neighborhood of $p$ in $B$.
Now $B$ is open in $A$, so we have a semialgebraic open subset $U$ of $A$ with $p\in U\subseteq B$. The connected component $C$ of $U$ with $p\in C$ is
open in $U$ and semialgebraic by Corollary~\ref{dcoc} and the remarks preceding it. These remarks also give $C\subseteq B(p)$. 
\end{proof}

\begin{cor}\label{corpath} If $B$ is a block in $\R^{m}$ $($of some dimension$)$,
 $A$ is a semialgebraic subset of $\R^{m}$ with $B\subseteq A$, and $\phi: A \to \R^n$ is a continuous semialgebraic map such that 
$\phi(B)$ has more than one point, then
$\phi(B)=\phi(B)^{\alg}$.
\end{cor}
\begin{proof} Use that the $\phi$-image of a path $\gamma$ as in Lemma~\ref{lempath} is a connected semialgebraic subset of $\phi(B)$. 
\end{proof} 

\noindent
The next result is basically a consequence of Theorem~\ref{pwvar}, as the proof will show.

\begin{theorem}\label{pwup} Given $\epsilon$, there are a natural number $N=N(X,d,\epsilon)\geq 1$, a definable set
$V_j\subseteq (E\times \R^d\times F_j)\times \R^n$ with definable $F_j\subseteq \R^{m_j}$, for $j=1,\dots,N$, and a constant $c=c(X,d,\epsilon)$, such that
for $j=1,\dots,N$ and all $(s,\lambda,t)\in E\times \R^d\times F_j$:\begin{enumerate}
\item[$\rm(i)$] $V_j(s,\lambda,t)\subseteq X(s)$ and $V_j(s,\lambda,t)$ is connected;
\item[$\rm(ii)$] if $\dim V_j(s,\lambda,t)\ge 1$, then $V_j(s,\lambda,t)\subseteq X(s)^{\alg}$,
\end{enumerate}
and such that for all $T$ and $(s,\lambda)\in E\times \R^d$, the set  
$X(s)_{\lambda}(T)$ is covered by at most $cT^\epsilon$ sections 
$V_j(s,\lambda, t),\ (1\leq j \leq N,\ t\in F_j)$.
\end{theorem}

\noindent
This yields a family version of Theorem~\ref{pwupcor} as follows. Let $V_1,\dots, V_N$ and $c$ be as in Theorem~\ref{pwup}. Then
for all $s\in E$ the definable set $V(s)\subseteq \R^n$ given by 
$$V(s)\ :=\ \bigcup\{V_j(s,\lambda, t): \ 1\le j\le N,\ (\lambda,t)\in \R^d\times F_j,\ \dim V_j(s,\lambda,t)\ge 1\}$$
is contained in  $X(s)^{\alg}$ and
$\oN_{\lambda}(X(s)\setminus V(s), T) \le cT^{\epsilon}$ for all $T$.

\begin{proof} Let $\pi: \R^d\times (\R^d)^n\to \R^n$ be given by $\pi(\lambda,a_1,\dots, a_n)=(\lambda \cdot a_1,\dots, \lambda \cdot a_n)$, where $a_1,\dots, a_n\in \R^d$.
Set
 $$ X^*\ :=\ \{(s,\lambda, a_1,\dots, a_n)\in (E\times \R^d)\times (\R^d)^n:\ 
 \big(s,\pi(\lambda, a_1,\dots, a_n)\big)\in X\},$$
 viewed as a definable family of subsets of $(\R^d)^n$. Note that
 for $s\in E$ and $\lambda\in \R^d$,
 $$(*)\qquad  \pi\big(\{\lambda\}\times X^*(s,\lambda)\big)\ \subseteq\ X(s), \qquad  \pi\big(\{\lambda\}\times X^*(s,\lambda)(\Q, T)\big)\ =\ X(s)_{\lambda}(T).  $$
We apply Theorem~\ref{pwvar} to $X^*$ in the role of $X$. It
gives $N=N(X^*,\epsilon)\ge 1$, a
block family $V_j^*\subseteq (E\times \R^d\times F_j)\times (\R^d)^n$ in $(\R^d)^n=\R^{dn}$ with definable $F_j\subseteq \R^{m_j}$, for $j=1,\dots,N$, and a constant $c=c(X^*,\epsilon)$ such that: \begin{enumerate}
\item[$\rm(i)^*$] $V_j^*(s,\lambda,t)\subseteq X^*(s,\lambda)$ for $j=1,\dots,N$ and $(s,\lambda,t)$ in $E\times \R^d\times F_j$; 
\item[$\rm(ii)^*$] for all $T$ and all $(s,\lambda)\in E\times \R^d$, the set  $X^*(s,\lambda)(\Q,T)$ is covered by at most $cT^\epsilon$ sections $V_j^*(s,\lambda, t)$, ($1\leq j \leq N$, $t\in F_j$).
\end{enumerate}
Now we set for $j=1,\dots, N$,
$$V_j\ :=\ \{\big(s,\lambda, t, \pi(\lambda, a)\big)\in (E\times \R^d\times F_j)\times \R^n:\ (s,\lambda,t,a)\in V_j^*\},$$
so $V_j(s,\lambda,t)=\pi\big(\{\lambda\}\times V_j^*(s,\lambda,t)\big)$ for $(s,\lambda,t)\in E\times \R^d\times F_j$. 
We now show that $V_1,\dots, V_N$ and $c(X,d,\epsilon):= c(X^*,\epsilon)$ have the desired properties.
Clause (i) is satisfied using $\rm(i)^*$ and $(*)$, and (ii) is satisfied in view of 
Corollary~\ref{corpath}. The rest follows from $\rm(ii)^*$ and $(*)$. 
\end{proof}

\subsection*{Extending the Counting Theorem to Algebraic Points}
Throughout this subsection we fix $d \geq 1$. Instead of rational points we now count algebraic points whose coordinates are of degree at most $d$ over $\Q$. We define the corresponding height of an algebraic number $\alpha \in \R$ with $[\Q(\alpha):\Q]\le d$  by
$$\H^{\text{poly}}_d(\alpha)\ :=\ \text{min}\{ \H(\xi):\ \xi\in \Q^d,\ \alpha^d+ \xi_1\alpha^{d-1} + \cdots + \xi_d=0\}\in \N^{\ge 1}. $$
(For us this height is notationally more convenient than the height for real algebraic numbers used by Pila in [P2]. The two heights are related as follows, where we use an extra subscript P for the height in [P2]: for $\alpha \in \R$ with $[\Q(\alpha):\Q]\le d$,
$$ \H^{\text{poly}}_{\text{P},d+1}(\alpha)\ \le\ \H^{\text{poly}}_d(\alpha)\ \le\ \H^{\text{poly}}_{\text{P},d+1}(\alpha)^2.$$
Thus the results below for our height also hold for the other height.)

\medskip\noindent
We extend the above height to all $\alpha \in \R$ by $\H^{\text{poly}}_d(\alpha):=\infty$ if $[\Q(\alpha):\Q]>d$, and to all points $\alpha=(\alpha_1,\dots, \alpha_n)\in \R^n$ by $\H^{\text{poly}}_d(\alpha):=\max\{\H^{\text{poly}}_d(\alpha_1),\dots,\H^{\text{poly}}_d(\alpha_n)\}$. 
For $Y\subseteq \R^n$ we introduce its finite subsets $Y_d(T)$ and their cardinalities:
$$ Y_d(T)\ :=\ \{\alpha \in Y:\ \H_d^{\text{poly}}(\alpha)\le T\},\qquad
\oN_{d}(Y,T)\ :=\ \# Y_d(T).$$

\begin{theorem}\label{pwalg}
Let $Y\subseteq \R^n$ be definable, and let $\epsilon$ be given. Then there is a constant $c=c(Y,d,\epsilon)$ such that for all $T$, $$N_{d}(Y^{\tr},T)\ \leq\ cT^{\epsilon}.$$
\end{theorem}

\noindent
We shall use the following easy consequence of semialgebraic cell decomposition: 

\begin{lemma}\label{lalg} Let $A_{n,d}\subseteq \R^{n\times d}\times \R^n$
be the semialgebraic set
$$\{(\xi,\alpha)\in \R^{n\times d}\times \R^n:\  \alpha_i^d+\xi_{i1}\alpha_i^{d-1} + \cdots +\xi_{id}=0 \text{ for }i=1,\ldots,n\}. $$ Then we have a natural number
$L=L(n,d)\ge 1$, a semialgebraic set 
$D_l~\subseteq~\R^{n\times d}$ with a semialgebraic continuous map $\phi_l: D_l\to  \R^n$, for $l=1,\dots, L$, such that $A_{n,d}=\bigcup_{l=1}^L \operatorname{graph}(\phi_l)$.
It follows that for all $\alpha \in \R^n$ with $ \H^{\operatorname{poly}}_d(\alpha)<\infty$
there is an $l\in \{1,\ldots, L\}$ and a $\xi\in D_l$ such that 
$\phi_l(\xi)=\alpha$ and $\H(\xi)=\H^{\operatorname{poly}}_d(\alpha)$. 
\end{lemma}

\noindent 
Towards Theorem~\ref{pwalg} we first prove something stronger:

\begin{theorem}\label{palg}
Let $\epsilon$ be given. Then there are  $N=N(X,d,\epsilon)\in \N^{\ge 1}$, a definable set
$V_j\subseteq (E\times F_j)\times \R^n$ with definable $F_j\subseteq \R^{m_j}$, for $j=1,\dots,N$, and a constant $c=c(X,d,\epsilon)$, such that
for $j=1,\dots,N$ and all $(s,t)\in E\times F_j$:
\begin{enumerate}
\item[$\rm(i)$] $V_j(s,t)\subseteq X(s)$ and $V_j(s,t)$ is connected;
\item[$\rm(ii)$] if $\dim V_j(s,t)\ge 1$, then $V_j(s,t)\subseteq X(s)^{\alg}$,
\end{enumerate}
and such that for all $T$ and $s\in E$, the set  
$X(s)_d(T)$ is covered by at most $cT^\epsilon$ sections 
$V_j(s, t),\ (1\leq j \leq N,\ t\in F_j)$.
\end{theorem}
\begin{proof} Let $\pi:\R^{n\times d}\times \R^n\to \R^{n\times d}$
be the obvious projection map. Take $L$ and
$\phi_1: D_1\to \R^n,\dots, \phi_L: D_L \to \R^n$ as in Lemma \ref{lalg}. Let $l\in \{1,\ldots,L\}$. We set
\begin{align*} X_l\ &:=\ \{(s,\xi,\alpha)\in  E \times D_l\times \R^n:\ \alpha\in X(s),\ \phi_l(\xi)=\alpha\},\\
   Y_l\ &:=\ \{(s,\xi)\in E \times D_l:\ \xi \in \pi\big(X_l(s)\big)\}\ =\ \{(s,\xi)\in E\times D_l:\ \phi_l(\xi)\in X(s)\},  
   \end{align*}
so for $s\in E$ we have $\phi_l\big(Y_l(s)\big)\subseteq X(s)$, and by Lemma~\ref{lalg}, for all $T$, $$X(s)_d(T)\ =\ \bigcup_{l=1}^{L} \phi_l\big(Y_l(s)(\Q,T)\big).$$ We now apply Theorem \ref{pwvar} to $Y_l$ in the role of $X$, and get $N_l=N_l(Y_l,\epsilon)\in \N^{\ge 1}$, a
block family $V_{l,i}\subseteq (E\times F_{l,i})\times \R^{n\times d}$ in $\R^{n\times d}$ with definable $F_{l,i}\subseteq \R^{m_{l,i}}$, for $i=1,\dots,N_l$, and a constant $c_l=c_l(Y_l,\epsilon)\in \R^{>}$ such that: \begin{enumerate}
\item[$\rm(i)$] $V_{l,i}(s,t)\subseteq Y_l(s)$ for $i=1,\dots,N_l$ and $(s,t)$ in $E\times F_{l,i}$; 
\item[$\rm(ii)$] for all $T$ and all $s\in E$, the set  $Y_l(s)(\Q,T)$ is covered by at most $c_lT^\epsilon$ blocks $V_{l,i}(s,t)$, ($1\leq i \leq N_l$, $t\in F_{l,i}$).
\end{enumerate}
Set $N:=N_1+\cdots+N_{L}$, and for $1\le i\le N_l$ and
 $j=N_1+ \cdots + N_{l-1} + i$, set  $F_j:= F_{l,i}$, 
and let $V_j\subseteq (E\times F_j)\times \R^n$ be the definable set given by  
$$V_j(s,t)\ :=\ \phi_l\big(V_{l,i}(s,t)\big), \qquad(s\in E,\ t\in F_j).$$
It is easily verified using Lemma \ref{corpath} that $V_1,\ldots, V_N$ and $c(X,d,\epsilon):=c_1+\cdots+c_{L}$ have the properties stated in the Theorem.
\end{proof}

\noindent
Just as with Theorem~\ref{pwup} this leads to a family version of Theorem \ref{pwalg} as follows. Let $V_1,\dots, V_N$ and $c$ be as in Theorem~\ref{palg}. Take the definable set $V\subseteq E\times \R^n$ such that for all $s\in E$,
$$V(s)\ :=\ \bigcup\{V_j(s, t): \ 1\le j\le N,\ t\in F_j,\ \dim V_j(s,t)\ge 1\}.$$ Then for all $s\in E$ and all $T$ we have
$$V(s)\ \subseteq\ X(s)^{\alg} \quad \text{and} \quad
\oN_{d}(X(s)\setminus V(s), T)\ \le\ cT^{\epsilon}.$$

\newpage

\begin{appendices}
\vspace*{0.01 cm}

\begin{center}
\textbf{APPENDIX ON O-MINIMALITY}
\end{center}

\vspace{0.5 cm}

\noindent
O-minimality as a subject started in \cite{vdD0} and \cite{PS}. Here we focus on material used in this paper. We give the key definitions in full detail, with examples, but state most results without proof. These proofs are in \cite{vdD2} as to general facts about o-minimal fields and the semialgebraic case, and in \cite{B-M, R_an, vdD-M, DMM,  RSW, W1} as to specific examples beyond the semialgebraic case. 
Notation is as in the introduction to this paper, in particular, $l,m,n\in \N=\{0,1,2,\dots\}$. 

We have divided this appendix into two parts. Part A is enough for sections 1--6, but section~\ref{par4} requires some model-theoretic compactness, {\it alias} saturation, which is fully exposed in part B.

\section{O-minimal fields}

\noindent

\subsection*{Structures} Let $M$ be a nonempty set.  We consider the finite cartesian powers  
$$M^n\ :=\ \{a=(a_1,\dots, a_n):\  a_1,\dots, a_n\in M\},$$ 
identifying in the usual way $M^1$ with $M$ and $M^{m+n}$ with $M^m\times M^n$. 
A {\em structure on $M$} is a sequence $\S$=$(\S_n)$ such that for all $n$,

\begin{enumerate}
\item $\S_n$ is a boolean algebra of subsets of $M^n$, that is, all $X\in \S_n$ are subsets of $M^n$, $M^n\in \S_n$, and for all $X,Y\in \S_n$ also $X\cup Y, X\cap Y, X\setminus Y\in \S_n$. 
\item For $n\ge 2$ and $1\le i < j\le n$ the diagonal $\{a\in M^n: a_i=a_j\}\in \S_n$.
\item  If $X\in \S_n$, then $M\times X\in \S_{n+1}$ and $X\times M\in \S_{n+1}$.
\item If $X\in \S_{n+1}$, then $\pi(X)\in \S_n$, where $\pi: M^{n+1}\to M^n$ is the projection map given by $\pi(a_1,\dots, a_n, a_{n+1})=(a_1,\dots, a_n)$.
\end{enumerate}
Let $\S$ be a structure on $M$. The definition of ``structure'' lacks symmetry, but in fact, if $X\in \S_n$ and $\sigma$ is a permutation of $\{1,\dots,n\}$, then $$\{\big(a_{\sigma(1)},\dots, a_{\sigma(n)}\big):\ a=(a_1,\dots, a_n)\in X\}\in \S_n.$$
Given a map $f: X \to M^n$ with $X\subseteq M^m$, we say that {\em $f$ belongs to $\S$\/} (or {\em $\S$ contains $f$}) if its graph, as a subset of $M^{m+n}$, belongs to $\S_{m+n}$; in that case
$X\in \S_m$, $f(X)\in \S_n$, $f^{-1}(Y)\in \S_m$ for every $Y\in \S_n$, and the restriction $f|_{X_0}: X_0\to R^n$ belongs to $\S$ for every $X_0\subseteq X$ in $\S_m$. If $f: X \to M^n$ and $g: Y\to M^l$ belong to $\S$, where $X\subseteq M^m$
and $Y\subseteq M^n$, then the composition $g\circ f: X\cap f^{-1}(Y)\to M^l$ belongs to $\S$.   
The class of all structures on $M$ is partially ordered by $\subseteq$: 
$$\S\subseteq \S'\ :\Longleftrightarrow\ \S_n\subseteq \S_n' \text{ for all }n.$$
Any collection $\mathcal{C}$ of sets $X\subseteq M^n$ for various $n$ gives rise to the {\em least\/} structure $\S$ on $M$ that contains every $X\in \mathcal{C}$, where ``least'' is with respect to $\subseteq$.

\subsection*{Ordered fields} Let $R$ be an ordered field: a field with a (strict) total order $<$ on its underlying set such that for all $a,b,c\in R$ we have $$a<b\Rightarrow a+c < b+c, \qquad a< b,\ 0 < c \Rightarrow ac<bc.$$ The case to keep in mind is the field $\R$ of real numbers with its usual ordering, but in sections
4--8 we work in bigger ambient ordered fields, since results in that setting have consequences for $\R$ that are less easy to obtain otherwise. (This is where model theory comes into play.) The ordered field $\Q$ of rational numbers embeds uniquely into $R$ as an ordered field. We use also the signs $\le, >, \ge$ with the usual meaning derived from $<$, and set 
$R^{>}:=\{a\in R:\ a>0\}$, $R^{\ge}:= \{a\in R:\ a\ge 0\}$. For $a\in R$ we set
$|a|:= a$ if $a\ge 0$ and $|a|:= -a$ if $a\le 0$. For $a=(a_1,\dots, a_n)\in R^n$ we set $|a|:= \max(|a_1|,\dots, |a_n|)\in R^{\ge}$, which by convention equals $0$ if $n=0$.
 
An \textit{interval} in $R$ is a set $(a,b):=\{x\in R:\ a< x < b\}$, where $a,b \in R_{\infty}:= R \cup \{-\infty,\infty\}$, $a< b$, where we extend $<$ to a total ordering on
$R_{\infty}$ by $-\infty < x < \infty$ for all $x\in R$. For $a\le b$ in $R_{\infty}$ we also set $[a,b]:=\{x\in R_{\infty}:\ a\le x \le b\}$, but we do not call this an interval.  
We endow $R$ with the order topology on its underlying set: it has the collection of intervals as a basis, and is a hausdorff topology. We also equip $R^n$ with the corresponding product topology.

 We call $R$ {\em real closed\/} if $R^{>}=\{b^2:\ 0\ne b\in R\}$ and
every polynomial $p(x)\in R[x]$ of odd degree has a zero in $R$. (This is equivalent to
the field $R[\imag]$ with $\imag^2=-1$ being algebraically closed.) In particular, the ordered field $\R$ of real numbers is real closed, and in some precise sense, all real closed fields have the same elementary properties as 
$\R$ (Tarski); we do not explicitly use that fact. Here and below $\R$ denotes the {\em ordered field\/} of real numbers, not just the set of real numbers.

\subsection*{O-Minimal Structures} Let $R$ again be an ordered field.
A {\em structure on $R$\/} is a structure $\S$ on its underlying set such that
\begin{itemize}
\item[(5)] $\{(a,b)\in R^2: a<b\}\in \S_2$ and the graphs of $+, \cdot: R^2\to R$ lie in $\S_3$.
\end{itemize}
Let $\S$ be a structure on $R$ with $\{a\}\in \S_1$ for all $a\in R$. Then
every interval is in $\S_1$, for every polynomial $p\in R[x_1,\dots, x_n]$ the corresponding function $a\mapsto p(a): R^n\to R$ belongs to $\S$, and so $\S_n$ contains the sets
$$\{a\in R^n:\ p(a)=0\} \text{ and\ } \{a\in R^n:\ p(a)>0\}.$$
For real closed $R$, a {\em semialgebraic subset of $R^n$\/} is a
finite union of sets 
$$\{a\in R^n: p(a)=0,\ q_1(a)>0,\dots, q_m(a)>0\}, \quad (p, q_1,\dots, q_m\in R[x_1,\dots, x_n]),$$
and setting $\S_n:=\{\text{semialgebraic subsets of $R^n$}\}$ gives by the Tarski-Seidenberg theorem a structure
$\S=(\S_n)$ on the ordered field $R$. This is the least structure on $R$ containing $\{a\}$ for all $a\in R$. In this case $\S_1$ contains exactly the finite unions of the sets $\{a\}$ with $a\in R$ and intervals. This fact about
$\S_1$ is a surprisingly strong minimality property of $\S$ which we now axiomatize:   

\medskip\noindent
An {\em o-minimal structure on $R$} is a structure on the ordered field $R$ such that

\begin{itemize}
\item[(6)] $\{a\}\in \S_1$ for every $a\in R$ and every element of $\S_1$ is a finite union of one-element subsets of $R$ and intervals.
\end{itemize} 

\medskip
\noindent
One can show that $R$ must be real closed if there is an o-minimal structure on it. The theory of o-minimal structures is a wide ranging generalization of the older subject of semialgebraic sets, and much of the tame properties of semialgebraic sets go through for the sets belonging to an o-minimal structure, as we shall see. The significance of o-minimality for applications is largely
due to the fact that there are interesting o-minimal structures on $\R$ beyond its
structure of semialgebraic sets. The important examples below are by way of illustration; the general facts about o-minimal structures that we focus on in this part of the appendix do not depend on the nontrivial theorems that establish the o-minimality of these examples.  

\medskip\noindent
Terminology: an {\em o-minimal field\/} is an ordered field equipped with an o-minimal structure on it (and this ordered field is then real closed). We let $\S_{\text{alg}}$ be the o-minimal structure on $\R$ consisting of the semialgebraic subsets of $\R^n$, for all $n$. The first examples of o-minimal fields beyond the semialgebraic case are: 

\medskip\noindent
(i) $\R_{\text{an}}$: this is $\R$ equipped with the smallest structure $\S_{\text{an}}$ on it that contains every $f: [-1,1]^n\to \R$ that extends to a real analytic function $U\to \R$ on some open neighborhood $U\subseteq \R^n$ of $[-1,1]^n$, for $n=0,1,2,\dots$. A set $X\subseteq \R^n$ belongs to $\S_{\text{an}}$ iff $X$ is subanalytic in the larger (compact) real analytic manifold $\mathbb{P}(\R)^n$, where $\mathbb{P}(\R)=\R\cup \{\infty\}$
is the real projective line. The study of $\R_{\text{an}}$ is essentially the theory of subanalytic sets due to Hironaka and Gabrielov: see \cite{B-M, R_an}. 

\medskip\noindent
(ii) $\R_{\text{exp}}$: this is $\R$ with the smallest structure $\S_{\exp}$ on it containing $\{r\}$ for all $r\in \R$, and the function 
$\exp: \R\to \R$, $\exp(r):=\ex^r$.
A set $X\subseteq\R^m$ belongs to $\S_{\exp}$ iff 
$X = \pi\big(\{a\in \R^n:\ P(a, \ex^a)=0\}\big)$ for some $n\ge m$ and some polynomial $P\in \R[x_1,\dots, x_n, y_1,\dots, y_n]$, where $\ex^a := (\ex^{a_1},\dots, \ex^{a_n})$ and $\pi: \R^n\to \R^m$ is given by $\pi(a)=(a_1,\dots, a_m)$  for $a=(a_1,\dots, a_n)\in \R^n$. This characterization of
$\S_{\exp}$ is part of  Wilkie's theorem in~\cite{W1}.

\medskip\noindent
(iii) For applications in arithmetic algebraic geometry it is important that we can amalgamate (i) and (ii) into an o-minimal field $\R_{\text{an}, \exp}$: this is $\R$ with the smallest structure
$\S_{\text{an}, \exp}$ on it such that $\S_{\text{an},\exp}\supseteq \S_{\text{an}}, \S_{\exp}$. A characterization of $\S_{\text{an}, \exp}$ in the style of (ii) is in \cite{vdD-M}, and a sharper one in \cite{DMM} where also the description of $\S_{\text{an}}$ in (i) is improved.

\medskip\noindent
In general, amalgamation as in Example (iii)  does not preserve o-minimality: \cite{RSW} describes two o-minimal structures $\S_1$ and $\S_2$ on $\R$ for which the smallest structure
$\S$ on $\R$ with $\S\supseteq \S_1, \S_2$ is not o-minimal. 

As to the appearance of exponentiation in the examples above, \cite{M} proves a striking dichotomy: for any o-minimal structure $\S$ on $\R$, either
$\exp$ belongs to $\S$, or every function $\R\to \R$ belonging to $\S$ is polynomially bounded, as $t\to \infty$. It is not known if there exists an o-minimal structures $\S$ on $\R$ with a function $\R\to \R$ belonging to it that grows faster, as $t\to \infty$, than any finite iterate of the exponential function.

\medskip\noindent
One way that o-minimality can fail (very badly) for a structure
$\S$ on $\R$ with $\{a\}\in \S_1$ for all $a\in \R$ is that $\Z\in \S_1$: one can show that then all closed subsets of all $\R^n$ belong to $\S$, and even
the Lebesgue-measurability of certain sets in $\S$ cannot be settled without unorthodox set-theoretic axioms.
In particular, the sine function on $\R$ cannot belong to any o-minimal structure on $\R$, although its restriction to any bounded interval belongs to the o-minimal structure $\S_{\text{an}}$ on $\R$.

\subsection*{Definable Sets}
In the rest of part A of the appendix we fix an o-minimal field $R$. Its underlying real closed ordered field is also denoted by $R$. For a set $X\subseteq R^m$ we call $X$ {\em definable\/} if $X$ belongs to the given o-minimal structure of $R$, and likewise for maps $X\to R^n$. (This use of the term ``definable''
has its origin in logic, for which see part B of this appendix.) 
In case the given o-minimal structure on $R$ consists just of the semialgebraic sets (in the sense of the real closed field $R$), we write {\em semialgebraic} in place of {\em definable}.

Topological notions like openness and continuity are with respect to the order topology on $R$ and the corresponding product topology on each $R^n$.
If $X\subseteq R^n$ is definable, then so are its closure $\cl(X)$ and its interior $\text{int}(X)$ in $\R^n$. The definable homeomorphism  
$t\mapsto \frac{t}{1+|t|}\ :\  R\to (-1,1)$
extends to an order preserving
bijection $R_{\infty} \to [-1,1]$ sending $-\infty$ to $-1$ and $\infty$ to $1$, and we equip $R_{\infty}$ with the
(hausdorff) topology on it making this bijection into a homeomorphism. 

Till further notice the results below are from \cite[Chapter 3]{vdD2}, where the o-minimal structures considered are more general, with just an underlying nonempty totally ordered set without least or greatest element and such that for any two distinct elements $a < b$ there is an $x$ with $a<x<b$, no field operations being included.  

Here is the key fact about univariate definable functions:

\begin{theorem}[Monotonicity Theorem]\label{mono} Let $I=(a,b)$ be an interval and let $f:(a,b)\to R$ be definable. Then $f$ has the following properties:
\begin{enumerate}
\item[\rm(i)] there are points $a=a_0< a_1<\cdots<a_n< a_{n+1}=b$ such that on each subinterval $(a_j, a_{j+1})$ with $0\le j \le n$ the function $f$ is continuous, and either strictly decreasing, or constant, or strictly increasing.
\item[\rm(ii)] if $f$ is continuous and $f(p) < c < f(q)$ with $p<q$ in $I$, then
$c=f(x)$ for some $x\in (p,q)$. $($Intermediate Value Property.$)$ 
\item[\rm(iii)] $\lim_{t\downarrow a} f(t)$ and $\lim_{t\uparrow b}f(t)$ exists in
$R_{\infty}$. 
\end{enumerate} 
\end{theorem}

\noindent
Of course, the intermediate value property (ii) is automatic when the underlying
ordered field is $\R$ and then requires no definability assumption. In the o-minimal setting, and certainly outside the familiar real environment, we confine attention to definable objects. For example, the correct analogue of ``connected'' is as follows: a definable set $X\subseteq R^m$ is said to be {\em definably connected\/} if there are no disjoint nonempty definable open subsets $X_0, X_1$ of $X$ with $X=X_0\cup X_1$. For such $X$ and any definable continuous map $f: X\to R^n$, the image $f(X)\subseteq R^n$ is also definably connected. 
Intervals are definably connected.  

\subsection*{Cells} Towards partitioning an arbitrary definable
set $X\subseteq R^n$ into finitely many nice pieces we introduce {\em cells}. These are definably connected sets of a form that makes them suited to proofs by induction (on $n$, for cells in $R^n$). First some notation. Let $X\subseteq R^n$ be definable. Set 
\begin{align*} C(X)\ &:=\ \{f:X\to R:\ f \text{ is definable and continuous} \},
\\
 C_{\infty}(X)\ &:=\ C(X)\cup \{-\infty,\infty\},
\end{align*}
where $-\infty$ and $\infty$ are viewed as constant functions on $X$. Let $f,g\in C_{\infty}(X)$, and suppose $f<g$, that is, $f(x)<g(x)$ for all $x\in X$.
Then we set
$$(f,g)\ =\ (f,g)_X\ :=\ \{(x,r)\in X\times R:\ f(x)<r<g(x)\},$$
so $(f,g)\subseteq R^{n+1}$ is definable; see next  picture.  

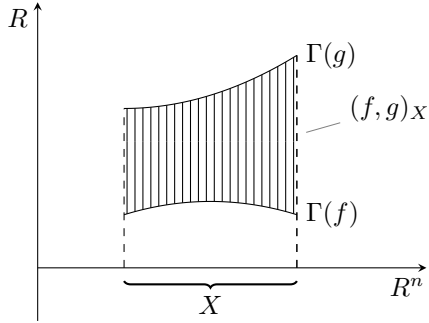
\begin{figure}[h]
\centering
\begin{tikzpicture}

\begin{axis}[thick,smooth,no markers, ticks=none,
    axis x line = center,
    axis y line= left,
    xmin=-0.5,xmax=1.75,
    ymin=-1, ymax=5,
   every axis x label/.style={at={(ticklabel* cs:0.95)},,anchor=north},
    every axis y label/.style={
    at={(ticklabel* cs:0.95)},
    anchor=east},
    xlabel= $R^n$,
    ylabel= $R$
    ]
        \addplot+[name path=A,domain=0:1,black] {x^2+3} node[right,pos=1] {$\Gamma(g)$};
        \addplot+[name path=B,domain=0:1,black] {x-x^2+1} node[right,pos=1] {$\Gamma(f)$};
        \addplot +[dashed, black] coordinates {(1, 0) (1, 4)};
        \addplot +[dashed, black] coordinates {(0,0) (0, 3)};
        
        \addplot +[dashed, black] coordinates {(1, 0) (1, 4)};
        
        \addplot[gray, dashed, pattern= vertical lines] fill between[of=A and B];
        \addplot[mark=none] coordinates {(1,2.5)} node[pin=5:{$(f,g)_X$}]{} ;
        
        \draw [thick,decoration={brace,mirror,raise=5pt},decorate] (axis cs:0,0) --
    node[below=7pt] {$X$} 
  (axis cs:1,0);
        
        
    \end{axis}
    
\end{tikzpicture}
\end{figure}

\medskip
\noindent
Let $n\ge 1$ and $(i_1,\ldots, i_n)$ a sequence of zeros and ones. An $(i_1,\ldots, i_n)$-cell is a definable subset of $R^n$ obtained via the following recursion:

\begin{itemize}
    \item[(i)] Case $n=1$:  a (0)-cell is a one-element subset of $R$, a (1)-cell is an interval;
    \item[(ii)]  An $(i_1,\ldots, i_n,0)$-cell is the graph $\Gamma(f)$ of a function $f\in C(X)$ on an $(i_1,\dots, i_n)$-cell $X$; an $(i_1,\ldots, i_n,1)$-cell is a set $(f,g)_X$ with $f,g\in C_{\infty}(X)$, $f<g$, and $X$ an $(i_1,\dots, i_n)$-cell.
\end{itemize}

\noindent
A {\em cell in $R^n$} is an $\i$-cell, for some (necessarily unique) $\i=(i_1,\ldots, i_n)\in \{0,1\}^n$. A cell in $R^n$ is definably connected, and locally closed (open in its closure in $R^n$). A cell in $R^n$ is open in $R^n$ (and called
an {\em open cell}) iff it is a $(1,\ldots,1)$-cell. 

Important in sections 2 and 8 is that every cell is homeomorphic under a coordinate projection to an open cell. In detail, let $C\subseteq R^n$ be an $\i$-cell, $\i=(i_1,\dots, i_n)$. Let
 $\lambda(1)<\cdots<\lambda(k)$ be the indices $\lambda\in \{1,\ldots, n\}$ with $i_\lambda=1$, and consider the (definable) coordinate projection $p_{\i}:R^n\to R^k$ given by 
$$p_{\i}(x_1,\ldots,x_n)=(x_{\lambda(1)},\ldots,x_{\lambda(k)}).$$
Then $p_{\i}$ maps $C$ homeomorphically onto an open cell in $R^k$. We denote this open cell  $p_{\i}(C)$ also by $p(C)$ and the homeomorphism $p_{\i}|_{C}: C \to p(C)$ by $p_{C}$.

\subsection*{Cell Decomposition} Let $n\ge 1$. A \textit{decomposition of $R^n$} is a partition of $R^n$ into finitely many cells, obtained by the following recursion: 
\begin{itemize}
    \item[(i)] case $n=1$: points $a_1<\cdots<a_m$ in $R$ determine a decomposition of $R=R^1$ consisting of
   $(-\infty,a_1),\ \{a_1\},\ (a_1,a_2),\ldots, (a_{m-1}, a_m),\ \{a_m\},\ (a_m,\infty)$.
    
    \item[(ii)]  a decomposition $\mathcal{D}$ of $R^{n+1}$ is a finite partition of $R^{n+1}$ into cells such that $\pi(\mathcal{D}):=\{\pi(C):\ C\in \mathcal{D}\}$ is a decomposition of $R^n$, where $\pi:R^{n+1}\to R^n$ is the projection map given by $\pi(x_1,\dots,x_n, x_{n+1})=(x_1,\dots, x_n)$.
\end{itemize}

\noindent
With $X(1),\ldots, X(k)$ the distinct cells of a decomposition $\mathcal{D}$ of $R^n$, let  functions $f_{i1}< \cdots <f_{im_i}$  in $C(X_i)$  be given for $i=1,\ldots,k$. Then 
$$\mathcal{D}_i=\{ (-\infty,f_{i1}), \Gamma(f_{i1}), (f_{i1},f_{i2}),\ldots, \Gamma(f_{im_i}),(f_{im_i},\infty) \}$$
is a partition of $X(i)\times R$, and $\mathcal{D}^*=\mathcal{D}_1\cup\cdots \cup \mathcal{D}_k$ is a decomposition of $R^{n+1}$ with $\mathcal{D}=\pi(\mathcal{D}^*)$. See the figure below. Every decomposition of $R^{n+1}$ is obtained in this manner from a decomposition of $R^n$.

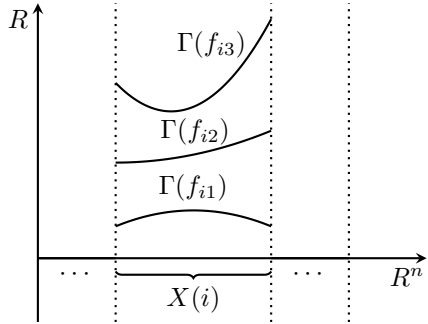
\begin{figure}[h]
\centering
\begin{tikzpicture}

\begin{axis}[thick,smooth,no markers, ticks=none,
    axis x line = center,
    axis y line =left,
    xmin=-0.5,xmax=2,
    ymin=-2, ymax=8,
   every axis x label/.style={at={(ticklabel* cs:0.95)},,anchor=north},
    every axis y label/.style={
    at={(ticklabel* cs:0.95)},
    anchor=east},
    xlabel= $R^n$,
    ylabel= $R$
    ]
        \addplot+[name path=A,domain=0:1,black] {x^2+3} node[above,pos=0.4] {$\Gamma(f_{i2})$};
        \addplot+[name path=B,domain=0:1,black] {2*x-2*x^2+1} node[above,pos=0.5] {$\Gamma(f_{i1})$};
        \addplot+[name path=C,domain=0:1,black] {7*x^2-5*x+5.5} node[left,pos=0.8] {$\Gamma(f_{i3})$};
        \addplot +[dotted,black] coordinates {(1, -2) (1, 8)};
        \addplot +[dotted,black] coordinates {(0,-2) (0, 8)};
        \addplot +[dotted,black] coordinates {(1.5,-2) (1.5, 8)};
        
        \addplot[mark=none] coordinates {(-0.5,0) (0,0)} node[below,pos=0.5] {$\cdots$};
        \addplot[mark=none] coordinates {(1,0) (1.5,0)} node[below,pos=0.5] {$\cdots$};

  \draw [thick,decoration={brace,mirror,raise=5pt},decorate] (axis cs:0,0) --
    node[below=7pt] {$X(i)$} 
  (axis cs:1,0);

    \end{axis}
    
\end{tikzpicture}
\end{figure}

\noindent
In these definitions of {\em cell\/} and {\em decomposition\/} we assumed 
$n\ge 1$, but it is convenient to also consider the one-point set $R^0$ as the unique cell in $R^0$, namely as an $\i$-cell where $\i\in \{0,1\}^0$ is the empty tuple of zeros and ones, and $\{R^0\}$ as the unique decomposition of $R^0$. 
In this way clause (i) in these definitions appears as the case $n=0$ of the corresponding clause (ii). So below we allow $n=0$.

A decomposition $\mathcal{D}$ of $R^n$  is said to \textit{partition} a set $X\subseteq R^n$ if each cell in $\mathcal{D}$ is  either contained in $X$ or disjoint from $X$ (so $X$ is a union of cells in $\mathcal{D}$). We can now state the fundamental Cell Decomposition Theorem:

\begin{theorem}\label{celldec} For any definable $X_1,\ldots,X_m\subseteq R^n$ some decomposition of $R^n$ partitions $X_1,\ldots, X_m$. If $X\subseteq R^n$ and $f: X\to R$ are definable, then some decomposition $\mathcal{D}$ of $R^n$ partitions $X$ with continuous $f|_C$ for all cells $C\subseteq X$ in $\mathcal{D}$.  
\end{theorem}

\noindent
Some  consequences: if the definable set $X\subseteq R^n$ is definably connected, then it is ``definably path connected'': for any points $p,q\in X$ there is a definable continuous $\gamma: [0,1]\to X$ with $\gamma(0)=p, \gamma(1)=q$. If the underlying ordered field of $R$ is
$\R$, then for definable $X\subseteq R^n$, {\em definably connected\/} agrees with {\em connected}. 

A {\em definably connected component\/} of a definable set $X\subseteq R^n$ is a definably connected definable nonempty subset of $X$ that is maximal with respect to inclusion. (So if $X=\emptyset$, it has no definably connected components.)

\begin{cor}\label{dcoc} For definable $X\subseteq R^n$, the definably connected components of $X$ are all open and closed in $X$, and form a finite partition of $X$.
\end{cor} 

\subsection*{Definable families} Let $E\subseteq R^m$ and $X\subseteq E\times R^n\subseteq R^{m+n}$ be definable. For  $a\in E$ we set
$$X(a)\ :=\ \{x\in R^n:\ (a,x)\in X\}.$$

\noindent
We view $X$ as describing the family $\big(X(a)\big)_{a\in E}$ of definable subsets of $R^n$. We call this a \textit{definable family}, and the sections $X(a)$ are the members of the family.

\begin{example}
The hypersurfaces in $R^2$ of degree at most 2 are the members of a semialgebraic family: such a hypersurface is the set of solutions in $R^2$ of an equation
$$a_1x^2+a_2xy+a_3y^2+a_4x+a_5y+a_6\ =\ 0 \quad \text{ with }(a_1,a_2,a_2,a_4,a_5,a_6)\in \R^6\setminus \{0\},$$ so here $E=R^6\setminus \{0\}$ and $X$ consists of the points
$(a_1,a_2,a_2,a_4,a_5,a_6,x,y)\in E\times R^2$ satisfying the above equation.  
By the same token, for any $n$ and $d$ the hypersurfaces in $\R^n$ of degree $\le d$ are the members of a semialgebraic family. 
\end{example}

\noindent
Let $\pi:R^{m+n}\to R^m$ be given by $\pi(x_1,\dots, x_{m+n})=(x_1,\dots, x_m)$.  

\begin{prop}\label{fdc} Suppose $\mathcal{D}$ is a decomposition of $R^{m+n}$ partitioning $X$. Then for each $a\in E$, 
$$\mathcal{D}(a):=\{C(a):C\in \mathcal{D}, a\in \pi(C)\}$$
is a decomposition of $R^n$ partitioning $X(a)$. This gives in particular a
finite bound on the number of definably connected components of $X(a)$ independent of $a\in E$. 
\end{prop}

\subsection*{Dimension} This subsection is taken from \cite[Chapter 4, section 1]{vdD2}. It is natural to assign to an $(i_1,\dots, i_n)$-cell $C$ the dimension $$\dim C\ :=\ i_1+\cdots +i_n\in \{0,\dots,n\},$$ since 
$p_{(i_1,\dots, i_n)}: R^n \to R^{i_1+\cdots + i_n}$ is definable and maps $C$ homeomorphically onto an open subset of $R^{i_1+\cdots + i_n}$.
Such $C$ does not contain any $(j_1,\dots, j_n)$-cell with
$j_1+\cdots + j_n>\dim(C)$. This fact allows us to extend the above dimension to 
arbitrary nonempty definable $X\subseteq R^n$ by
$$\dim X\ :=\ \max\{\dim C:\ C\subseteq X \text{ is a cell}\}\in \N.$$ 
We also set $\dim \emptyset := -\infty$. Here are some basic facts on dimension:

\begin{prop} Let $X\subseteq R^m$ be definable. Then: \begin{enumerate}
\item[\rm(i)] $\dim X = 0\ \Longleftrightarrow\ X \text{ is finite and nonempty}$;
\item[\rm(ii)] $\dim X =m\ \Longleftrightarrow\ X \text{ has nonempty interior in } R^m$;
\item[\rm(iii)] if $Y\subseteq R^m$ is definable, then $\dim X\cup Y = \max(\dim X, \dim Y)$;
\item[\rm(iv)] if $Y\subseteq R^n$ is definable, then $\dim X\times Y=\dim X + \dim Y$;
\item[\rm(v)] if $f: X \to R^n$ is definable, then $\dim X\ge \dim f(X)$;
\item [\rm(vi)]if $f: X \to R^n$ is definable and injective, then $\dim X = \dim f(X)$;
\item[\rm(vii)] if $X\ne \emptyset$, then $\dim (\cl(X)\setminus X) < \dim X$. 
\end{enumerate} 
\end{prop}

\noindent
In (v), (vi) we do not assume $f$ is continuous. Here is a stronger version of (v): 

\begin{prop} Let $f: X \to R^n$ be definable, $X\subseteq R^m$. For $d\le m$, set $$Y(d)\ :=\ \{y\in R^n: \dim f^{-1}(y) = d\}.$$ 
Then $Y(d)$ is definable and $\dim X = \max_{d\le m} d+\dim Y(d)$.
\end{prop}

\noindent
We also have a {\em local\/} dimension: Let $X\subseteq R^m$ be definable and $a\in R^m$.
Then there is a definable neighborhood $V$ of $a$ in $R^m$ such that 
$\dim(X\cap U)=\dim(X\cap V)$ for all definable neighborhoods $U\subseteq V$ 
of $a$ in $R^m$; thus $\dim(X\cap V)$ is independent of the choice of such $V$, and we set $\dim_a X := \dim(X\cap V)$ for such $V$.

\subsection*{Definable Compactness} This subsection and the next are from \cite[Chapter 6, section 1]{vdD2}.  The ordinary notion of compactness from point set topology is useless in our setting, but we do have a good substitute. Call a set $X\subseteq R^m$ {\em bounded\/} if 
$X\subseteq [-r,r]^m$ for some $r\in R^{>}$.

\begin{prop}\label{dc1} If $f: X\to R^n$ is a continuous definable map on a closed and bounded $($definable$)$ set $X\subseteq R^m$, then $f(X)\subseteq R^n$ is also closed and bounded. 
\end{prop}

\noindent
This has the expected consequences:

\begin{cor}\label{dc2} If $f: X \to R$ is a continuous definable function on a nonempty closed bounded set $X\subseteq R^m$, then $f$ has a maximum and a minimum value
on $X$.
\end{cor}

\begin{cor} If $f: X\to R^n$ is an injective continuous definable map on a closed and bounded set $X\subseteq R^m$, then $f: X \to f(X)$ is a homeomorphism.
\end{cor}

\subsection*{Definable Selection} For any interval $(a,b)$ we can ``definably''
select a point in it:  $(a+b)/2$ if $a, b\in R$; $b-1$ if $a=-\infty$ and $b\in R$; $a+1$ if $a\in R$ and $b=\infty$; $0$ if $a=-\infty$ and $b=\infty$. This can be exploited to give two very useful selection principles, the second a consequence of the first:

\begin{prop}\label{defsel} Any definable equivalence relation
on a definable set $X\subseteq R^n$ has a definable set of representatives, that is, a definable subset of $X$ that has exactly one point in common with each equivalence class. Any definable map
$f: X \to R^n$, $X\subseteq R^m$, has a definable right-inverse $g: f(X)\to X$, that is, $f\circ g=\operatorname{id}_{f(X)}$.
\end{prop}

\noindent
In these last two subsections we got to use the underlying additive group of $R$, but not yet its multiplication. Accordingly this material goes through in the more general o-minimal setting of \cite[Chapter 6]{vdD2} (not needed for our purpose). We now turn to a topic where multiplication does come into play. 

\medskip\noindent
The rest of part A is from \cite[Chapter 7]{vdD2}. 

\subsection*{Differentiability} In this subsection we don't need o-minimality or definability, and $R$ can be any ordered field. The elementary facts stated here have the same proofs as for $R=\R$. 
For $a,b\in R^n$ we set $a\cdot b:= a_1b_1+ \cdots + a_nb_n\in R$ (dot product).  
Let $I\subseteq R$ be open. A map $f: I \to R^n$ is said to be differentiable at a point $a\in I$ with derivative $b\in R^n$ if
$$\lim_{t\to 0}\frac{1}{t}\big(f(a+t)-f(a)\big)\ =\ b.$$
In that case $f$ is continuous at $a$ and we set $f'(a):= b$. If $f,g: I \to R^n$
are differentiable at $a$, then so are $f+g: I \to R^n$ and $f\cdot g: I \to R=R^1$, with 
$$(f+g)'(a)\ =\ f'(a)+g'(a),\quad (f\cdot g)'(a)\ =\ f'(a)\cdot g(a) +f(a)\cdot g'(a),$$
and if in addition $n=1$, $g$ is continuous, and $g(a)\ne 0$, then $f/g: I\setminus g^{-1}(0)\to R$
is differentiable at $a$ with $(f/g)'(a)=\big(f'(a)g(a)-f(a)g'(a)\big)/g(a)^2$. 
Constant maps $I\to R^n$ are differentiable at every $a\in I$ with derivative $0\in R^n$, and the inclusion map $I\to R$ is differentiable at every $a\in I$ with derivative $1\in R$. 

{\em Chain Rule}: if $f: I \to R$ is continuous,  differentiable at $a\in I$, and $f(a)\in J$ with open $J\subseteq R$, and $g: J \to R$ is differentiable at
$f(a)$, then $g\circ f: I\cap f^{-1}(J)\to R$ is differentiable at $a$ with
$(g\circ f)'(a)=g'(f(a))\cdot f'(a)$.

\medskip\noindent
Next we consider directional derivatives. We consider a map
$f: U \to R^n$ with open $U\subseteq R^m$. For a point $a\in U$ and a vector $v\in R^m$ we say that {\em $f$ is differentiable at $a$ in the $v$-direction\/}  if the $R^n$-valued map $t\mapsto f(a+tv)$ (defined on an open neighborhood of $0\in R$) is differentiable at $0$, that is,
$\lim_{t\to 0}\frac{1}{t}\big(f(a+tv)-f(a)\big)$ exists in $R^n$, 
in which case we set $$d_af(v)\ :=\ \lim_{t\to 0}\frac{1}{t}\big(f(a+tv)-f(a)\big)\in R^n.$$ 
For the standard basis vectors $e_1,\dots,e_m$ of the $R$-linear space $R^m$ we also write $\frac{\partial f}{\partial x_i}(a)$ for
$d_a f(e_i)$.  

\medskip\noindent
Let $a\in U$ and let $T: R^m \to R^n$ be an $R$-linear map. We call $f$ {\em differentiable at $a$ with differential $T$} if for every $\epsilon\in R^{>}$ we have, for all sufficiently small $v\in R^m$,
$$|f(a+v)-f(a) -T(v)|\ \le\ \epsilon|v|.$$
Then $f$ is continuous at $a$ and $T$ is uniquely determined by $f,a$, so we
can set
$d_af:= T$, a notation consistent with that for directional derivatives:  for each vector $v\in R^m$ the map $f$ is differentiable at $a$ in the $v$-direction
with $d_af(v)=T(v)$ where $d_af(v)$ denotes the directional derivative
defined earlier. For $m=1$ this notion of differentiability at $a$ agrees with
the one defined earlier, with $d_af(1)=f'(a)$. 

The map $f=(f_1,\dots, f_n): U \to R^n$ is differentiable at $a$ iff $f_1,\dots, f_n: U \to R$ are differentiable at $a$. 
In that case all partials $\big(\partial f_i/\partial x_j\big)(a)$ exist and the $n\times m$ matrix $\big(\partial f_i/\partial x_j\big)(a)$ is the matrix of
$d_af$ with respect to the standard basis vectors of $R^m$ and $R^n$. Each $R$-linear map $R^m\to R^n$ is differentiable at each point of $R^m$ with itself
as differential. If the maps $f, g: U\to R^n$ are differentiable at $a\in U$, then
$f+g$ and $cf$ for $c\in R$ are differentiable at $a$ with
 $$d_a(f+g)\ =\ d_af+d_ag,\qquad d_acf\ =\ c\cdot d_a f.$$ 
{\em Chain Rule}: Suppose $U\subseteq R^m,\ V\subseteq R^n$ are open, $a\in U$, 
$f: U \to R^n$ is continuous, $f$ is differentiable at the point $a\in U$,
$f(a)\in V$, and $g: V \to R^l$  is differentiable at $f(a)$. Then
$g\circ f: U\cap f^{-1}(V)\to R^l$ is differentiable at $a$, and
$$d_a(g\circ f)\ =\ \big(d_{f(a)} g\big)\circ d_a f.$$

\subsection*{Preserving Definability and the Mean Value Theorem} We now revert to the setting where $R$ is an o-minimal field and our sets and maps are definable. So let $U\subseteq R^m$ be
open and definable, and let $f: U \to R^n$ be definable. Then
the set $D_f$ of points $(a,v)$ in $U\times R^m\subseteq R^{2m}$ such that $f$ is differentiable at
$a$ in the $v$-direction is definable, and so is the map
$(a,v)\mapsto d_af(v): D_f\to R^n$.

\begin{lemma}\label{mvt1} Let $a<b$ in $R$, and suppose $f:[a,b]\to R$ is definable and continuous, and differentiable at each point of $(a,b)$. Then there is $c\in (a,b)$ with 
$$f(b)-f(a)=f'(c)\cdot(b-a).$$
\end{lemma}

\noindent
This ``mean value'' lemma opens the door to {\em continuous differentiability}. Let
$$f=(f_1,\dots, f_n): U\to R^n$$
be a definable map on a (definable) open set $U\subseteq R^m$. 
We say that  $f$ is {\em of class $C^1$} (or just $C^1$) if
$f$ is differentiable at every point $a\in U$ in the directions $e_1,\dots, e_m$, and the resulting (definable) functions $\frac{\partial f}{\partial x_1},\dots, \frac{\partial f}{\partial x_m}: U \to R^n$ are continuous. In the next lemma we take matrices with respect to the standard bases of $R^m$ and $R^n$.

\begin{lemma} If $f$ is $C^1$, then $f$ is differentiable at each point of $U$ and
$$a\mapsto \text{matrix of }d_af\ :\  U \to R^{n\times m}$$
is continuous. Conversely, if $f$ is differentiable at each point of $U$ and the
above matrix-valued map is continuous, then $f$ is a $C^1$-map. 
\end{lemma} 

\noindent
For an $R$-linear map $T: R^m\to R^n$ we put
$|T| := \max\{|Ta|:\ |a|\le 1,\ a\in R^m\}$.
(The maximum exists in $R$ since $T$ is definable and continuous.) Thus $|Ta|\le |T|\cdot |a|$ for all $a\in R^m$. Now we can state an extended mean value result:

\begin{lemma}\label{mvt2} 
Suppose $f:U\to R^n$ is of class $C^1$, and $a,b\in U$ are such that the line segment $[a,b]:=\{(1-t)a +tb:\ 0\leq t \leq 1\}$ is contained in $U$. Then

$$|f(b)-f(a)|\leq |b-a|\cdot \max_{y\in [a,b]} |d_yf|.$$
\end{lemma}

\subsection*{Smooth Cell Decomposition} It is convenient to extend the notions of 
$C^1$-map and $C^1$-cell.  We say that a definable map $f:X\to R^n$, $X\subseteq R^m$, is  {\em $C^1$} if there are a definable open $U\subseteq R^m$ such that $X\subseteq U$, and a definable $C^1$-map $F:U\to R^n$ such that $f=F|_X$. We define {\em $C^1$-cells} as in the recursive definition for {\em  cells}, except that  we require the (definable) functions $f$ and $g$ there, when $R$-valued, to be $C^1$ instead of just being continuous.

Every inclusion map $X\to R^m$ for definable $X\subseteq R^m$ is $C^1$.
If the definable map $f: X \to R^n$ with $X\subseteq R^m$ is $C^1$ and the 
definable map $g: Y \to R^l$ with $Y\subseteq R^n$  is $C^1$, then
$g\circ f: f^{-1}(Y)\to R^l$ is $C^1$. For definable $f=(f_1,\dots, f_n): X\to R^n$ with $X\subseteq R^m$, $f$ is $C^1$ iff $f_1,\dots, f_n$ are $C^1$.

The following is a $C^1$-version of cell decomposition.

\begin{theorem}\label{C1celldec} For any definable $X_1,\ldots,X_m\subseteq R^n$ there is a decomposition of $R^n$ into $C^1$-cells partitioning $X_1,\ldots, X_m$. If $X\subseteq R^n$ and $f: X\to R$ are definable, then there is a decomposition of $R^n$ into $C^1$-cells which partitions $X$ such that $f|_C$ is $C^1$ for each cell $C\subseteq X$ of the decomposition.  
\end{theorem}

\subsection*{$C^k$-maps} Let $k\ge 1$ below.  Let $f=(f_1,\dots, f_n): U \to R^n$
be definable, with (definable) open $U\subseteq R^m$. By recursion on $k$ we specify what it means for $f$ to be a $C^k$-map. For $k=1$ this has been defined earlier, and for $k>1$ we say that $f$ is {\em $C^k$} if $f$ is $C^1$ and every partial $\partial f_i/ \partial x_j: U\to R$ ($1\le i\le n,\ 1\le j\le m$) is $C^{k-1}$.

We also extend being $C^k$ to definable maps $f: X \to R^n$ with $X\subseteq R^m$ not necessarily open, just as we did for $k=1$, and likewise we define the
notion of a $C^k$-cell just like we did for $k=1$. The remarks we made above about this extended notion of definable $C^1$-map go through when ``$C^1$''
is replaced by ``$C^k$'' everywhere. 
The $C^1$-cell decomposition theorem goes through with ``$C^1$'' replaced by
``$C^k$'' everywhere.  

\subsection*{Semialgebraic Functions} We finish part A of this appendix with some facts on real semialgebraic functions, and use this to determine the algebraic part of the set 
$$X\ =\ \{(a,b,c)\in \R^3:\ 1 < a,b< 2,\ c=a^b\}\ \subseteq\ \R^3$$
of the Example in the Introduction to this paper. 

Let $f: I \to \R$ be a continuous function on an interval $I\subseteq \R$. Then $f$  is semialgebraic (that is, its graph as a subset of $\R^2$ is semialgebraic)
iff for some nonzero polynomial $P\in \R[x,y]$ we have $P\big(t,f(t)\big)=0$ for all $t\in I$. (See \cite[Chapter 2]{vdD2} for a detailed treatment of real semialgebraic sets and functions, in particular Exercise 3 in
 (3.7) there, with Hint on p. 169.) Thus if $f$ is real analytic, and its restriction to some subinterval of $I$ is semialgebraic, then $f$ is semialgebraic. 
 
 Suppose $f$ is semialgebraic and $I=(0,b)$ with $b\in (0,\infty]$. Then either $f(t)=0$ for all sufficiently small $t>0$, or for some
 $q\in \Q$ and $c\in \R^{\times}$ we have $f(t)/t^q\to c$ as $t\downarrow 0$. 
 (See for example the end of \cite{R_an}.) Combining these facts we see that the following real analytic functions on $(0,\infty)$ cannot be semialgebraic on any
 subinterval of $(0,\infty)$: 
 for $r\in \R\setminus \Q$ the function $t\mapsto t^r$; for
 $a\in (1,\infty)$ the function $t\mapsto a^t$; 
 the function $t\mapsto \log t$.  

\medskip\noindent
By semialgebraic cell decomposition the
algebraic part $Y^{\alg}$ of any set $Y\subseteq \R^n$ is the union of the sets $\cl(C)\cap Y$ with $C\subseteq Y$ a 1-dimensional semialgebraic cell. 
We can now prove the fact stated in the Introduction for the above $X\subseteq \R^3$ that 
$$X^{\alg}\ :=\ \bigcup_{q\in \Q\cap (1,2)} X_q.$$
The inclusion $\supseteq$ is clear. The sets $X_q$ ($q\in \Q\cap (1,2)$) are closed in $X$, so given any  1-dimensional semialgebraic cell $C\subseteq X$ it suffices to show that $C\subseteq X_q$ for some $q\in \Q\cap (1,2)$. Now $C$ is a $(0,0,1)$-cell, or a $(0,1,0)$-cell, or a $(1,0,0)$-cell. But $X$ does not contain any $(0,0,1)$-cell. Also $C$ cannot be a $(0,1,0)$-cell: if it were, then for a fixed $a\in (1,2)$ and an
interval $I\subseteq (1,2)$ the function $t\mapsto a^t$ on $I$ would be semialgebraic, which is false. Finally, suppose $C$ is a $(1,0,0)$-cell. Then $C=\{\big(t, f(t), t^{f(t)}\big): t\in I\}$
where $f: I \to \R$ is a continuous semialgebraic function on an interval 
$I \subseteq (1,2)$ (and $t\mapsto t^{f(t)}: I \to \R$ is also semialgebraic). 
Consider any interval $J\subseteq I$ on which $f$ is of class $C^1$. Then taking the logarithmic derivative of $t\mapsto t^{f(t)}=\ex^{f(t)\log t}$
on $J$ gives that $f'(t)\log t + f(t)/t$ is semialgebraic as a function of $t\in J$, and so $f'=0$ on $J$. Using several such $J$ we see that $f$ is constant
on $I$. This constant value must be a rational $q\in (1,2)$, so
$C\subseteq X_q$.

\section{Some Model theory}

\noindent
In section ~\ref{par4} we use the notion of an $\aleph_0$-saturated elementary extension, requiring a little excursion into model theory. We shall give precise definitions of the necessary model-theoretic notions, motivating them by examples, and stating carefully a few needed results. For most proofs we refer to \cite[Appendix B]{ADH}; there the basics of model theory are developed in the setting of many-sorted structures, while here we stay with one-sorted structures (which have only one underlying set, while a many-sorted structure has a {\em family\/} of underlying sets).   

\bigskip\noindent 
A {\em language\/} is a set $L$ whose elements are called symbols, each
symbol $s\in L$ being equipped with a natural number $\text{arity}(s)\in \N$. These symbols
are either {\em relation symbols\/} or {\em function symbols}, and $L$ is the disjoint union of 
$\Lr$, its subset of relation symbols, and $\Lf$, its subset of function symbols. 

Below $L$ is a language. Let $\cM$ be an {\em $L$-structure}, that is, a triple 
$$\cM\ =\ \big(M; (R^{\cM})_{R\in L^{\text{r}}},(F^{\cM})_{F\in L^{\text{f}}}\big)$$ consisting of a {\em nonempty} set $M$, an $m$-ary relation $R^{\cM}\subseteq M^m$ on $M$ for each $R\in \Lr$ of arity $m$, and an $n$-ary function $F^{\cM}: M^n \to M$ on $M$ for each $F\in \Lf$ of arity $n$. We call $M$ the {\em underlying set\/} of $\cM$, we think of a symbol $R\in\Lr$ as {\em naming\/} the corresponding relation $R^{\cM}$ on $M$, and likewise for $F\in \Lf$.

Thus a nullary function symbol (also called a {\em constant symbol}) names a function $M^0 \to M$, to be identified with its unique value in $M$, so a constant symbol names a distinguished element of $M$. Usually we drop the superscripts $\cM$ in $R^{\cM}$ and $F^{\cM}$ when $\cM$ is understood from the context, the distinction between the symbols and what they name to be kept in mind. We shall also feel free to denote $\cM$ and its underlying set $M$ by the same letter, when convenient.
The reason we need the distinction between symbols and what they name in a particular $L$-structure is that we have to be able to say that a statement expressed in terms of these symbols holds in, say, two {\em different\/} $L$-structures $\cM$ and $\cN$. 

We need to consider two (unrelated) ways of increasing $\cM$.
The {\em first\/} is when $L$ is a sublanguage of $L'$. Then an {\em $L'$-expansion
of $\cM$} is an $L'$-structure $\cM'$ with the same underlying set
as $\cM$ and where the symbols of $L$ name the same relations and functions in $\cM$ as in $\cM'$. We also say that then $\cM'$ {\em expands\/} $\cM$.

\medskip\noindent
{\bf Example.} The language $L_{\text{OF}}$ of ordered fields has a binary relation symbol $<$, constant symbols $0$ and $1$, a unary function symbol $-$, and binary function symbols $+$ and $\cdot$. Any ordered field
$K$ is viewed as an $L_{\text{OF}}$-structure by having $<$ name the (strict) ordering of the field, and the function symbols name the functions on $K$ that these symbols traditionally denote. Thus an ordered field $K$ {\em expands\/}
its underlying field. Equipping the ordered field $\R$ of real numbers with the
exponential function $\exp: \R\to \R$ gives the  expansion $\R_{\exp}$ of $\R$.
This is not exactly how we specified $\R_{\exp}$ in part A of the appendix, but the difference is immaterial: the two descriptions lead to the same sets 
$X\subseteq \R^n$ being definable in $\R_{\exp}$, see below. For {\em model-theoretic\/} use we take $\R_{\exp}$ as the above expansion of $\R$. 

\medskip\noindent
A {\em second\/} way: {\em $\cM$ is a substructure of $\cN$} (or {\em $\cN$ is an extension of $\cM$}); notation: $\cM\subseteq \cN$. This means: $\cN=(N;\dots)$ is an $L$-structure, $M\subseteq N$, $R^{\cN}\cap  M^m=R^{\cM}$ for $m$-ary $R\in \Lr$, and
$F^{\cM}(a)=F^{\cN}(a)$ for $n$-ary $F\in \Lf$ and $a\in M^n$. For example,
if $K_1$ and $K_2$ are ordered fields viewed as $L_{\text{OF}}$-structures,
$K_1\subseteq K_2$ means that $K_1$ is an ordered subfield of $K_2$ (so the ordering of $K_2$ restricts to the ordering of $K_1$).

\medskip\noindent
The $0$-definable (or absolutely definable) sets of $\cM$ are the sets $X\subseteq M^n$ for $n=0,1,2,\dots$ obtained recursively as follows: \begin{enumerate}
\item[(D1)] $R^{\cM}\subseteq M^{m}$ for $m$-ary $R\in \Lr$ and 
$\text{graph}(F^{\cM})\subseteq M^{n+1}$ for $n$-ary $F\in \Lf$ are $0$-definable.
\item[(D2)] if $X, Y\subseteq M^n$ are $0$-definable, then so are $X\cup Y$ and
$M^n\setminus X$;
\item[(D3)] if $X\subseteq M^n$ is $0$-definable, then so are $X\times M, M\times X\subseteq M^{n+1}$;
\item[(D4)] for any $i < j$ in $\{1,\dots,n\}$ the diagonal
$\{(a_1,\dots, a_n):\ a_i=a_j\}\subseteq M^n$ is $0$-definable;
\item[(D5)] if $X\subseteq M^{n+1}$ is $0$-definable, then so is $\pi(X)\subseteq M^n$ where $\pi: M^{n+1}\to M^n$ is given by $\pi(a_1,\dots, a_n, a_{n+1})=(a_1,\dots, a_n)$. 
\end{enumerate}
Thus the $0$-definable sets $X\subseteq M^n$ are exactly those that belong to the smallest structure on $M$ that contains all sets $R^{\cM}$ with $R\in \Lr$ and all sets $\text{graph}(F^{\cM})$ with $F\in \Lf$. 
If $X\subseteq M^n$ is a $0$-definable set of $\cM$ and the ambient $\cM$ is clear from the context, we also just say that $X$ is $0$-definable. A map $f: X \to M^n$ with $X\subseteq M^m$ is said to be $0$-definable (in $\cM$) if
its graph as a subset of $M^{m+n}$ is $0$-definable in $\cM$. In that case its domain $X$ is $0$-definable and for every $0$-definable $X'\subseteq X$ its image
$f(X')\subseteq M^n$ is $0$-definable. 

We need a notation system
to describe $0$-definable sets in a uniform way in varying $L$-structures. Towards this we assume that in addition to the symbols of $L$ we have an infinite set $\Var$ of symbols, called {\em variables} (with $\Var$ disjoint from the language $L$ and independent of $L$). Given a tuple $x=(x_1,\dots x_m)$ of distinct variables we define $L$-terms $t$ in $x$ recursively as follows:
each $x_i$ with $1\le i\le m$ is an $L$-term in $x$, and for $n$-ary
$F\in \Lf$ and $L$-terms  
$t_1, \dots, t_n$ in $x$, the expression
$F(t_1,\dots, t_n)$ is an $L$-term in $x$. Formally, these expressions are words on some alphabet together with the specification of the tuple $x$, but we prefer not to go into detail on such syntactical matters.  We let $t(x)$ indicate a term $t$ in $x$. 
An $L$-term $t=t(x)$ names in an obvious way a function
$t^{\cM}: M^m\to M$, with $x_i$ naming the function $(a_1,\dots, a_m)\mapsto a_i: M^m \to M$. Functions named by $L$-terms are $0$-definable in $\cM$.    

For example, when $K$ is an ordered field, any $L$-term $t$ in $x=(x_1,\dots, x_m)$
names a function $K^m\to K$ given by a polynomial 
in $\Z[x_1,\dots, x_m]$, and each such polynomial function 
is named by an $L$-term in $x$. (Different terms can name the same function: $x+(-x)$ and $0$ are different as terms
in the single variable $x$, but name the same function
$K \to K$, which takes the constant value $0$.)

Going back to our $L$-structure $\cM$ we introduce for every element $a\in M$ a constant symbol $\underline{a}\notin L$ as a name for $a$, with $\underline{a}\ne \underline{b}$ for all $a\ne b$ in $M$. For every set $A\subseteq M$ we extend $L$ to the language $L_A:= L\cup\{\underline{a}:\ a\in A\}$, and expand $\cM$ to the $L_A$-structure $\cM_A$
with $\underline{a}$ naming $a$ for $a\in A$.
A set $X\subseteq M^n$ is said to be $A$-definable (or definable over $A$) in $\cM$ if $X$ is $0$-definable in $\cM_A$. When $A=M$ we just use write ``definable'' instead of
``$M$-definable''. A set $X\subseteq M^n$ is $A$-definable (in $\cM$) iff
$X=Y(a)$ for some $0$-definable set $Y\subseteq M^{m+n}$ and some $a\in A^m$.
(Here for $Y\subseteq M^{m+n}$ and $a\in M^m$ we set
$Y(a):=\{b\in M^n:\ (a,b)\in Y\}$.)  

\medskip\noindent
{\bf Examples}. For an algebraically closed field $\k$, viewed as an $L$-structure
with $L=\{0,1,-, +, \cdot\}$ (the language of rings), the subsets of $\k^n$ definable in $\k$ are exactly the so-called constructible subsets of $\k^n$: the
finite unions of sets $X\setminus Y$ with $X,Y$ Zariski-closed subsets of $\k^n$.
(This is basically the constructibility theorem of Tarski-Chevalley: 
the image of a constructible subset of $\k^{n+1}$ in $\k^n$ under the projection map $(a_1,\dots, a_{n+1})\mapsto (a_1,\dots, a_n): \k^{n+1}\to \k^n$ is constructible in $\k^n$.) The same holds with ``$A$-definable'' and ``$A$-constructible'' instead of ``definable'' and ``constructible'' for any set $A\subseteq \k$, where an $A$-constructible subset of $\k^n$ is a finite
union of sets $X\setminus Y$ with $X, Y\subseteq \k^n$ given by the vanishing of
polynomials in $D[x_1,\dots,x_n]$ where $D$ is the subring of $\k$ generated by $A$.

For us the more relevant example is when $K$ is an ordered real closed field. 
Then the subsets of $K^n$ definable in $K$ are exactly the semialgebraic subsets of $K^n$, that is, the finite unions of sets (with $f,g_1,\dots, g_m$ ranging over $K[x_1,\dots, x_n]$)
$$\{a\in K^n:\ f(a)=0,\ g_1(a) > 0,\dots, g_m(a)>0\}.$$
This is the Tarski-Seidenberg theorem (like 
the Tarski-Chevalley theorem, but with ``semialgebraic'' instead of ``constructible''). Requiring the polynomials $f, g_1,\dots, g_m$ above to have coefficients in $\Z$, we obtain likewise exactly the subsets of $K^n$ that are $0$-definable in $K$. 

\subsection*{Saturation} This notion functions as a kind of compactness for definable sets. Let $\kappa$ be a cardinal. An $L$-structure $\cM$ is said to be {\em $\kappa$-saturated\/} if for every set $A\subseteq M$ of cardinality $<\kappa$ and any family $(X_i)$ of $A$-definable subsets of $M$ with the finite intersection property we have 
$\bigcap_i X_i\ne \emptyset$. (The {\em finite intersection property for $(X_i)$\/} says that $X_{i_1}\cap \cdots \cap X_{i_n}\ne \emptyset$ for all indices $i_1,\dots, i_n$.) This property of families of definable subsets of $M$ is inherited by
families of definable subsets of $M^m$ for any $m$. One can indeed think of this in terms of compactness: if $\cM$ is $\kappa$-saturated, then for $A\subseteq M$
of cardinality $<\kappa$, the $A$-definable subsets of $M^m$ are a basis for a topology on $M^m$, the $A$-topology, which makes $M^m$ a compact hausdorff space in which these $A$-definable sets are exactly the open-and-closed sets. We need this only for $\kappa=\aleph_0$, which means that in the definition above we restrict to {\em finite\/} $A\subseteq M$. For $\kappa=\aleph_1$ the restriction is to {\em countable\/} $A\subseteq M$.  

For example, any algebraically closed field of infinite transcendence degree over its prime field is $\aleph_0$-saturated, and the field $\C$ of complex numbers
is even $\aleph_1$-saturated. The ordered field $\R$ is not even $1$-saturated, since $\bigcap_n (n,\infty)=\emptyset$.  [The referee asked for an explicit example of an $\aleph_0$-saturated elementary extension of the ordered field of real numbers.  An attractive example is the real closed ordered field of {\em surreal numbers of countable length}, which is actually $\aleph_1$-saturated; for surreal numbers,  see Gonshor~\cite{G}. In this connection we mention that a real closed ordered field $R$
is $\aleph_1$-saturated iff for all countable subsets $A < B$ of $R$ (that is $a<b$ for all $a\in A$ and $b\in B$), there is a $c\in R$
such that $A<c<B$.] Any finite structure (a structure with finite underlying set) is $\kappa$-saturated for every $\kappa$.

Towards the study of a structure $\cM$ of interest we can always pass to an $\aleph_1$-saturated extension $\cN$ with the same elementary properties, do our work in $\cN$ and then pass the information gained back to $\cM$. This will be made precise below. To make sense of ``the same elementary properties'' we need a notation system for definable sets. This is the reason for introducing formulas and sentences below. 

\subsection*{Formulas and sentences} Let $y=(y_1,\dots, y_n)$ be a tuple of distinct variables. We define $L$-formulas $\phi$ in $y$ inductively as follows: 

\smallskip
(i) The {\em atomic\/} $L$-formulas in $y$ are the expressions 
$$\top,\quad \bot,\quad R\big(t_1,\dots, t_m\big),\ \text{ and }t_1=t_2$$ 
\qquad\quad for $m$-ary $R\in \Lr$ and $L$-terms $t_1,\dots, t_m$ in $y$, and $L$-terms $t_1, t_2$ in $y$.

\smallskip 
(ii) Given any $L$-formulas
$\phi$ and $\psi$ in $y$, we have new $L$-formulas in $y$: $$\neg \phi,\quad 
\phi\vee \psi,\ \text{ and } \phi\wedge \psi.$$

\smallskip
(iii) Let $\phi$ be a formula in $(y_1,\dots, y_i, z, y_{i+1},\dots, y_n)$, where $0\le i \le n$ and $z$   

{}\qquad is a variable different from $y_1,\dots, y_n$; then 
$$\exists z \phi\ \text{ and }\
\forall z \phi$$
\qquad\quad  are new $L$-formulas in $y$. 

\smallskip\noindent
Formally, these formulas in $y$ are words on a certain alphabet, together with the specification of the tuple $y$. We also write $\phi(y)$ to indicate that we are
dealing with a formula $\phi$ in $y$. Each $L$-formula $\phi=\phi(y)$ names (we also say: {\em defines}) a $0$-definable set $\phi^{\cM}\subseteq M^n$: the atomic formulas $\top$ and $\bot$ name the subsets $M^n$ and $\emptyset$ of $M^n$, and the atomic formulas $R(t_1,\dots, t_m)$ and $t_1=t_2$ as above name the sets 
$$\{a\in M^n:\ (t_1^{\cM}(a),\dots, t_m^{\cM}(a))\in R^{\cM}\} \text{ and }\{a\in M^n:\ t_1^{\cM}(a)=t_2^{\cM}(a)\},$$ and for $L$-formulas $\phi, \psi$ as above,
$$(\neg \phi)^{\cM}=M^n\setminus \phi^{\cM}, \quad 
(\phi\vee \psi)^{\cM}\ :=\ \phi^{\cM}\cup \psi^{\cM},\  (\phi\wedge \psi)^{\cM}\ =\ \phi^{\cM}\cap \psi^{\cM},$$
and for an $L$-formula $\phi$ in $(y_1,\dots, y_i, z, y_{i+1},\dots, y_n)$ as above: $(\exists z \phi)^{\cM}$ is the set
$$ \{(a_1,\dots, a_n)\in M^n:\ (a_1,\dots, a_i,b,a_{i+1},\dots,a_n)\in \phi^{\cM} \text{ for some }b\in M\},$$
that is, the image of $\phi^{\cM}\subseteq M^{n+1}$ under the projection map $$(a_1,\dots, a_i, b, a_{i+1},\dots, a_n) \mapsto (a_1,\dots, a_n): M^{n+1} \to M^n,$$
and $(\forall z \phi)^{\cM}:=(\neg\exists z \neg\phi)^{\cM}$,
which equals the set
$$ \{(a_1,\dots, a_n)\in M^n:\ (a_1,\dots, a_i,b,a_{i+1},\dots,a_n)\in \phi^{\cM} \text{ for every }b\in M\}.$$
The $0$-definable subsets of $M^n$ are exactly the sets $\phi^{\cM}$ with $\phi$ an $L$-formula in $y$. Likewise for $A\subseteq M$, the $A$-definable subsets of
$M^n$ are exactly the sets $\phi^{\cM_A}$ with $\phi$ an $L_A$-formula in $y$, but for convenience we write this also as $\phi^{\cM}$. 

The $L$-formulas $\phi$ in $y=(y_1,\dots, y_n)$ with $n=0$ have a special status and are called {\em $L$-sentences}, typically denoted by $\sigma$. One can think of a sentence as making an assertion. Formally, a sentence $\sigma$ names a set  $\sigma^{\cM}\subseteq M^0$, so it equals $M^0$, in which case we say that $\sigma$ is {\em true in $\cM$} (notation: $\cM\models \sigma$),  or it is empty, in which case we way that $\sigma$ is {\em false in $\cM$}. For example, if $L$ is the language of rings and $x,y$ are distinct variables, then the $L$-sentence 
$\forall x \exists y (x=y\cdot y)$ is true in exactly those fields in which every element is a square. 

\subsection*{Elementary extensions} An {\em elementary extension\/}
of the $L$-structure $\cM$ is an extension $\cN\supseteq \cM$ of $\cM$ such that
the same $L_M$-sentences are true in $\cM$ and $\cN$ (where of course for $a\in M$ the constant symbol $\underline{a}$ names $a$ in both $\cM$ and $\cN$).
Notation: $\cM\preceq \cN$.  
Here are two wellknown situations where this is the case: any algebraically closed field
is an elementary extension of any algebraically closed subfield, any real closed
field is an elementary extension of any real closed subfield.

Suppose $\cM\preceq \cN$. Then for any $L_M$-formula $\phi=\phi(x_1,\dots, x_n)$ we have
$\phi^{\cM}=\phi^{\cN}\cap M^n$. Moreover, if $\psi=\psi(x_1,\dots, x_n)$ is a second $L_M$-formula and $\phi^{\cM}=\psi^{\cM}$, then $\phi^{\cN}=\psi^{\cN}$, so a definable set $X=\phi^{\cM}\subseteq M^n$ (of $\cM$) yields a definable set
$X(\cN)=\phi^{\cN}\subseteq N^n$ that does not depend on the choice of defining
formula $\phi$.  

To profit from saturation and elementary extensions we use:

\begin{prop} Any $L$-structure has an $\aleph_1$-saturated elementary extension.
\end{prop}

\noindent
Indeed, for any $L$-structure $\cM$ and nonprincipal ultrafilter $U$ on the set $\N$, the ultrapower $\cM^{\N}/U$ is an $\aleph_1$-saturated elementary extension of $\cM$, where $\cM$ is identified with a substructure of $\cM^{\N}/U$ via the diagonal embedding; this is a remark intended for those who know about ultrapowers.   In Section ~\ref{par4} of this paper we only need 
$\aleph_0$-saturation instead of the stronger $\aleph_1$-saturation. 

\subsection*{Revisiting o-minimality} Let $K$ be an {\em expansion\/} of an ordered field. Among the definable subsets of $K$ in this expansion are at least the open intervals $(a,b)_K$ with $a < b$ in $K\cup\{-\infty, +\infty\}$, and thus the finite unions of such open intervals and singletons $\{a\}$ with 
$a\in K$. We say that $K$ is {\em o-minimal\/} if there are no other subsets of $K$ definable in this expansion. To see how this is related to the concept of o-minimality considered in the first part of this appendix, let $\text{Def}^n(K)$ be the collection of sets $X\subseteq K^n$ that are definable in this expansion $K$. Note that then $K$ is o-minimal if and only if the family $\big(\text{Def}^n(K)\big)$ is an o-minimal structure on the underlying ordered field of $K$. In particular, if $K$ is o-minimal, then the underlying ordered field of $K$ is real closed. 

\begin{lemma} If $K$ is o-minimal, then so is any elementary extension of $K$. 
\end{lemma}
\begin{proof} Assume $K$ is o-minimal, and $K\preceq K^*$, so the underlying ordered field of $K^*$ extends the underlying ordered field of $K$. Let $X\subseteq K^*$ be definable.
Then we have a definable set $Y\subseteq K^{n+1}$ and a point $b\in (K^*)^n$
such that $X=Y^{K^*}(b)$. Take $N\in \N$ and cells $C_1,\dots, C_N$ in $K^{n+1}$ such that $Y=C_1\cup\cdots \cup C_N$. Then for every
$a\in K^n$ the set $Y(a)$ is a union of at most $N$ sets
$\{c\}$ with $c\in K$, and intervals of $K$. With $K$ an $L$-structure, this fact can be expressed by a certain $L_K$-sentence being true in $K$, hence in $K^*$,
which then means in particular that $X=Y^{K^*}(b)$ is a union of at most $N$ sets $\{c\}$ with $c\in K^*$, and intervals of $K^*$.    
\end{proof} 

\noindent
Let $K$ be o-minimal and $K\preceq K^*$. To each definable set $X\subseteq K^n$ we associate the set $X^*:= X(K^*)\subseteq (K^*)^n$, which is definable (over $K$) in $K^*$. If $C\subseteq K^n$ is an $(i_1,\dots, i_n)$-cell in the sense of $K$, then $C^*\subseteq (K^*)^n$ is an $(i_1,\dots, i_n)$-cell in the sense of $K^*$. It follows that for definable $X\subseteq K^n$ we have $\dim X=\dim X^*$
where the dimension on the left is in the sense of $K$, and the dimension on the right is in the sense of $K^*$. 

\medskip\noindent
In the main text and in part A of this appendix we used the letter $R$ to refer to an o-minimal field, but in this part B we used $R$ to indicate a relation symbol.
That is why in this part B we prefer the letter $K$ when dealing with o-minimal expansions of ordered fields and o-minimal fields. The distinction between the two concepts ({\em o-minimal expansion of an ordered field\/} and {\em o-minimal field\/}) is often immaterial: we saw that an o-minimal expansion $K$ of an ordered field gives rise to an o-minimal field with the same underlying ordered field and the same definable sets. 

When considering {\em elementary extensions\/} and
{\em $\aleph_0$-saturation}, the distinction is significant: When referring to
an {\em elementary extension of an o-minimal field\/} we really mean an elementary extension of an o-minimal expansion $K$ of an ordered field, so $K$ is then an $L$-structure for a suitable language $L\supseteq L_{\text{OF}}$.  Likewise
when referring to an o-minimal field as being $\aleph_0$-saturated, we mean:
an $L$-structure $K$ that gives rise to this o-minimal field is $\aleph_0$-saturated.

\subsection*{How to use the above?} To illustrate typical uses, we first show how Theorem~\ref{pars} follows from it being true when the ambient o-minimal field is $\aleph_0$-saturated. So let $K$ be an o-minimal field, $X\subseteq K^m$ a strongly bounded definable set, and $k\ge 1$; we need to show that
$X$ has a $k$-parametrization. We can assume $X\ne \emptyset$, and set $l:= \dim X$. Take $N\in \mathbb{N}$ such that $X\subseteq [-N,N]_K^m$. As explained, $K$ is viewed as an $L$-structure for a language $L\supseteq L_{\text{OF}}$. Fix an  $L_K$-formula $\phi(x)$, $x=(x_1,\dots, x_m)$, such that $X=\phi^K$. 
 Take an $\aleph_0$-saturated elementary extension $K^*$ of $K$. Then $K^*$ is an o-minimal field and $$X^*\  =\phi^{K^*}\ \subseteq\ [-N,N]^m_{K^*}$$ is strongly bounded, and so has a $k$-parametrization $\{f_1,\dots, f_M\}$ (with respect to $K^*$), since Theorem~\ref{pars} was established in Section~\ref{par4} when the ambient o-minimal field is $\aleph_0$-saturated. For $\mu=1,\dots, M$ we have $f_{\mu}: (0,1)_{K^*}^l\to (K^*)^m$, so the graph of $f_{\mu}$ is defined in $K^*$ by an $L_{K^*}$-formula 
 $\phi_{\mu}(b,v,x)$
with $\phi_{\mu}=\phi_{\mu}(u,v,x)$ an $L_K$-formula, $u=(u_1,\dots, u_p),\ v=(v_1,\dots, v_l)$ and $b\in (K^*)^p$ (the same $p$ and $b$ for all $\mu$, without loss of generality). The fact that there exists $b\in (K^*)^p$
such that $\phi_1(b,v,x),\dots, \phi_M(b,v,x)$ define in $K^*$
 the graphs of functions of a $k$-parametrization of $X^*$ can be expressed by a certain
$L_K$-sentence $\exists u\theta(u)$
being true in $K^*$. (This sentence is complicated but its construction is routine and just mimicks the definitions of the various notions involved.) Hence this sentence
$\exists u\theta(u)$ is also true in $K$, which then means that for some
$a\in K^m$ the formulas  $\phi_1(a,v,x),\dots, \phi_M(a,v,x)$ define in $K$
 the graphs of functions of a $k$-parametrization of $X$.

\medskip\noindent
In a very similar way Theorem~\ref{parm} follows from the fact, established in Section~\ref{par4}, that it  is true when the ambient o-minimal field is $\aleph_0$-saturated.

\bigskip\noindent
Next we explain the use of ``$\aleph_0$-saturation plus Definable Selection'' in obtaining Corollary~\ref{unifstr}  as a consequence of Corollary~\ref{corstr}. (The other use of this device earlier in Section~\ref{par4} is along the same lines. The argument we give may seem lengthy, but such arguments are utterly routine in model theory and are therefore usually not spelled out but left to the reader.) 

We are now dealing with an o-minimal field $K$ which is $\aleph_0$-saturated when viewed as an $L$-structure for suitable
$L\supseteq L_{\text{OF}}$ as before. We are given $d,k,m,n$ with $k,n\ge 1$ and definable $E\subseteq K^m$ and definable $Z\subseteq E\times [-1,1]^n\subseteq K^{m+n}$ with $\dim Z(s)=d$ for all $s\in E$. Take a finite $A\subseteq K$ such that $E$ and $Z$ are $A$-definable. Corollary~\ref{corstr} yields for every
$s\in E$ a definable $C^k$-map $$f=(f_1,\dots, f_N): (0,1)^d\to (K^n)^N=K^{Nn}$$
with $N=N(s)\in \mathbb{N}$ depending on $s$, such that (i) and (ii) of that corollary hold for $\Phi:=\{f_1,\dots, f_N\}$ and $X(s)$ in the role of $X$. 
For each $L$-formula $\phi=\phi(u,x,y)$, 
$$ u=(u_1,\dots, u_M),\quad x=(x_1,\dots, x_d),\quad y=(y_{11},\dots,y_{Nn})$$
with $M,N\in \mathbb{N}$ depending on $\phi$ we consider the $A$-definable
set $E_{\phi}\subseteq K^m$ of all $s\in E$ such that for some $b\in K^M$ the $L_K$-formula
$\phi(b,x,y)$ defines the graph of a map $f$ parametrizing $X(s)$ as above. Since $K$ is $\aleph_0$-saturated and $E$ is covered by the sets $E_{\phi}$, it is covered by finitely many of them,
say by $E_{\phi_1},\dots, E_{\phi_e}$, 
$$\phi_i\ =\ \phi_i(u_1,\dots, u_{M(i)},x,y_{11},\dots, y_{N(i)n})\qquad (i=1,\dots,e).$$ 
For $i=1,\dots,e$, let $E(i)$ be the definable set of all $s\in E$ with $s\in E_{\phi_i}$ and $s\notin E_{\phi_j}$ for $1\le j<i$.
Definable selection then yields for such $i$ a definable map 
$$s\mapsto b_i(s): E(i)\to K^{M(i)}$$ such that for every $s\in E(i)$ the $L_K$-formula $\phi_i(b_i(s),x,y)$ defines
in $K$ the graph of a $C^k$-map $f: (0,1)^d\to (K^{n})^{N(i)}$ parametrizing
$X(s)$ as specified earlier. Now $E$ is the disjoint union of $E(1),\dots, E(e)$.
Adding suitable constant maps it is routine to obtain from this
for $N=\max_i N(i)$ a definable
set $F\subseteq E\times K^d\times K^{Nn}$  for which the conclusion of 
Corollary~\ref{unifstr} holds.

\end{appendices}

\end{document}